\documentclass{article}
\usepackage{amsmath,latexsym,amssymb,amsthm,enumerate}
\theoremstyle{plain}
\newtheorem{theorem}{Theorem}[section]
\newtheorem{proposition}[theorem]{Proposition}
\newtheorem{corollary}[theorem]{Corollary}
\newtheorem{lemma}[theorem]{Lemma}
\newtheorem*{theoremWONum}{Theorem}

\theoremstyle{definition}
\newtheorem{definition}[theorem]{Definition}
\newtheorem{example}[theorem]{Example}
\newtheorem{remark}[theorem]{Remark}

\newtheorem*{claimWONum}{Claim}
\newtheorem*{remarkWONum}{Remark}

\numberwithin{equation}{section}
\newcommand{\Gor}{\rm{Gor}}

\newcommand{\U}{\mathbf{U}}

\def\endproof{\hfill$\square$\medskip}

\def\CB{\mathrm{CB}}
\def\cha{\mathrm{char}\ }

\def\Hom{\mathrm{Hom}}
\def\Homog{\mathrm{Homog}}

\def\Ann{\mathrm{Ann}\ }
\def\Proj{\mathrm{Proj}\ }

\def\PGOR{\mathrm{PGOR}}
\def\ZGOR{\mathrm{ZGOR}}
\def\Hilb{\mathrm{Hilb}}
\def\Soc{\mathrm{Soc}}

\providecommand{\bysame}{\makebox[3em]{\hrulefill}\thinspace}

\def\<{\left<}
\def\>{\right>}

\def\Sym{\mathrm{Sym}}
\def\Spec{\mathrm{Spec}}

\def\U{\mathrm{U}}

\def\ns{\footnotesize \it}

\def\type{\mathrm{type}}
\def\Z{\mathfrak{Z}}

\def\Ext{\mathrm{Ext}}
\title{Inverse Systems of\\ Zero-dimensional Schemes in
$\mathbb P^n$}
\author{Young Hyun Cho\\[.05in]
{\ns Department of Mathematics, Seoul National University, Seoul
151-742, Korea.\
}\\[.2in]
Anthony Iarrobino\\[.05in]
{\ns Department of Mathematics, Northeastern University, Boston, MA 02115, USA.
}\\[.2in]}

\date{April 10, 2012}
\begin{document}

\maketitle
\begin{abstract}
The authors construct the global
Macaulay inverse system $L_\Z$ for a zero-dimensional subscheme $\Z$
of projective $n$-space $\mathbb P^n$ over an algebraically closed
field ${\sf k}$, from the local inverse systems of the irreducible
components of $\Z$. They show that when $\Z$ is locally Gorenstein a generic element $F$ of degree $d$ apolar to $\Z$ determines $\Z$  if $d$ is larger than an invariant $\beta(\Z )$.  As a consequence of this globalization, they show that a natural upper bound for the Hilbert function of Gorenstein Artin quotients of the coordinate ring of $\Z$ is achieved for large socle degree. They also show the uniqueness of generalized additive decompositions of a homogeneous form into powers of linear forms, under suitable hypotheses.
\par The main tools are elementary, but delicate. They involve a careful study of how to
homogenize a local inverse system and of the behavior of the
homogenization under a change of coordinates.
\small
\footnote{{{\bf 2010 Mathematics Subject Classification}: Primary: 14N05; Secondary: 13N10,13D40,13H10, 14C05.}\par
{{\bf keywords}: Macaulay inverse system, regularity degree, globalization, zero-dimensional scheme, Gorenstein\par Artin ring, irreducible components, generalized additive decomposition.}}
\end{abstract}
\normalsize

\section{Introduction.}\par We study Macaulay's inverse systems
for the defining ideals of a zero-dimensional (punctual) scheme $\Z$ of the
projective space $\mathbb P^n$ over an algebraically closed field
${\sf k}$. Of course, we may suppose that such schemes are contained in an
affine subspace $\mathbb A^n$ of $\mathbb P^n$. For any graded ideal
$I$ in the coordinate ring $R={\sf k}[x_1,\ldots , x_{n+1}]$ of $\mathbb
P^n$, Macaulay's inverse system $I^{-1}$ is an $R$-submodule of the
dual ring, the divided power series ring $\Gamma={{\sf k}_{DP}}[X_1,\ldots ,
X_{n+1}]$ and $I^{-1}$ contains the same information as is in the
original ideal. Thus, it is not hard to determine which inverse
systems arise from zero-dimensional schemes (Proposition
\ref{satinvsys}), or which arise from a zero-dimensional scheme
concentrated at a single point (Lemma \ref{homogdefp}).  \par

 Our main work here begins with an Artinian quotient $A=R'/J$ of the
coordinate ring $R'={\sf k}[y_1,\ldots ,y_n]$ of affine $n$-space $\mathbb
A^n\subset \mathbb P^n: x_{n+1}=1$ that defines a zero-dimsional
subscheme $\Z\subset \mathbb A^n$, concentrated at a finite set of
points. Its $\it{local}$, or affine inverse system $L'(J)$ is an
$R'$-submodule of the completion $\widehat{\Gamma'}$ of the divided
power ring $\Gamma '={{\sf k}}_{DP}[Y_1,\ldots ,Y_n]$ dual to $R'$ --- this
completion is the $R'$-injective envelope of ${\sf k}$. We then determine
from $L'(J)$ the $\it{global}$ inverse system
$L_\Z=(I_\Z)^{-1}\subset \Gamma$ over $\mathbb P^n$ of the defining
ideal $I_\Z\subset R$ for $\Z$. Our goal is to write
$\it{generators}$ of the global inverse system
$L_\Z$, in terms of generators of the local inverse systems of the
irreducible components of $\Z$.
\par Let the zero-dimensional scheme $\Z\subset \mathbb A^n$ have degree $s$. Then the ring $A=R'/J$ has dimension $s$ as ${\sf k}$-vector space. The local, or
affine inverse system $L'(J)$ also has dimension $\dim_{{\sf k}} L'(J)=s$.
Since $J=\cap_k J(k)$, the intersection of its primary components,
the inverse system $L'(J)$ is a direct sum of the local inverse
systems $L'(J(k))=L'(J\mathcal O_{p(k)})\subset \widehat{\Gamma'}$
at the points $p(k)$ of support of $\Z$. The scheme $\Z$ has a
unique saturated global defining ideal $I_\Z\subset R$, and the coordinate
ring $\mathcal O_\Z=R/I_\Z$ has Krull dimension one. The global
Hilbert function $H_\Z=H(\mathcal O_\Z)$, satisfies 
\begin{equation*}
H_\Z=(1,\ldots
, s,s,s,\ldots ),
\end{equation*} the first difference $\Delta H_\Z$ -- often called the $h$-vector of $\Z$ -- is an
$O$-sequence of total length $s$ (Theorem \ref{zero-dim}). The
global inverse system $L_\Z$ is a non-finitely generated, graded
$R$-submodule of $\Gamma$, whose Hilbert function $H(L_\Z)$
satisfies $H(L_\Z)=H_\Z$. Suppose now that $\Z$ is concentrated at a
single point $p$. Since $(H_\Z)_i=s$ for $s\ge \tau(\Z)$, an
invariant of $\Z$, it is natural to expect that $(L_\Z)_i$ should
be a homogenization of $(L'(J))_{\le i}$ where
$L'(J)=(L_\Z)_{x_{n+1}=1}$, at least for $i\ge \alpha(\Z)$, the
$\it{socle} \, \it{degree}$ of $A$. This is what occurs (Proposition \ref{AuxMainL},
Theorem \ref{Homogp}).
\par However, in general the global Hilbert function
$H_\Z$ is not determined by the local Hilbert functions $\mathcal
O_{p(k)} / J\mathcal O_{p(k)}$ at the points of its support --- even when the
support of $\Z$ is a single point! An exception is when $J\mathcal
O_p$ is $\it{conic}$, a graded ideal in the local ring
$\mathcal O_p$ (Example \ref{homogHilbf}).\ In this $\it{conic}$
local case, we have $\Delta H_\Z=H(\mathcal O_p/J\mathcal O_p)$, and
the ideal $I_\Z$ and its global inverse system $L_\Z$ is easily read
from $J,L'$ (\cite[Lemma 6.1]{IK}, Proposition \ref{aGor}).\par
How do we determine the global inverse system $L_\Z$
from the local inverse systems $L'(J(k))$? We answer by suitably homogenizing the
local inverse systems (Definition~\ref{defhomcoordp}). Our main result is that we determine $\it{generators}$ of $L_\Z$ from generators of the local
inverse systems:  we give an algorithm to determine the
Macaulay dual $L_\Z$ of the one-dimensional coordinate ring
$\mathcal O_\Z$ directly from the local inverse systems (Lemma \ref{homoginvsys}, Theorems \ref{Homogp},
\ref{decompIS}).
 In Section \ref{2.1} we consider the case $\Z$ has support the coordinate point $p=p_0
=(0:\ldots :0:1)$ $\in \mathbb P^n$. Since $\mathcal I_p=J\mathcal
O_p$ defines an  Artinian quotient, we have $J\supset {M'}^{j+1},
M'=(y_1,\ldots ,y_n)$ for some integer $j>0$, and we may replace
$\mathcal O_p$ by $R'$. Thus $J$ has an inverse system $L'(J)\subset
\Gamma'$ and there is no need to complete to $\widehat{\Gamma '}$. However $L'(J)$ will usually not be graded. We define the
homogenization of $L'(J)$ when $p=p_0$ in Definition \ref{defhomcoordp}, then
show it is the same as $L_\Z$ and find suitable $\it{generators}$ of
$L_\Z$ in Lemmas \ref{homoginv}, and \ref{homoginvsys}.  \par
 We give examples that show a
surprising behavior of this globalization with respect to the
regularity degree $\sigma(\Z)$. The degree $i$ component
$(L_\Z)_i$ may not be determined by the local $L_i'$, and may not be the homogenization of $(L_\Z)_{\sigma(\Z)}$ to degree $i$ (Examples \ref{regdeg} to
\ref{noncloserp}), although the  socle degree component
$L_{\alpha(\Z)}$ does determine $L_\Z$ (Proposition \ref{AuxMainL}(\ref{is3})). \par
 In Section \ref{arbitrarypoint} we
 determine the global inverse system $L_\Z$ for a scheme
concentrated at an arbitrary point $p\in \mathbb A^n\subset \mathbb
P^n$. We then prove a projective space $\it{Comparison} \,\it{
Theorem}$ (Theorem~\ref{Homogp}) relating the inverse system at $p$
to one concentrated at the origin. This result is different from our
version of Macaulay's Comparison Lemma, which describes local
inverse systems at points $p\in \mathbb A^n$, as the product of a
local inverse system at the origin and an $\it{exponential}$ power
series $f_p$ (Lemma~\ref{transforma}). Rather, our
$\it{Comparison}\, \it{Theorem}$ shows that $L_\Z$ and its
generators can be obtained from $L_{\Z'}$, the corresponding
inverse system $L_{p_0}$ at the origin $p_0$ of $\mathbb A^n$, by suitably substituting the divided
powers of the linear form $L_p=\sum_k a_kX_k$ determined by the
coordinates $(a_1:\ldots :a_n:1)$ of $p$, for the powers of $Z=X_{n+1}$
in $L_{p_0}$.
\par
 We complete our study of globalization in Section \ref{schfinsup}.  The
Decomposition Theorem \ref{decompIS} handles the transition to
arbitrary zero-dimensional schemes. We discuss regularity
degree, giving an upper bound in terms of the invariant $\alpha(\Z)$
when the number of irreducible components of $\Z$ is less or equal $n+2$ (Proposition
\ref{regbound}).

In Section \ref{application} we give the application that motivated
our globalization of inverse systems. When $\Z$ is a locally
Gorenstein punctual subscheme of $\mathbb P^n$, then $\Sym(H_\Z,j)$ is an upper bound for the Hilbert function $H(A)$
of a Gorenstein Artinian (GA) quotient $A$ of $\mathcal O_\Z$,
having socle degree $j$. As a consequence of our construction of
$L_\Z$ we show that this
upper bound is always achieved by some GA quotient of $\mathcal
O_\Z$, hence for almost all GA quotients of socle degree $j$,
provided that $j$ is sufficiently large (Theorem \ref{main}).\vskip 0.3cm

\par In an earlier related paper  we showed that
there is no level Artinian algebra $A$ of Hilbert function
$H(A)=(1,3,4,5,\ldots,6,2)$ \cite{ChoI}. This shows
that Theorem \ref{main} cannot be simply extended to a scheme $\Z$ that is
locally of type two
--- having two-dimensional socle in a single degree. In a sequel paper \cite{ChoI2} we will determine
the global Hilbert function $H_\Z$ for compressed Gorenstein
subscheme $\Z\subset \mathbb P^n$. Then, using Theorem \ref{main},
we will  exhibit families $\mathbb P\mathbf{GOR}(T)$ of graded
Gorenstein Artin algebras of embedding dimension $r$ and certain
Hilbert functions $T=H(s,j,r),r\ge 5, s$ large enough, that contain
several irreducible components (Remark \ref{rem324}.)

\subsection{Inverse systems and Gorenstein subschemes of $\mathbb P^n$.}\par\subsubsection{Uses of inverse systems}
Macaulay used his inverse systems, a version of the classical notion
of apolarity, to develop a theory of primary decomposition of ideals
\cite{Mac}. Consider a subscheme $\Z=\Spec (\mathcal O_p/\mathcal
I_p)$ of affine $n$-space concentrated at the point $p$ of $\mathbb
A^n$, whose maximal ideal is $m_p\subset R'=k[y_1,\ldots ,y_n]$, and
local ring $\mathcal O_p$. We may write also $\Z=\Spec (A), A
=R'/I'$, where $A$ has finite length, and where $I'$ satisfies
$m_p\supset I'\supset m_p^{\alpha+1}$ for some $\alpha$. The affine
inverse system $L(I')\subset \widehat{\Gamma '}$ is a finite
$R'$-module, isomorphic to the dualizing module $\Omega (A)$.
The number of generators of the submodule
$L(I')\subset \widehat{\Gamma '}$ is the $\it{type}$ of $A$, the
vector space dimension of the \emph{socle} $\Soc(A)=(0:m)$
(Definition \ref{type}). In particular when $A$ is a Gorenstein
Artin algebra
--- one whose socle is a vector space of dimension $1$  --- the
local inverse system has a single generator, and was termed by
Macaulay a $\it{principal}\,\it{ system}$ \cite[\S 60]{Mac}.
\par A. Terracini translated questions about the Hilbert function of ideals of
functions vanishing to specified order at a set of general enough
points in $\mathbb P^n$ -- the $\it{interpolation}\,\it{ problem}$ --
to questions concerning the Hilbert functions of ideals generated by
powers of the corresponding linear forms. This translation has led
to new insights, some still conjectural, when $n\ge 3$
\cite{T1,T2,EhR,I5}, and also contributed to the solution of a
$\it{Waring}\, \it{ problem}$ for forms - whether a generic homogeneous form of degree $j$ can be written as a sum of powers of linear forms. The answer followed from  J. Alexander and A.
Hirschowitz's solution of the order-two interpolation problem (\cite{T2,AH,Ch2,I4},\,\cite[\S 2.1]{IK},\cite{BrOt}). Principal inverse systems
generated by certain forms associated to partitions, occurred as spaces
of $\it{harmonics}$ in the recent $n$-$\it{factorial}\,\it{
conjecture}$ in combinatorics and geometry \cite{Ha}; they are also
related to constant-coefficient partial differential equations
\cite{Rez}. Inverse systems have been studied further, sometimes as
Matlis duality/injective envelope (see \cite{N,NR},\,\cite[Chapter
10]{BS}). Related to the simpler Matlis duality are the deeper
topics of dualizing modules, residues, and local cohomology 
\cite{Lj,BS,Sch}. \par
 F.H.S.  Macaulay introduced inverse systems in the context of affine space. To
our knowledge, before 2000 when we submitted an earlier version, there had been no systematic study of
inverse systems in the context of projective spaces, beyond the case
of fat points considered by A.~Terracini and others
(\cite{T1,T2,EhR,EmI,Ge,I4}, see \cite{Tes} for an exception). Since 2000 several other authors have explored this topic, notably M. Elkadi and B.~Mourrain, and  J. Brachat in his thesis who discuss generalized additive decompositions \cite{ElMo,Bra}. There has been recent interest in non-homogeneous principal inverse systemes in connection with the study of scheme or ``cactus'' length of forms, by K. Ranestad and F.-O. Schreyer \cite{RS1,RS2} and others \cite{BeRa,BuB}. 
\par  D. M. Meyer and L. Smith develop the theory of Poincar\'{e} duality algebras - the Gorensstein case of inverse systems -- in the language of Hopf algebras. They construct new Gorenstein algebras from existing ones by using their corresponding principal inverse systems \cite{MeSm}.
L. Smith and R.~E.~Stong consider a pair of Poincar\'{e} duality algebras 
\begin{equation*}
A={\sf k}[x_1,x_2,\ldots , x_n]/J\mapsto {\sf k}[x_1,x_2,\ldots , x_n,x_{n+1}]/I=B,
\end{equation*} where $B$ is a free $A$-module and $f_J,f_I$ are generators of principal inverse systems of $J$ and $I$, respectively. They obtain $f_I$ from $f_J$ by means of a homogenization process  \cite[Theorem 2.5]{SmSt} and they make new Poincar\'{e} duality algebras from a given $A$ \cite[Corollary 3.5]{SmSt}.\vskip 0.2cm

\par A zero-dimensional scheme
$\Z \subset \mathbb P^n$ is the union of a finite number of schemes
$\Z (i)$, each supported at a single point $p(i)\in \mathbb P^n$. So we
may define dualizing modules $D(V)$ for $B$-modules $V$, where
$B=R/I_{\Z}$, as direct sums of the dualizing modules at the finite
number of points: thus $D(V)=\Hom (V,\bigoplus E(i))$, where
$E(i)=E(R/M(i))$ is the injective hull of the residue field $\sfÊk$ of
$B$ at the maximal ideal $M(i)$ at the $i$-th point $p(i)$ of the
support. This viewpoint is adopted by Curtis-Reiner \cite[p.37]{CR},
and was used by R. Michler in \cite{Mi}. However, our task is in one
sense easier, and in another harder. Easier, since our ideal
$I_\Z (i)$ includes a power of the maximal ideal $m_{p(i)}$ at
$p(i)$ so we may avoid the full injective hull and deal locally with
$\it{dual}\,\it{ polynomials}$, or dual polynomials times an
exponential (\cite{Lj,Mac}, Lemma \ref{transforma}, Remark
\ref{compare}). Harder since we wish here to consider a
global inverse system for $R/I_\Z$ that is embedded in $\Gamma$,
rather than simply being an $R$-module. We pass from the local
inverse systems for the ideals $I_{\Z (i)}$ at each point, finite
submodules of $\widehat{\Gamma '}$, to the global inverse system for
the ideal $I_{\Z}\subset R$, which is not finitely generated, but
is determined by a finite number of its elements.  \par
Gorenstein Artin (GA) algebras are minimal reductions of Gorenstein algebras.
Gorenstein algebras are a natural generalization of complete
intersections. Artin algebras, and in particular GA algebras occur
in the study of mapping germs of differentiable maps. Recently  a
category of commutative Frobenius algebras, that correspond to non-graded Gorenstein Artin
algebras have been identified with the category of two-dimensional
topological quantum field theories \cite{Ab}. \par J. Watanabe
showed that the family $\ZGOR(T)$ of all, not necessarily graded
standard GA algebras having a {\it symmetric} Hilbert function $T$ 
is fibred
over the family $\mathbb P\mathbf{GOR}(T)$ parametrizing graded GA
algebras of Hilbert function $T$
by the map $A\rightarrow Gr_m(A)$ to the associated graded algebra (\cite{Wa},\,\cite[Prop. 1.7]{I3}. 
Our work here relates to the component structure of $\mathbb
P\mathbf{GOR}(T)$ and we
hope there could be application to these other fields. 
\par The inverse system viewpoint can be used to parametrize Gorenstein Artin algebra
quotients of $R'$ having a given Hilbert function 
\cite{I3}). Several authors have studied from this or related
viewpoints $\it{compressed}\,\it{ algebras}$ --- those having a
maximum possible Hilbert function, given the socle degree and
embedding dimension (see \cite{I1,FL,Bj1,Za}).  See also the ``points \'{e}pais' discussion by J. Emsalem \cite{Em} and the foundational study of D.~Laksov \cite{La}.
\subsubsection{Main Results and Applications.}
 We first translate into the language of global inverse systems, some basic
algebraic properties of the coordinate ring $R/I_\Z$, where $I_\Z$
is the defining ideal of a zero-dimensional subscheme of $\mathbb
P^n$.  We consider such properties as ``there is a linear non-zero
divisor $\ell$ on $R/I_\Z$'', and the $\it{type}$ of $R/I_\Z$.
We then use the inverse sytems to study such questions as ``When is
$\Z$ arithmetically Gorenstein (aG)?'', and ``When can $I_\Z$ be recovered from a general form $F$
annihilated by
$I_\Z$?'' ($F$ must have sufficiently high degree). We discuss``When is $\Z$ aG?''
in Example~\ref{CayleyB}, Proposition \ref{aGor}, Corollary~\ref{aGor2}, Remark
\ref{critagor}, and in Examples \ref{critagorex1}, \ref{critagorex2}, \ref{exlink}. As to the latter
question, it is not hard to see that if $\Z$ is Gorenstein and is also either smooth, or
 concentrated at a single point and $\it{conic}$ --- defined by a homogeneous ideal $\mathcal I_p$ of
the local ring $\mathcal O_p$ ---  then we can recover $I_\Z$ from
$F$ (see \cite{Bj1},\,\cite[Lemma 6.1]{IK}). However, it is also
easy to see that the second order neighborhood of a point $p\in
\mathbb P^n, n\ge 2$, a non-Gorenstein scheme defined by ${m_p}^2$,
cannot be recovered in this manner (Example \ref{fatpt}). When can we recover $\Z$ from an Artinian
Gorenstein quotient?
\par We answer this question in Theorem \ref{main}. Given a positive integer $j$ and a sequence $H_\Z$, we let $\Sym(H_\Z,j)$ be the sequence
\begin{equation}\label{symhz}
\Sym(H_\Z,j)_i= \begin{cases}
(H_\Z)_i\ , &\text{if $i\le j/2$;}\\
(H_\Z)_{j-i}\ , &\text{if $i \ge j/2$.}
\end{cases}
\end{equation}

 We denote by
$\sigma (\Z)$ the $\it{Castelnuovo}-\it{Mumford}\,\it{ regularity}$
of $\Z$, we set $\tau(\Z)=\sigma(\Z)-1$, and let $\alpha (\Z)$ be
the $\it{maximum}\,\it{ socle}\,\it {degree}$ of the local
coordinate ring of any irreducible component $\Z(i)$ (see Definition
\ref{socled}). We let $\beta (\Z )=\tau(\Z)+\max\{ \tau (\Z ),\alpha
(\Z )\}$, $I_\Z$ be the defining ideal, $L_\Z$ its inverse system,
and now state our main result, Theorem \ref{main}.
\par

\begin{theoremWONum} {\sc Recovering the scheme $\Z$ from a Gorenstein Artin
quotient.} Let $\Z$ be a locally Gorenstein zero-dimensional
subscheme of $\mathbb P^n$ over an algebraically closed field $\sfÊk$,
$\cha \sfÊk=0$ or $\cha \sfÊk>j$, and let $ L_\Z=(I_\Z)^{-1}$. Then we have
\begin{enumerate}[\rm(i)]
\item If $j\ge \beta(\Z)$, and $F$ is a general enough element of $(L_\Z)_j$, then
$H(R/\Ann(F))=\Sym(H_\Z,j)$. \par
\item If $j\ge \beta(\Z)$, and $F$ is a general enough element of
$(L_\Z)_j$, then for $i$ satisfying $\tau(\Z)\le i\le j-\alpha (\Z)$
we have $\Ann(F)_i= (I_{\Z})_i$. Equivalently, we have $R_{j-i}
\circ F = (L_\Z)_{i}$. \par
\item If $j\ge \max\{ \beta (\Z), 2\tau (\Z)+1\}$, and $F\in (L_\Z)_j$ is general enough, then
$\Ann(F)$ determines $\Z$ uniquely. If $I_\Z$ is generated in degree
$\tau (\Z)$, then $j\ge \max\{\beta (\Z),2\tau(\Z)\}$ suffices.
\end{enumerate}
\end{theoremWONum}
 \noindent
Thus, we may recover $\Z$ from a general dual form $F$ when $\Z$ is
locally Gorenstein and $j$ is large enough. The authors show
elsewhere that Theorem \ref{main} does not extend simply to
subschemes $\Z$ that are not Gorenstein, by showing that the
sequence $H=(1,3,4,5,...,6,2)$ cannot occur as the Hilbert function
of a level algebra --- one having socle in a single degree
(\cite{ChoI}). \par The question of which symmetric sequences $T$ of
integers are Gorenstein sequences ---  Hilbert functions
of a graded Gorenstein Artin algebra --- is open in embedding dimension
$r\ge 4$. We do not here find any new Gorenstein sequences
of the form $T=\Sym(H_\Z,j)$: each sequence $H_\Z$ already
occurs for a smooth scheme $\Z$ by \cite{Mar} (see Theorem
\ref{zero-dim}), and  Theorem \ref{main} was already known for
smooth schemes \cite{Bj1},\,\cite[Theorem 5.3E, Lemma 6.1]{IK}.
So an application of Theorem~\ref{main} is to relate
the postulation punctual Hilbert scheme $\mathbf{Hilb}^s_{\Gor
,H}(\mathbb P^n)$ parametrizing degree-$s$ Gorenstein subschemes
$\Z\subset \mathbb P^n$, satisfying $H_\Z=H$, with the scheme
$\mathbb P\mathbf{GOR}(T)$ parametrizing graded Gorenstein Artin
algebras of Hilbert function $T=\Sym(H,j)$. In this direction see also \cite{Kl1,Kl2}. \par
 In Section \ref{gad} we explore a second viewpoint on our construction of a global inverse system
from the local inverse system: we discuss $\it{generalized}\,\it{
additive}\,\it{ decomposition}$ of a form $F$ when $r>2$. A binary
form $F$ of degree $j$, always has a length-$s$ generalized additive
decomposition (GAD), with $s\le (j+2)/2$: this is either a sum of
$j$-th powers of $s$ distinct linear forms, or a sum
\begin{equation}\label{gad1}
F=\sum_i B_iL_i^{[j+1-s_i]}, \deg B_i=s_i-1, \deg L_i=1, s=\sum s_i.
\end{equation}

The existence of such an additive decomposition when $r=2$ is
equivalent to there being a form $h\in \Ann(F)$ that can be written
$h=\prod_i \ell_i^{s_i}$, where $\ell_i \circ L_i=0$. Thus, the
additive decomposition of equation \eqref{gad1} corresponds to a
zero-dimensional scheme $\Z: h=0\subset \mathbb P^1$,  whose
irreducible components $\Z_i: \ell_i^{s_i}=0$ have specified
multiplicities $s_i$. If $2s\le j+1$ then it is classical that the
GAD as in equation \eqref{gad1} is unique: see \cite[\S 1.3, Proposition 1.36, Theorem 1.43]{IK}. For any
embedding dimension, we say that a zero-dimensional scheme
$\Z\subset \mathbb P^n$ is an \emph{annihilating scheme} of the form
$F\in R$, if $I_\Z\circ F=0$. Since for a zero-dimensional scheme,
$I_\Z=\cap I_{\Z_i}$, where $\Z_i$ are the irreducible components of
$\Z$, and we have ${I_\Z}^{\perp}=\sum_i {I_{\Z_i}}^{\perp}$, it follows
that any form $F$ annihilated by $\Z$ can be written as a
$\it{generalized}\,\it{ sum}$, a sum of forms annihilated by the
components $\Z_i$. In determining very concretely the inverse
systems of $I_{\Z_i}$, we are partially answering the question, what
is a generalized additive decomposition? In particular, when $r=3$,
many forms $F$ have a $\it{tight}$ annihilating scheme $\Z_F\subset
\mathbb P^2$ that is unique. As well, there is often a unique
$\it{generalized}\,\it{additive}\,\it{ decomposition}$, up to
trivial multiplications. This occurs when the Hilbert function
$H(R/\Ann (F))$\ contains as a consecutive subsequence $(s,s,s)$. Then, there is
a unique degree-$s$ annihilating scheme according to \cite[Theorem
5.31]{IK}, and, as we shall see, a corresponding unique
$\it{generalized}\,\it{ additive}\,\it{ decomposition}$ for $F$
(Theorems \ref{GADd3} and \ref{GAD3}). Some results about uniqueness of GAD in a quite different language have occurred since our original submission in the thesis of J. Brachat \cite[\S 4]{Bra}.

\subsection{Notation and basic facts.}\label{notation}
 We will assume throughout, unless
specifically stated otherwise, that the base field $\sf k$ satisfies
$\cha \  {{\sf k}} = 0$, or $\cha \  {{\sf k}} = p >j$, where $j$ is the maximum
degree of any form considered (see Example \ref{char} for the
necessity of this assumption). We will also assume either that ${{\sf k}}$
is algebraically closed, or that all zero-dimensional schemes
considered have as support ${{\sf k}}$-rational points. Let $C$ denote the
${{\sf k}}$-vector space $\langle x_1,\ldots ,x_{n+1}\rangle$, and
$C^*=\langle X_1,\ldots ,X_{n+1}\rangle$ denote its dual; recall
that the divided power ring $\Gamma =\Gamma (C^*) ={{\sf k}}_{DP}[X_1,\ldots
,X_{n+1}]$ satisfies,
\[ \Gamma =
\bigoplus \Gamma _j=\bigoplus \Hom (R_j,{{\sf k}}), \hbox { with } \Gamma
_j=\langle \{ X^{[U]}| |U|=j\}\rangle,\]
 the span of the dual
generators to  $x^U\in R$, where here $U$ denotes the multiindex
$U=(u_0,\ldots u_n)$, of length $|U|=\sum u_i$. For convenience we
set $X^{[U]}=0$ if any component of $U$ is negative. The
multiplication in $\Gamma $ is defined by
\begin{equation}
X^{[U]}\cdot X^{[V]}= {\binom{U+V} { U}}X^{[U+V]}.
\end{equation}

 We denote by $R',\Gamma'$,
respectively, the corresponding rings $R'={{\sf k}}[y_1,\ldots ,y_n]$ and
$\Gamma '=$\linebreak ${{\sf k}}_{DP}[Y_1,\ldots ,Y_n]$, respectively. We
have $\Gamma '\cong E'=\Hom _{R'}(R',R'/M'), M'=(y_1,\ldots ,y_n)$.
We denote by $\widehat{\Gamma '}$ the completion of $\Gamma '$ with
respect to $M'$; thus, $\widehat{\Gamma '}$ is a divided power
series ring. The rings $R',\Gamma'$ correspond to the point
$p_0=(0:\ldots :0:1)$ in $\mathbb P^n$, whose maximal ideal is
$m_{p_0}=(x_1,\ldots ,x_n)\subset R$.  We recall
 below the contraction action of
$R={{\sf k}}[x_1,\ldots ,x_{n+1}]$ on the divided power ring.  Note that,
given our assumption excluding low characteristics, each theorem
about inverse systems stated in the context of the contraction
action of $R$ on $\Gamma $, has an analogue for the partial
differential operator (PDO) action of $R$ on $\mathcal
R={{\sf k}}[X_1,\ldots ,X_{n+1}]$, a second copy of the polynomial ring.
When $\cha {{\sf k}}=0$, there is a natural Gl-invariant homomorphism $\phi
:\mathcal R \rightarrow \Gamma $, $\phi (X^U) = U!X^{[U]}\in \Gamma
$ \cite[Appendix A]{IK}. To keep the
exposition simple, we will in general restrict ourselves to the
contraction action.    Note that we use here a different notation
than Macaulay's ${{\sf k}}[x_1^{-1},....,x_{n+1}^{-1}]$ for the injective
envelope $E=E({{\sf k}})$ (\cite{Mac,MeSm},\cite[Theorem 21.6]{E}). The claims
implicit in (iv),(v) of the following Definition are shown in Lemmas
\ref{pairVS} and \ref{sat} below.

\begin{definition}\label{cont}{\it Inverse Systems}\par
\begin{enumerate}[(i)]
\item  ({\it contraction action})  If $h=\sum a_Kx^K \in R, F=\sum b_UX^{[U]} \in \Gamma $,
then
\[ h\circ F = \sum_{K,U}a_Kb_UX^{[U-K]}.
\] \par
\item  ({\it partial differentiation action} --- PDO) If $h\in R, F\in \mathcal R =
{{\sf k}}[X_1,\ldots ,X_{n+1}]$, then
 \[ h\circ F = h(\partial
\ /\partial X_1,\ldots,
\partial \ /\partial X_{n+1})(F) \in \mathcal R  .\]\par
\item A {\it homogeneous inverse system} $W\subset \Gamma $ is a graded $R$-submodule of $\Gamma $ under
the contraction action. Thus $W=W_0\bigoplus \cdots \bigoplus
W_j\oplus \cdots \subset \Gamma $ is an inverse system f and only if
$\forall i\le j, R_i\circ W_j\subset W_{j-i}$.\par
\item ({\it inverse system of a graded ideal})
If $I$ is a graded ideal of $R$,
we will denote by $I^{-1}$ or by $I^{\perp}$ the homogeneous inverse system
of $I$, namely the
$R$-submodule of

$\Gamma $ given by $I^{\perp}=\oplus {{I_j}^\perp}$, where
\[
  {{I_j}^\perp}=\{ F\in \Gamma _j| \, h\circ F=0\,\,\,\,\,\forall h\in I_j\} .
\] \par
\item ({\it ideal of an inverse system}) If $W\subset \Gamma $ is an inverse system,
then we denote by $I_W$ the ideal $I_W=\Ann(W)$ where $(I_W)_j=\{h\in R_j\mid
h\circ w=0, \, \  \forall w\in W\}$.
\item ({\it local inverse system})  An inverse system in ${\widehat{\Gamma '}}$ is an $R'$-submodule of ${\widehat{\Gamma '}}$ under the contraction
action. If $J$ is any ideal of $R'={{\sf k}}[y_1,\ldots ,y_n]$, then we denote by
$J^{\perp}=J^{-1}\in {\widehat{\Gamma '}}$, the inverse system of all elements of ${\widehat{\Gamma '}}$ annihilated by
$J$, in the contraction action. The ideal $I_W\subset R'$ of an inverse system $W\subset {\widehat{\Gamma '}}$ is
the annihilator of $W$ under the contraction action.  [Warning: in general neither $J$ nor
$J^{\perp}$ is homogeneous].\par
\end{enumerate}
\end{definition} \noindent

Henceforth in this paper, inverse systems in $\Gamma$ (but not in $\Gamma ',{\widehat{\Gamma '}}$) are
assumed to be homogeneous. 
We will later need that the elements of $R_1$ act as differentials on
$\Gamma $ (Lemma \ref{transforma}).
\begin{lemma}\label{lindiff} If $\ell$ is an element of $R_1$, and $F,G \in \Gamma _u,\Gamma _v$,
respectively, then
\begin{equation}\label{diff} \ell \circ (F\cdot G)=(\ell \circ F)\cdot G+
F\cdot (\ell
\circ G).
\end{equation}
\end{lemma}
\proof By bilinearity, it suffices to show \eqref{diff} when $\ell$ is a variable,
and $F,G$ are monomials, whence it suffices to show it when $R={{\sf k}}[x], \Gamma ={{\sf k}}[X]$ in a single
variable, and for $
\ell = x, F=X^{[a]}, G=X^{[b]}$. There, it results from the definition of the multiplication
in the divided power ring $\Gamma $, and the usual Pascal triangle binomial identity.
\endproof
\par
\noindent We need a simple result relating inverse systems and ideals. First
we recall
\begin{definition}
\begin{enumerate}[(i)]
\item
If $V\subset R_j$, and $i\ge 0$, we have $R_i\cdot V=\langle hv\mid
h\in R_i,v\in V\rangle$;  if also $i\le j$ we have $(V:R_i)=\langle
h\in R_{j-i}\mid R_ih\subset V\rangle$. \par
\item If $W\subset \Gamma _j$ and $i\ge 0$, we have $R_{-i}\circ W =_{def}
(W:R_i)=\langle \{F\in \Gamma _{j+i}\mid R_i\circ F \subset
W\}\rangle$. If $W\subset \Gamma _j$ and $0\le i \le j$ we have
$R_i\circ W=\langle h\circ w\mid h\in R_i,w\in W\rangle$.
\end{enumerate}
\end{definition}
\begin{lemma}\label{pairVS}{\sc Inverse system and Matlis duality.} Assume that $(V,W)$ is a pair of vector
spaces satisfying
$ V\subset R_j,\ W\subset \Gamma _j$ and
$V^{\perp}\cap \Gamma _j=W$. Then
\begin{enumerate}[\rm (i)]
\item\label{pairVSi} If $0\le i$, then  $\quad (R_i\cdot V)^{\perp}\cap \Gamma _{j+i} =
W:R_i.$
\item\label{pairVSii} If $0\le i \le j$, then $\quad (V:R_i)^{\perp}\cap \Gamma _{j-i} =
R_i\circ W$.
\item
If $L\subset \Gamma $ is a homogeneous inverse system, then
$\Ann(L)\subset R$ is a graded ideal of $R$;
if
$I$ is a graded ideal of
$R$, then
$I^{-1}\subset
\Gamma $ is a homogeneous inverse system. Furthermore, $\Ann(L)^{-1} =L$; and $\Ann(I^{-1}) = I$.
Also $I^{-1}\cong \Hom _{\sfÊk}(R/I,{{\sf k}})$, the Matlis dual of $R/I$.\par \noindent
\item\label{pairVSiv} If the inverse system $L'\subset \Gamma '$  (not necessarily graded) has finite
dimension as
${{\sf k}}$-vector space, then
$I'=\Ann(L')$ is an $M'$-primary ideal of $R'$, where $M'=(y_1,\ldots ,y_n)$. Conversely,
an $M'$-primary ideal $I'$ of $R'$ determines a finite-dimensional inverse system of $L(I')\subset \Gamma
'$.
\par
\noindent
\item If $I'\subset R'$ is an ideal of finite colength $c$, defining an Artin quotient $R'/I'$ with
$s$ distinct maximal ideals,  then ${I'}^{-1}\subset \widehat{\Gamma
'}$ is a dimension-$c$ inverse system of the form
${I'}^{-1}=\bigoplus _1^s L'(i), L'(i) = V'(i)f_{p(i)}$, where
$V'(i)\subset \Gamma'$ is a finite inverse system, and $f_{p(i)}$ is
a specific power series (see \eqref{fp}).
\end{enumerate}
\end{lemma}
\proof For (i), note that $(R_i \cdot V)\circ F = 0$ if and only if
$V\circ (R_i\circ F)=0$. And the last equality is equivalent to
$R_i\circ F \subset W$.\par \noindent For (ii), note that for any
$h\in R_{j-i}$, $h\circ (R_i\circ W)=0$ if and only if $R_i\circ (
h\circ W) =0$. And the last equality is equivalent to $ R_ih\subset
V$.\par \noindent For (iii), note that $I$ is a graded ideal of $R$
if for each pair of non-negative integers $(i,j),\quad R_i\cdot
I_j\subset I_{i+j}$, or, equivalently, if $(I_{i+j}:R_i) \supset
I_j$. By (ii) the latter is equivalent to $ R_i\circ \langle
I^{-1}_{i+j}\rangle \subset I^{-1}_j$, implying that $I^{-1}$ is a
homogeneous inverse system. One shows similarly that the annihilator
$\Ann(L)\subset R$ of a homogeneous inverse system $L$ is an ideal,
using (i). That the double duals are the identities in this case
follows from the exactness of the pairing $R_i\circ \Gamma _i
\rightarrow k$. For (iv), note that if $L'\subset \Gamma'$ is finite
dimensional then $L'\subset \Gamma'_{\le j}$ for some integer $j$,
hence $\Ann(L')\supset {M'}^{j+1}$, and conversely. For (v), note
that since $I'=\cap I'(i)$, the inverse system ${I'}^{-1}$ is the
direct sum of the inverse systems $L'(i)$ of the components $I(i)$
at the points $p(i)$ of support. Then use Lemmas \ref{Macaffinedual}
and \ref{transforma} below.
\endproof
\smallskip \par

Usually, a homogeneous inverse system $W\subset \Gamma$ is not
finitely generated. In fact, if $W$ is finitely generated, then
$\dim_k W$ is finite, and by Lemma \ref{pairVS}\eqref{pairVSiv} $W$
determines an Artin algebra $A_W=R/I, I=\Ann(W)$ with $I$ an
$M=(x_1,\ldots ,x_{n+1})$-primary ideal. Recall
\begin{definition}
A graded ideal $I\subset R$ is \emph{saturated} if it has no irreducible component
primary to the irrelevent ideal $M$, equivalently, if $I=I:M^{\infty}=\{f\mid \exists k
\ge 0, M^k\cdot f\subset I\}$. This is equivalent to,
\begin{equation}\label{esat1}
 \forall a, b \in \mathbb N, a\le b \quad
I_a=(I_b:R_{b-a}) = \{f\in R_a\mid R_{b-a}\cdot f \subset I_b\}.
\end{equation}
 If $\dim(R/I)=1$, $I$ is saturated if and only if
there is a linear non-zero divisor for $R/I$ in $R$.
\end{definition}
Note that the condition of equation \eqref{esat1} results from the more usual saturation condition,
\begin{equation}
\label{esat1'}\exists N\in\mathbb N\mid \forall a, \forall b\ge
\max(N,a),\ I_a=(I_b:R_{b-a}).
\end{equation}
\begin{lemma}\label{sat} {\sc Macaulay's correspondence.} There is a one-to-one correspondence between
homogeneous inverse systems $W\subset \Gamma $ and graded ideals $I$
of $R$, given by $I \mapsto I^{-1}\subset \Gamma $, and $W\mapsto
I_W= \Ann(W)\subset R$. The ideal $I_W$ is saturated if and only if
the inverse system $W$ satisfies
\begin{equation}\label{esat2}
\forall a,b \in \mathbb N, a \le b \quad
W_a=R_{b-a}\circ W_b.
\end{equation}
Furthermore, the element $\ell \in R_i$ is a non-zero divisor for
$R/I$ if and only if $W=I^{-1}$ satisfies
\begin{equation}\label{e1nzd}
\forall b\in \mathbb N, b \ge i,\text { we have }\ell \circ
W_b=W_{b-i}.
\end{equation}\label{nzd}
\end{lemma}
\proof The 1-1 correspondence has been shown in Lemma \ref{pairVS}.
The relation \eqref{esat2} follows from \eqref{esat1}, using Lemma
\ref{pairVS}(ii). That $\ell$ is a non-zero divisor for $R/I$ is
equivalent to
\begin{equation}\label{enzd}
 \text{ for each integer }b \ge i, \text{and }\forall h\in R_{b-i},\,
 \ell \cdot h \in I_b \,\,\text{implies  that}\, h\in
I_{b-i}.
\end{equation}
Letting $W=I^{-1}$, we may translate the implication in \eqref{enzd}
equivalently as followings:
\begin{align}
(\ell \cdot h) \circ W_b= 0& \,\,\text{implies that}\, h\circ W_{b-i}=0, \notag\\
h\circ (\ell \circ W_b) =0 &\,\,\text{implies that}\, h\circ W_{b-i}=0, \notag\\
(\ell \circ W_b)^{\perp} \cap R_{b-i} &\subset (W_{b-i})^{\perp}\cap R_{b-i}, \notag\\
\ell \circ W_b &\supset W_{b-i}. \notag
\end{align}
Since by definition $R_i \circ W_b\subset W_{b-i}$, this shows the criterion
\eqref{e1nzd}.
\endproof \par

We will term an inverse system $W$ of $\Gamma $ \emph{saturated} if
$W$ satisfies \eqref{esat2} ; that is, $W$ arises from a saturated
ideal. We now recall the definitions of socle, and type.
\begin{definition}\label{type}
 (i) Let $A$ be a local ring with the
maximal ideal $m$. Then the $\emph{socle}$ of $A$, $Soc(A)$, is
defined as $(0:m)\subset A$. Further, if $A$ is an Artin algebra,
the $\emph{type}$ of $A$ is the dimension $\dim_{\sfÊk}\Soc(A)$, and the
socle degree is the maximum degree $i$ in which $\Soc(A)_i$ is
non-zero.\par \noindent (ii) If $\Z\subset \mathbb P^n$ is a
zero-dimensional scheme, and $I_{\Z}$ is a saturated ideal defining
$\Z$, then the type of $\mathcal O_\Z=R/I_\Z$ is defined as
$\dim_{\sfÊk}(\mathcal O_\Z/\ell \mathcal O_Z)$. Here, $\ell$ is a linear
non-zero divisor of $\mathcal O_\Z$ and $\mathcal O_\Z/\ell \mathcal
O_\Z$ is the Artin local ring with the maximal ideal
$(x_1,...,x_{n+1})/(I_\Z,\ell)$.
\end{definition}\noindent
It is well known that this notion of type does not depend on the
non-zero divisor $\ell$ used: the type is 
the rank of the last module in a free $R$-resolution of $A$, and
these ranks remain the same when we quotient by any non-zero divisor.
See \cite[Lemma 1.2.19]{BH} for the analogue in the case $B$ is
local of arbitrary dimension.
\begin{corollary}\label{macdual} Suppose $I\subset R$ has inverse system $W\subset \Gamma$. The vector space
$I_j/\langle R_1\cdot I_{j-1}\rangle$ of degree-$j$ generators of
$I$ is dual to the vector space $\langle W_{j-1}:R_1\rangle /W_j$. The vector space $(I_{j+1}:R_1)/I_j$ of
degree-$j$ socle elements of $A=R/I$ is dual to the vector space $W_j/R_1W_{j+1}$ of degree-$j$ generators
of $W$.
\end{corollary}
\proof This is immediate from Lemma \ref{sat} and Lemma \ref{pairVS} (\ref{pairVSi}),(\ref{pairVSii}).
\endproof \par

 We now show how to recognize the type of $I$ from the
inverse system; we then describe the inverse system of the projective closure of
a scheme. We will complete our listing of basic facts by characterizing the ideals defining
zero-dimensional schemes and their inverse systems (Theorem~\ref{zero-dim}, Proposition
\ref{satinvsys}).
\par
\begin{lemma}\label{socle} Let $I=I_\Z$ be the homogeneous saturated ideal defining a zero-dimensional
subscheme $\Z\subset \mathbb P^n$, let $W=I^{-1}\subset \Gamma $ be the inverse system of $I$. Let
$\ell\in R_1$ be a non-zero divisor for
$B=R/I$, set $A=B/\ell B$ with maximal ideal $m$. Denote by $\Gamma_{\ell}=\ell
^{\perp}\subset \Gamma$ the $R$-submodule of $\Gamma$ perpendicular to $\ell$,
and let $W_{\ell}=W\cap \Gamma_{\ell}$. Then
\begin{enumerate}[\rm(i)]\label{dualize}
\item\label{dualizei} $W_\ell$ is the dual module of $A$.
\item\label{dualizeii} $W_{\ell}/\langle M\circ W_{\ell}\rangle
\cong ( \Soc (A))^\vee$, the dual space to $\Soc (A)$.
\end{enumerate}
\end{lemma}
\begin{proof} Since $A=B/\ell B$ is isomorphic to $R/(I,\ell)$, its dual module is the inverse system
of $(I,\ell)$, so ${A}^{\vee}\cong I^{-1}\cap \ell ^{\perp}=W_\ell$: this shows (\ref{dualizei}). Also, $((I,\ell):M)$ is
perpendicular to $M\circ W_{\ell}$. Thus, we have $A=R/(I,\ell )$
and  $\Soc (A)=(0:m)=((I,\ell ):M)/(I,\ell )$, hence 
\begin{equation*}
(\Soc(
A))^\vee =(R/(I,\ell))^\vee / (R/((I,\ell ):M))^\vee{}
                    = (I,\ell)^\perp /((I,\ell):M)^\perp \cong W_{\ell}/\langle M\circ
W_{\ell}\rangle.
\end{equation*} This is (\ref{dualizeii}), and completes the proof.
\end{proof}\par
When $\Z$ is a zero-dimensional scheme of $\mathbb A^n\subset
\mathbb P^n$, its projective closure has an empty intersection with
the hyperplane at infinity: $z=0$, since $\Z$ is already closed.
However, in fact there is a graded Artinian algebra
$R'/(I_\Z)_{z=0}$ lying on the hyperplane at infinity, uniquely
determined by $\Z$, and whose Hilbert function determines
$H(R/I_\Z)$. We also show the connection with the global inverse
system. Recall that the Hilbert function $H(B)$ for an $R$-module
$B$ is the sequence $H(B)_i=\dim_kB_i$, with $B_i$ the degree-$i$
component of the associated graded module $Gr_M(B)$. We will write
Hilbert functions of submodules of $\Gamma$ in the order of
increasing degrees, so that $H(\Gamma)=H(R)$. We define the sequence
$\Delta H$ by $\Delta H_i=H_i-H_{i-1}$.
\begin{lemma}\label{pc}{\sc Projective Closure.} When $R={{\sf k}}[x_1,\ldots , x_n,z]$ and $\ell = z$ then
$\Gamma_z=z^{\perp}={{\sf k}}_{DP}[X_1,\ldots ,X_n]$. Suppose that $\Z$ is a
zero-dimensional scheme of $\mathbb A^n: z=1$, with global inverse
system $W=L_\Z$. Then $z$ is a non-zero divisor for $R/I_\Z$, and
$W_z=W\cap \Gamma_z$ satisfies
\begin{enumerate}[\rm(i)]
\item\label{pc(i)} $\Delta H(R/I_\Z) = H(R/(I_\Z,z))=H(W_z).$ \par
\item\label{pc(ii)} There is an exact sequence, $0\rightarrow W_z(i)\rightarrow
W(i)\xrightarrow{z{\scriptscriptstyle \circ}} W(i-1)\rightarrow 0$,
where the homomorphism $z{\scriptstyle \circ} :\ W(i)\rightarrow
W(i-1)$ is the contraction action of $z\in R$ on $\Gamma$ as in Definition \ref{cont}(i).
\item\label{pc(iii)} The above sequence is dual to $0\rightarrow (R/I_\Z)(i-1)\xrightarrow{m_z\cdot}
(R/I_\Z)(i)\rightarrow {{\sf k}}[x_1,\ldots ,x_n]/(I_\Z)_{z=0}\rightarrow 0$ where the homomorphism $\ m_z\cdot$ is
multiplication by
$z$.
\end{enumerate}

In (i),(ii) above, $z,W_z$ may be replaced by $\ell, W_\ell$, when
$\Z$ is an arbitrary zero-dimensional subscheme of $\mathbb P^n$,
provided $\ell$ is a non-zero divisor for $R/I_\Z$.
\end{lemma}
\begin{proof} If $z$ were a zero divisor for $R/I_\Z$, then $z$ would be
contained in an associated prime of $I_\Z$, contradicting the
assumption $\Z\subset \mathbb A^n$. (\ref{pc(iii)}) is
immediate, That $z$ is a non-zero divisor implies \eqref{pc(iii)}.  The statement
(\ref{pc(ii)}) follows from (\ref{pc(iii)}) by dualizing, and
(\ref{pc(i)}) follows from these exact sequences by taking vector
space dimensions.
\end{proof}\par
\begin{example}\label{homogHilbf} Let $I_\Z=(xy,x^2z-y^3,x^3)\subset R={{\sf k}}[x,y,z]$; then $\Z$ is a degree-5
scheme concentrated at the point $p_0=(0:0:1)$ of $\mathbb P^2$ (the
origin of $\mathbb A^2$), having global Hilbert function
$H_\Z=H(R/I_\Z)=(1,3,5,5,\ldots )$. The Artin algebra
$A=R/(I_\Z,z)\cong {{\sf k}}[x,y]/(xy,x^3,y^3)$ has Hilbert function
$H(A)=\Delta H_\Z=(1,2,2,0)$, and is the
$\it{boundary}$ of $\Z$ on the line at infinity: $z=0$. The inverse
system $W=(I_\Z)^{-1}\subset \Gamma={{\sf k}}_{DP}[X,Y,Z]$ satisfies
\begin{align*}
W_3&=\langle X^{[2]}Z+Y^{[3]},Y^{[2]}Z,YZ^{[2]},XZ^{[2]},Z^{[3]}\rangle\\
W_2&=\langle X^{[2]},Y^{[2]},YZ,XZ,X^{[2]}\rangle\\
W_1&=\langle X,Y,Z\rangle ;\quad W_0=\langle 1 \rangle,
\end{align*}
and $W_z=\langle 1,X,Y, X^{[2]},Y^{[2]} \rangle=W\cap \Gamma_z
=W\cap {{\sf k}}_{DP}[X,Y]\subset \Gamma$ is the dual module to $A$.
\par When we consider $\Z\subset \mathbb A^2$, by setting $z=1$ in $I_\Z$, we find $I'=(xy,x^2-y^3)$, which
defines a scheme concentrated at $p_0$ of local Hilbert function
$H'=(1,2,1,1)$, different from $\Delta H_\Z$. \par If we consider
instead $\Z',$ defined by $(x^2, xy,y^4)$, we would find the same
local Hilbert function $H'$ for $\Z'$, but now $H_{\Z'} =
(1,3,4,5,\ldots)$, the sum function, since $\Z'$ is $\it{conic}$.
This example shows that the local Hilbert function $H'$ does not
determine the global Hilbert function $H_\Z$.
\end{example}
We recall next a well known result, see for example
\cite{GM,Mar,Or}. We quote most of it from \cite[Theorem 1.69]{IK}.
A scheme $\Z\subset \mathbb P^n$ is \emph{arithmetically
Cohen-Macaulay} if $R/I_\Z$ is Cohen-Macaulay; if $\dim \Z=0$, this
is equivalent to there being a non-zero divisor in $R$ for $R/I_\Z$.
Recall that $\tau (\Z)=\min\{i\mid
\dim_{\sfÊk}((R/I_\Z)_i)=s\}$.\pagebreak[3]
\begin{theorem}\label{zero-dim} {\sc Punctual schemes.} Let $\Z$
be a degree-$s$ zero-dimensional subscheme of
$\mathbb P^n$, and let
$I=I_\Z$ be its saturated defining ideal. Then $\Z$ is arithmetically Cohen-Macaulay, and
\begin{enumerate}[\rm(i)]
\item\label{zeroi} The Hilbert function $H(R/I)$ is nondecreasing in $i$, and stabilizes at the
value $s$ for $i\ge \tau(\Z)$.  We have $\tau (\Z)\le s-1$, with
equality if and only if $\Z$ is contained in a line. \par
\item\label{zeroii} The Castelnuovo-Mumford
regularity
$\sigma = \sigma (\Z)$ satisfies $\sigma = \tau (I)+1$. In particular, if $i\ge \sigma$, then
$I_i=R_{i-\sigma}\cdot I_{\sigma}$. Thus, $I$ is generated  by degree $\sigma$.\par
\item\label{zeroiii} The first difference $\Delta (H(R/I))=C=(1,c_1,\ldots
c_{\tau},0)$ is an $O$-sequence (the Hilbert function of some Artin quotient of
$R'$), with $s=\sum c_i$. \par
\item\label{zeroiv} Every
$O$-sequence $C=(1,c_1,\ldots ,c_\tau ,0 ),\  c_1\le n, \ c_\tau
\not=0, \sum c_i=s,$ occurs as $\Delta H(R/I)$ for some degree-$s$
zero-dimensional scheme $\Z$ with $ \tau (\Z)=\tau$, consisting of
smooth points.
\end{enumerate}
Conversely, any saturated ideal $I\subset R$ satisfying the  Hilbert
function conditions (\ref{zeroi}), (\ref{zeroiii}) above for
$H(R/I)$ is the defining ideal of such a zero-dimensional subscheme,
namely $\Z=\Proj (R/I)\subset \mathbb P^n$.
\end{theorem}
\begin{proof}[Proof outline]  There are direct proofs of   (\ref{zeroi})
--(\ref{zeroiii}) in \cite{Or,GM}; see also \cite[Theorem 1.69]{IK}.
Let $I$ be an ideal of $R$, such that $R/I$ has dimension one. One
can show cohomologically that $I$ saturated is equivalent to $R/I$
being Cohen-Macaulay (see, for example, \cite[Lemma 1.67]{IK}), and
this is equivalent to a general element $\ell$ of $R_1$ being a
non-zero divisor for $R/I$. Then $\Delta H$ is the Hilbert function
of $R/(I_\Z,\ell )$, so is an $O$-sequence. That $\sigma=\tau+1$ is
the Castelnuovo-Mumford regularity is shown cohomologically. That
$\tau=s-1$ if and only if $\Z$ is on a line is a consequence of
$\Delta H$ being an O-sequence summing to $s$. So $\tau=s-1$ if and
only if $\Delta H=(1,1,\ldots ,1)$, which is equivalent to
$(H_\Z)_1=2$. P. Maroscia's result (\ref{zeroiv}) is shown by
deforming monomial ideals defining Artin quotients of $R'$ having
Hilbert function $C$ (see \cite{Mar,GM}).  The last statement
concerning a converse follows from the 1-1 correspondence between
saturated ideals of $R$ and subschemes of $\mathbb P^n$.
\end{proof} \par
The first difference $\Delta H=(1,c_1,\ldots
,c_{\sigma -1},0,\dots )$ is sometimes termed the \emph{$h$-vector} of $\Z$ (see, for example, \cite[\S
1.4]{Mig}).
\begin{proposition}\label{satinvsys}{\sc Inverse system of a punctual scheme.} The inverse system $W$ is the
inverse system of a saturated ideal $I_\Z$, where  $\Z$ a degree-$s$
zero-dimensional scheme of $ \mathbb P^n$, regular in degree $\sigma
$ if and only if
\begin{enumerate}[\rm(i)]
\item\label{satinvsysi(i)} $\dim_{\sfÊk} W_j=s\,\  \forall j\ge \sigma -1$, and \par
\item\label{satinvsysi(ii)}$\exists N\in \mathbb N\mid \forall a, \forall b\ge \max(N,a),\
W_a=R_{b-a}\circ W_b$.
\end{enumerate} \par \noindent
The condition (ii) implies the apparently stronger \eqref{esat2}.
Furthermore, if $\Z$ is such a degree $s$ scheme regular in degree
$\sigma$,  then for all $b\ge \sigma$,
\begin{equation}\label{sigmadetermines}
 W_b= W_{\sigma}:R_{b-\sigma}=\{ f\in \Gamma_b\mid R_{b-\sigma}\circ f \subset W_{\sigma}\}.
\end{equation}
\end{proposition}
\proof That an inverse system $W$ arising from such a scheme $\Z$
must satisfy (i),(ii), is immediate from Lemma \ref{sat}, and
Theorem \ref{zero-dim}. Suppose conversely that $W$ satisfies
(i),(ii). The condition (\ref{satinvsysi(ii)}) implies that
$I=\Ann(W)$ is a saturated ideal, by Lemma \ref{pairVS}(ii) applied
to \eqref{esat1'}. By (i), its Hilbert polynomial is $s$, so $I$
defines a zero-dimensional scheme of degree-$s$ having regularity
degree no greater than $\sigma $. And condition
(\ref{satinvsysi(i)}) implies that $H(R/I) = (1,\ldots ,s,s,\ldots
)$, with the first $s$ occurring before degree $\sigma-1$. The two
imply that $R/I$ is Cohen-Macaulay of dimension 1, and regularity
degree no greater than $\sigma$ (see Theorem \ref{zero-dim}). By
Theorem \ref{zero-dim} (\ref{zeroii}) if $\Z$ is such a scheme, the
ideal $I=I_\Z$ is generated by degree $\sigma$; the last equation
\eqref{sigmadetermines} is a translation of this generation fact
into the inverse system language, using Lemma \ref{pairVS}
(\ref{pairVSi}).
\endproof \par

\section{Inverse system of a zero-dimensional scheme.}\label{invsyspuncsch} \par
In Section \ref{2.1} we consider a zero-dimensional scheme
$\Z\subset \mathbb P^n$ concentrated at a single point $p_0$ that is
a coordinate point. These are simpler since the local inverse system
lies in the ring $\Gamma '$. In Section~\ref{arbitrarypoint} we
study a scheme $\Z$ concentrated at an arbitary point $p$, for which
the local inverse system lies in the completion $\widehat{\Gamma
'}$. In Section \ref{schfinsup} we consider the inverse system for a general zero-dimensional
scheme $\Z$ with finite support. In each case we show how to directly
homogenize the local inverse system for $\Z$ to
obtain the global inverse system $L_\Z \subset \Gamma$ of the global
defining ideal $I_\Z \subset R$. Recall that we denote by
${m_p}\subset R$ the homogeneous ideal of the point $p$; if
$p=(a_1:\ldots :a_n:1)$, then ${m_p}=(a_1z-x_1,\ldots ,a_nz-x_n)$.
Recall also that the homogeneous ideal $I\subset R$ is concentrated
at the point $p\in \mathbb P^n$ if and only if there exists an
integer $u>0$ such that
\begin{equation}\label{eqn1}
 {m_p} \supset I \supset {m_p}^u.
\end{equation}
When $\cha {{\sf k}} = 0 \hbox{ or } \cha {{\sf k}} >j$ we have  (\cite{T1}, \cite[Theorem I]{EmI}, \cite{EhR} for $u=2$)
\begin{equation}\label{mpsannih}
({m_p}^u)^{\perp}\cap \Gamma _j = \Gamma _{u-1}\cdot L_p^{[j+1-u]}.
\end{equation}
Here, the right hand side is interpreted as $\Gamma_j$ if $u>j$.
Thus, the condition
\eqref{eqn1} corresponds to the following condition on the
inverse system

\begin{equation}\label{singpt}
 L_p^{[j]}\subset [I^{-1}]_j \subset \Gamma _{u-1}\cdot L_p^{[j+1-u]},
\end{equation}
where if $p=(a_1: \ldots :a_n:1)$ then $L_p=a_1X_1+\cdots +
a_{n}X_{n}+X_{n+1}$, and $L_p^{[j]}$ denotes the form
$L_p^{[j]}={L_p}^j/j!=\sum_{J\mid |J|=j}a^J\cdot X^{[J]}$,
proportional to the divided power $L_p^j$. We have shown
\begin{lemma}\label{homogdefp}
The following conditions are equivalent:
\begin{enumerate}[\rm(i)]
\item The homogeneous ideal $I$ of $R$ defines a zero-dimensional
scheme concentrated at the point $p$ of $\mathbb P$;
\item There exists an integer u such that ${m_p} \supset I \supset
{m_p}^u=(a_1z-x_1,\ldots ,a_nz-x_n)^u$;
\item There exists an integer $\alpha=u-1$ such that the inverse system $I^\perp$
satisfies
\begin{equation}\label{econcentatp}
{{\sf k}}_{DP}[L_p]\subset I^\perp \subset ({m_p}^u)^{\perp}=\Gamma _{\le \alpha}\cdot {{\sf k}}_{DP}[L_p].
\end{equation}
\end{enumerate}
In particular, if the homogeneous ideal $J$ of $R$ defines a
zero-dimensional scheme concentrated at
the point $p_0=(0:\ldots :0:1)\in \mathbb P^n$, then ${{\sf k}}[Z]\subset J^\perp
\subset \Gamma _{\le a}\cdot {{\sf k}}[Z],Z=X_{n+1}$ for some $a\ge 0$.
\end{lemma}
The following example shows the need for our limitation on the characteristic of ${{\sf k}}$ (\S
\ref{notation}).
\begin{example}\label{char} Let $n=1$,$R={{\sf k}}[x,y], \Gamma ={{\sf k}}[X,Y]$. Choose the point
$p=(a_1:1)\in \mathbb P^1$, and $I={m_p}^2=(x-a_1y)^2$, then we have that ${[I^{\perp}]}_2$ satisfies
\begin{align} \left( a_1X+Y\right) ^{[2]}\subset{[I^{\perp}]}_2 \subset \Gamma _1\cdot L_p&=\langle X,Y
\rangle \cdot
\left( a_1X+Y\right) \notag\\
                   &=\langle 2a_1X^{[2]}+XY,
a_1XY+2Y^{[2]}\rangle ,
\end{align}
 provided $\cha {{\sf k}} \not=2$. When $\cha {{\sf k}}=2$ and $a_1=0$ the space on the right is just
$\langle XY\rangle$, so is one-dimensional, and is not all of
${({{m_p}^2)}_2^{\perp}}$, which also includes
${L_p}^{[2]}=a_1^2X^{[2]}+a_1XY+Y^{[2]}$. Thus, equation
\eqref{mpsannih} and the equality on the right of Lemma
\ref{homogdefp} \eqref{econcentatp} do not extend to characteristic
$p\not= 0$, when $p$ is less than or equal to the degree $j$ (here
$j=2$) of the forms being considered.
\end{example}
Recall that the \emph{socle degree}
$\alpha$ of a local Artin algebra $A$ of maximal ideal $m$ is the highest integer such that
$m^{\alpha}A\not=0$, but
$m^{\alpha +1}A=0$, and that the point $p_0=(0:\ldots :0:1)$. For a zero-dimensional scheme $\Z$, we now define
$\alpha(\Z)$ to be the maximum local socle degree of
$\Z$. 
\begin{definition}\label{socled} (i)If $\Z$ is a zero-dimensional scheme concentrated at $p_0$, we let
$\alpha (\Z)$ denote the highest socle degree of $(R'/J)$, where
$J\subset R'$ defines $\Z$. Equivalently, $\alpha (\Z)$ is the
highest degree of an element of $J^{-1}\in \Gamma '$. If $\Z$ is
concentrated at a point $p$, then $\alpha (\Z)$ is defined similarly
using the local ring at $p$ (see Section \ref{arbitrarypoint}).\par
\noindent (ii) More generally, if a zero-dimensional scheme $\Z$ has
decomposition $\Z=\Z(1)\cup \cdots \cup \Z(k)$ as the union of
irreducible components $\Z(1),\ldots ,\Z(k)$, each concentrated at
(distinct) points $p(1),\ldots ,p(k)$, then $\alpha(\Z)=\max\{
\alpha (\Z(1)), \ldots ,\alpha ( \Z(k))\}$ of the local socle
degrees.
\end{definition}

\subsection{Schemes concentrated at a coordinate point.}\label{2.1}
We will fix the coordinate point as $p=p_0=(0:\ldots :0:1)$; we
denote $x_{n+1},X_{n+1}$ by $z,Z$, respectively.
We let $R'={{\sf k}}[y_1,\ldots ,y_n]$ be the coordinate ring of affine
space $\mathbb A^n$, the locus on $\mathbb P^n$ where
$x_{n+1}\not=0$ and we let $\Gamma '={{\sf k}}_{DP}[Y_1,\ldots ,Y_n]$ be the
divided power ring.
Let $\mathcal I_p\subset \mathcal O_p$ be an ideal defining a
zero-dimensional scheme $\Z$ concentrated at $p$. Then $\mathcal I_p\supset
m_p^{\alpha(\Z) +1}$ and each element of $\mathcal I_p$ may be
written mod $m_p^{\alpha (\Z)+1}$ as a polynomial $h$ in
$R'={{\sf k}}[y_1,\ldots ,y_n]$ of some degree $t$ no greater than
$\alpha(\Z)$. The homogenization of $h$ to degree $u$ is
\begin{equation} \Homog(h,z,u)= z^u\cdot h(x_1/z,\ldots
,x_n/z),
\end{equation}
for $u\ge t$, and $0$ otherwise. The homogenization $I_\Z$ of
$\mathcal I_p$ is spanned by
${m_p}^{\alpha(\Z)+1}$, and by all homogenizations of such elements
$h\in \mathcal I_p$:
\begin{equation} I_\Z = \left( \Homog(h,z,u) \mid u \in
\mathbb {Z}^+,h\in \mathcal I_p \text{ degree }h\le \alpha(\Z)\right) +{m_p}^{\alpha(\Z)+1}
\end{equation}
\par
Recall that the inverse system $L_\Z$ in $\Gamma $ of $I_\Z$ consists of
all elements of $\Gamma $, annihilated by $I_\Z$.  Given a point $p=(a_1:\ldots :a_n:1)$ of
$\mathbb P^n$, we let $L_p=a_1X_1+\cdots +a_nX_n+Z\in \Gamma$.
\begin{definition}{\it Homogenization of an inverse system at a
point.}\label{defhomcoordp}
\par
\begin{enumerate}[(i)]
\item  Let $F\in\Gamma[1/Z,1/Z^{[2]},\ldots ]$. We denote by $F\cdot_{rp} Z^{[u]}$ the result of raising
the
$Z$-degree of the $Z$-factor in each term by $u$, without changing the coefficients that
appear. For example, if $F=X_1X_2/Z^{[2]}+X_2^{[4]}/Z^{[4]}\in {{\sf k}}_{DP}[Y_1,Y_2]$,
then
$F\cdot_{rp}Z^{[4]}=X_1X_2Z^{[2]}+X_2^{[4]}$. We may also write
$Z^{[u]}\cdot_{rp}F$ for  $F\cdot_{rp}Z^{[u]}$. If $w\in \Gamma$ has the form $\ w=\sum w_i\cdot
L_p^{[k-i]}, w_i\in \Gamma '$, then we denote by
 $\ w\cdot_{rp}L_p^{[u]}$ the product
\begin{equation}
w\cdot_{rp}L_p^{[u]}=\sum w_i \cdot L_p^{[k+u-i]}.
\end{equation}
\item Let
$f\in \Gamma '={{\sf k}}_{DP}[Y_1,\ldots ,Y_n]$ satisfy $f=\oplus f_i,
f_i\in \Gamma'_i$, and let $L_p=a_1X_1+\cdots +a_nX_n+Z$. Then for
any integer $u\ge 0$ we define the inverse system homogenization
\begin{equation}
\Homog(f,L_p,u) = \sum_{0\le i\le u}f_i(X_1,\ldots ,X_n)\cdot L_p^{[u-i]}.
\end{equation}
For example, if $f=Y_1Y_2+Y_2^{[4]}$, then $f(X_1/Z,X_2/Z)=F$ above, and
$\Homog(f,Z,4)=X_1X_2Z^{[2]}+X_2^{[4]}$, while $\Homog(f,Z,3)=X_1X_2Z$.
\item  Let $L'\subset \Gamma '$ be an
inverse system (so $L'$ is an $R'$-submodule of $\Gamma '$), and suppose $p$ fixed. Then we define
\begin{equation*}L'[u]=\langle \{ \Homog(f,L_p,u)\ \forall f\in L'\} \rangle \quad \quad \quad \quad
\quad
\quad
\end{equation*}
and we define the homogenization of the inverse system $L'$,
\begin{equation} \Homog(L',L_p) =\bigoplus_{u\ge 0} L'[u] =\langle \Homog(f,L_p,u)\mid f\in
L',u\ge 0\rangle .
\end{equation}
If we leave out the homogenizing form or do not specify $p$, then we assume $L_p=Z$, $p=p_0$.
\end{enumerate}
\end{definition}
\noindent Note that this
definition allows $\Homog(f,Z,u)$ to be nonzero even if $u$ is smaller than the degree
 of $f$; this
is natural here, since the global inverse system is closed under the
contraction action of
$R$. Thus, for example
\begin{equation*}
z\circ \left( X_1X_2Z^{[2]}+X_2^{[4]}\right) = X_1X_2Z.
\end{equation*}
We of course wish to show that if $L'\subset \Gamma '$ is the
inverse system of $\mathcal I_p\subset \mathcal O_p$, then
$L=\Homog(L',Z)\subset \Gamma $ is the inverse system of $I_\Z$
(Lemma \ref{homoginv}). We also wish to show how to obtain from $L'$
the key $\it{generators}$  of $L$ --- which is infinitely generated. To this end, we need a
basic result.
\begin{lemma}\label{homogann}{\sc Homogenization and duality.}
Suppose that $h'\in R'$ has degree $a$, that $f'\in \Gamma '$,
that $i\ge a$, and that $w\in \mathbb Z$.
 Let $h=h'[i]=\Homog(h',z,i)$ and $f=f'[i+w]=\Homog(f',Z,i+w)$. Then
\begin{equation}\label{ehomogann}
 h\circ f=h'[i]\circ f'[i+w] = ( h'\circ f') [w].
\end{equation}
\noindent
In particular,
\begin{equation}\label{2.11}
h'\circ f' = 0 \Rightarrow (h'\circ f')_{\le w} =
0\Leftrightarrow (h'\circ f')[w]=0
\Leftrightarrow h'[i]\circ f'[i+w]= 0;
\end{equation}
and if $f'$ has degree $b$, then
\begin{equation}\label{vanish} h'\circ f' = 0 \Leftrightarrow (h'\circ f')_{\le b}
= 0\Leftrightarrow (h'\circ f')[b]=0
\Leftrightarrow h'[i]\circ f'[i+b]= 0.
\end{equation}
\end{lemma}

\begin{proof}
Let $h'=\sum_{u=0}^a h_u$ and $f'=\sum_{v=0}^b f_v$. Then $h=h'[i]=\sum_{u=0}^a
h_u z^{i-u}$ and
$f'[i+w]=\sum_{v=0}^{\min\{b,i+w\}} f_v Z^{[i+w-v]}$. Now, we have formally (below, $Z^{[c]}=0$ if
$c<0$),

\begin{align}
 h'[i]\circ f'[i+w]&=\sum_{u=0}^a \left(
\sum_{v=u}^{\min\{b,u+w\}}h_u(x_1,\ldots ,x_n) \circ f_{v}(X_1,\ldots ,X_n)\cdot Z^{[w+u-v]}\right)
\notag \\ &= \sum_{u=0}^a\left( \sum_{v=u}^b
h_u\circ f_{v}\cdot  Z^{[w+u-v]}\right)= h'\circ f' [w] . \notag
\end{align}
\noindent
The second equation is immediate from the first, since
homogenization to degree $w$ in $\Gamma '$ annihilates terms in $h\circ f$
having degree greater than $w$. The third is immediate from the second.

\end{proof}

\begin{lemma}\label{basicann} Suppose $h'\in R'$ has degree no greater than $a$, and
$f'\in \Gamma '_{\le b}$, and let $h=\Homog(h',z,a)\in R$, $f=\Homog(f',Z,b)\in \Gamma $. Then
\begin{equation} h'\circ f' = 0 \Leftrightarrow h\circ (f\cdot_{rp} Z^{[a]})=0.
\end{equation}
\end{lemma}
\begin{proof} In \eqref{vanish}, take $i=a$, and note that
$\Homog(f',Z,a+b)=\Homog(f',Z,b)\cdot_{rp} Z^{[a]}$.
\end{proof}
\begin{lemma}\label{homoginv}{\sc local to global inverse systems.} Suppose that $p=p_0=(0:\ldots
:0:1)$ in $\mathbb P^n$, and that $L'\subset \Gamma '$ is the
inverse system of $\mathcal I_p\subset \mathcal O_p$, where
$\mathcal I_p$ defines a degree-$s$ zero-dimensional scheme $\Z$
concentrated at $p$. Then $\Homog(L',Z)\subset \Gamma $ is the
inverse system $L_\Z$ of $I_\Z\subset R$.
\end{lemma}
\begin{proof} Let $S=\Homog(L',Z)$. It is
immediate from \eqref{2.11} in Lemma
\ref{homogann} that
$I_\Z\circ S=0,$ so
$L_\Z\supset S$. Also, note that $S$ is an $R$-module:  $R\circ S\subset S$. To show this, it
suffices to check that if
$f\in S_u, h\in R_1$, then $h\circ f\in S_{u-1}$. Let $f=\Homog(f',Z,u), f' \in L'$. Note that
$z\circ f=\Homog(f',Z,u-1)$, so is in $S$. Also, if $1\le i\le n$ then considering each term, it
is easy to see that $\ x_i\circ f =
\Homog(y_i\circ f',Z,u-1)$, so is in $S$. This shows $R_1\circ S\subset S$, and by induction that
$S$ is an $R$-module. \par
For $i\ge \tau (\Z)$, $\dim_{\sfÊk} (L_\Z)_i=s$. For $i\ge \alpha (\Z)$, the socle degree, $\dim_{\sfÊk} S_i=s$, since
the homomorphism $f\in L'\rightarrow f[i]$ is an isomorphism of $\Gamma '_{\le i}$ into $\Gamma _i$,
and $\dim_{\sfÊk} L'=s$. Since $S\subset L_\Z$, we have $S_i=(L_\Z)_i$ for $i\ge \max\{
\alpha(\Z),\tau(\Z)\}$. Since $I_\Z$ is saturated, by Lemma \ref{sat} we have
that there is an integer $N$ such that $(L_\Z)_u=R_{i-u}\circ (L_\Z)_i$ for all
$i\ge N$ and $u\le i$. We conclude that $L_\Z\subset S$, completing the proof of the Lemma.
\end{proof} \par \noindent
The following result is a consequence of Lemma
\ref{homoginv}. We give a direct proof.
\begin{lemma}\label{closerp} When $j\ge \alpha (\Z)$, $(L_\Z)_{\ge j}$ is closed under the raised power
action of
$Z$:
$f\rightarrow f\cdot_{rp}Z^{[u]}$.
\end{lemma}
\begin{proof} Let $f\in (L_\Z)_j$, set $f_1=f\cdot_{rp}Z, I=I_\Z$, and suppose by way of
contradiction that $h\in I_{j+1}$ satisfies $h\circ f_1\not= 0$.
Then $h=zh_1+h', h_1\in R_j,h'\in R'_{j+1}$. Since $j\ge \alpha
(\Z)$, $h'\in J_{j+1}$, where $J$ is the ideal defining $\Z\subset
\mathbb A^n$; hence $zh_1\in I_{j+1}$, implying $h_1\in I_j$, since
the homogenizing variable is a non-zero divisor of $R/{I_\Z}$. But
we have $h'\circ f_1=0$ (as each term of $f_1$ has a $Z$-factor),
hence $zh_1\circ f_1=(h\circ f_1 -h'\circ f_1) \not= 0$. Then
$h_1\circ f= zh_1\circ f_1\not=0$,  a contradiction since $h_1\in
I_j$.
\end{proof} \par \noindent
The assumption $j\ge \alpha (\Z)$ in the above Lemma is necessary (see Example \ref{noncloserp}). We now
state a key result concerning the generation of the homogenized inverse system.
\begin{lemma}\label{homoginvsys} {\sc Generators for the global inverse system.} Suppose that
$V'\subset
\Gamma '_{\le
\alpha}$ generates the inverse system $L'$ of $\mathcal I_p$, and denote by $I_\Z$ the homogenization of
$\mathcal I_p$, and by $V$ the subspace $\Homog (V',Z,\alpha)$ of $\Gamma _\alpha$. Then the
inverse system
$L_\Z={I_\Z}^{-1} \subset \Gamma $ satisfies
\begin{align}\label{isgen} (L_\Z)_j &= \Homog (L'_{\le \alpha},Z,j) \notag\\
          &= R_\alpha \circ(V\cdot_{rp} Z^{[j]}).
\end{align}
\end{lemma}
\begin{proof} Since $L'=L'_{\le \alpha}$, the first equality follows from Lemma \ref{homoginv}.
That $V'$ generates $L'$ is equivalent to $L'=R'_{\le \alpha }\circ
V'$. If $h'\in R'_{\le \alpha }$ and $v'\in V'$, let $v=v'[\alpha]$;
then by Lemma \ref{homogann} $h'[\alpha ]\circ (v\cdot_{rp}Z^{[j])}=
h'[\alpha ]\circ v'[j+\alpha ]=(h'\circ v')[j]\in \Homog(L',Z,j)$.
This shows $\Homog(L',Z,j)\subset R_\alpha \circ(V\cdot_{rp}
Z^{[j]})$.  Lemmas \ref{closerp} and \ref{homoginv} show that
$V\cdot_{rp} Z^{[j]}\subset
\Homog(L',Z,j+\alpha)=(L_\Z)_{j+\alpha}$, implying the opposite
inclusion. This completes the proof of \eqref{isgen}.
\end{proof}
\begin{example}\label{findhomogid} The above Lemma \ref{homoginvsys} can be used to calculate the
homogenization of an ideal, given generators of the local inverse system. Begin with the local
ideal $I'\subset R'={{\sf k}}[y_1,y_2]$,
$I'=\Ann(f'), f'=Y_1^{[8]}+Y_2^{[8]}+Y_1^{[3]}Y_2^{[3]}+(Y_1+Y_2)^{[6]}$. Then
\begin{equation*}
I'=(3y_1^6-4y_1y_2^5+y_2^6+y_1^2y_2^2-2y_1y_2^3, y_1^6-y_2^6+y_1^3y_2-y_1y_2^3,m_{p_0}^9),
\end{equation*}
 of local Hilbert function
$H'=H(R'/I') = (1,2,3,4,3,2,2,2,1)$, and $I'$ defines a degree-20
zero-dimensional scheme $\Z=\Spec(R'/I')$ concentrated at $p_0=(0:0:1)\in
\mathbb P^2$, with $\alpha (\Z)=8$. The homogenized ideal
$I=I_{\Z}\subset R$ has more than two generators, and is tricky to
find directly --- we may homogenize a standard basis, using a computer algebra program. However, by homogenizing $f'$, forming
$f=\Homog(f',Z,8)=X_1^{[8]}+X_2^{[8]}+X_1^{[3]}X_2^{[3]}Z^{[2]}+(Y_1+Y_2)^{[6]}Z^{[2]}$,
we may calculate $W_8=R_8\circ (f\cdot_{rp}Z^{[8]})$, and we can
find $J=\Ann W_8$. In the {\sc{Macaulay}} algebra program \cite{BSE}
we found the contraction of $R_8$ with $f$, then used the script
``$<$l$\_ \text{from}\_ \text{dual}$'' to find $J$. The
homogenized ideal $I=J_{\le 8}+m_p^9$, where $m_p=(x_1,x_2)$. In
this case $J_{\le 8}$ already generates $I$, since $\Delta
(H(R/J_{\le 8}))$ has the correct degree 20. We found $I=I_\Z$ satisfies
\begin{multline*}
I=(x_1^3x_2^2+x_1^2x_2^3-3x_1x_2^4, \ x_1^4x_2-x_1x_2^4,\
x_1x_2^5-x_2^6+(3/4)x_1^3x_2z^2-(1/4)x_1^2x_2^2z^2-(1/4)x_1x_2^3z^2,\\
x_1^6-x_2^6+x_1^3x_2z^2-x_1x_2^3z^2),\qquad
\end{multline*}
of Hilbert function $H_\Z=H(R/I)$ satisfying $\Delta H_{\Z}=(1,2,3,4,5,4,1)$, and $\sigma (\Z)=7$.
\end{example}
The following Proposition extends some of the above
results. Recall that we use the notation $z$ for $x_{n+1}$ and $Z$
for $X_{n+1}$ and we denote by $m_z$ or $m_{z^a}$ multiplication by
$x_{n+1}$ or by $x_{n+1}^a$ in $R$ or in $A=R/I$.  We denote by
$L=\Homog (J^{-1},Z)\subset \Gamma$, the homogenization of the
inverse system $J^{-1}\subset \Gamma '={{\sf k}}_{DP}[Y_1,\ldots ,Y_n]$ and
we let $L_i[j]= \Homog({L_i}_{\{X_{n+1}=1\}},X_{n+1},j)$. This is
just $L_i[j]=z^{i-j}\circ L_i\in \Gamma _j$ if $j\le i$, and
$L_i[j]=Z^{[j-i]}\cdot_{rp} L_i$ if $j\ge i$ and is obtained in any case by
changing each $X_{n+1}^{[u]}$ factor appearing in a monomial term of
an element $F\in L_i$ to $X_{n+1}^{[u+j-i]}$, forming an element
$F[j] \in \Gamma _j$. Note that if $j>i$, and $F\in L_i$, then
$F[j]$ is not necessarily in $L_j$ (see Remark \ref{strange} and
Example \ref{easyhomog} below). Recall that $M'$ is the maximal
ideal of $R'={{\sf k}}[y_1,\ldots ,y_n]$ at the origin, and $m_p=(x_1,\ldots
,x_n)\subset R$ is the homogeneous maximal ideal of $R$ at the
corresponding point $p\in \mathbb P^n$. 
\begin{proposition}\label{AuxMainL}{\sc Homogenization for schemes with support $p_0$.}
Let $J\subset R'$ define a zero-dimensional scheme
$\Z=\Spec(R'/J)$ concentrated at the origin and let
$\alpha=\alpha(\Z)$ be the socle degree of $R'/J$ (Definition
\ref{type}). Let $L'=J^{-1}\subset \Gamma '$ be the affine
inverse system of $\Z$, and set $I = \Homog(J,z)\subset R$,
$L=\Homog(L',Z)=\bigoplus _i L'[i]$. Then we have
\begin{enumerate}[\rm(i)]
\item\label{is1} $I=I_\Z$, and is a saturated ideal primary to the maximal ideal ${m_p}, p=(0:
\ldots :0: 1) \in \mathbb P^n$, and it satisfies $m_p\supset
I\supset {m_p}^{\alpha +1}$. In particular, $I_a=(I_b:R_{b-a)}$ for
$a\le b$, and $I_b=R_{b-a}I_a$ for $b\ge a\ge \alpha (\Z)$.
Furthermore, if $a\le b$, then $I_a=(I_b:z^{b-a)}$, and if $b\ge
\alpha$ we have $I_b=z^{b-\alpha}\cdot I_\alpha +({m_p}^{\alpha
+1})_b$.
\item\label{is3}  $L=L_\Z$ and satisfies $L_{a} = z^{b-a}\circ L_b$ for $a\le b$, and
${{\sf k}}_{DP}[Z]\subset L\subset
\Gamma_{\le
\alpha}\cdot {{\sf k}}_{DP}[Z]$. Furthermore, for $\alpha \le a \le b $ the map $F\rightarrow F[b]$ taking
$L_a$ to
$L_{b}$, and the map $z^{b-a}\circ :L_{b}\rightarrow L_a$ are
inverse isomorphisms.  Also, $L$ satisfies $L_b=L_a[b]$ for any pair $(a,b)$ satisfying $a\ge
\alpha$.
\item\label{is4}  $A$ satisfies $m_{z^{b-a}}: A_a \rightarrow A_{b}$ is injective for $a\le b$, and
furthermore $m_{z^{b-a}}$ defines an isomorphism $A_a\cong A_{b}$ for
$\alpha \le a \le b $. In
particular, for $k>0$, $m_{z^k}: A_{\alpha -k} \rightarrow
A_\alpha $ is an injection, and $m_{z^k}:A_\alpha  \rightarrow A_{\alpha +k}$ is an
isomorphism onto.
\item\label{is5} Let $\dim_{\sfÊk}R'/J=s$. The subscheme $\Z=\Proj (A)$ of $\mathbb P^n$ has degree
$s$, and $J\supset {M'}^s$. Furthermore, the regularity $\sigma
(\Z)$ satisfies
$\sigma (\Z)=\tau (\Z)+1\le
\alpha +1\le s$.
\item\label{is6} Let $
\mathfrak{L}\subset \Gamma$ be an inverse system satisfying ${{\sf k}}_{DP}[Z]\subset \mathfrak{L}\subset
\Gamma_{\le
\alpha}\cdot {{\sf k}}_{DP}[Z]$, and let $\mathfrak{I}=\Ann(\mathfrak{L})$.  Then, letting
$\mathfrak{J}=(\mathfrak{I})_{(z=1)}\subset R'$, and $\mathfrak{L}'=(\mathfrak{L}_b)_{Z=1}, b\ge \alpha$ we have that $\mathfrak{L}'=(\mathfrak{J})^{-1}$, $\mathfrak{J}$ defines a scheme
$\Z$ concentrated at the origin with $\mathfrak{I}=I_\Z, \mathfrak{L}=L_\Z$, $\alpha(\Z)\le
\alpha$, and
$\mathfrak{L}=
\Homog(\mathfrak{L}',Z)$.
\end{enumerate}
\end{proposition}
\proof That $I=I_\Z$ and is saturated is well-known, since the primary
decomposition of an ideal carries over to its
homogenization (see \S VII.5 Theorem 17 of \cite{ZS}). The next statements of \eqref{is1} are standard, since $z$ is
a non-zero divisor in $R/I_\Z$. That $J=(J_{\le \alpha})+(M')^{\alpha +1}$ implies the last statement of
\eqref{is1}.
\par
That  $L=L_\Z$, and ${{\sf k}}_{DP}[Z]\subset L\subset \Gamma_{\le
\alpha}\cdot {{\sf k}}_{DP}[Z]$ in \eqref{is3} follow from Lemma
\ref{homoginvsys} and \eqref{is1}. That $L_{a} = z^{b-a}\circ L_b$
follows from $z$ being a non-zero divisor in $R/I_\Z$, and Lemma
\ref{sat}.  Lemma \ref{closerp} and an easy verification implies
that the two maps given are inverse isomorphisms when $a,b\ge
\alpha$. For any $f\in J^{-1}$, $\deg\ f\le \alpha $, so if $a\ge
\alpha $ we have $f[b]=(f[a])[b]$ and this implies the last
statement of \eqref{is3}. The statements of (\ref{is4}) follow from and are weaker
than those of (\ref{is1}) or (\ref{is3}); they are about the
quotient algebra $A$, rather than the ideal $I$ or inverse system
$L$. Now (\ref{is3}) and (\ref{is4}) imply the key inequality $\tau
(\Z)\le \alpha (\Z)$ of (\ref{is5}) since $\tau(\Z)=\min\{i\mid
\dim_{\sfÊk}(R/I_\Z)_i=s\}$. That $J\supset {M'}^s$ is well known; any
monomial of ${M'}^s$ has a length-$(s+1)$ chain of monomials that
divide it, so some linear combination of elements of the chain must
be in $J$ --- since the degree of $\Z$ is only $s$ --- implying the
monomial itself is in $J$.
\par
Note that the main condition of  \eqref{is6} is that of
\eqref{singpt} with $L_p=Z$; this is the condition for
$\mathfrak{I}$ to define a zero-dimensional scheme $\Z$ concentrated
at $p_0$, so \eqref{is6} would follow from the standard fact,
$J=(I_\Z)_{z=1}$ defines the portion of $\Z$ in $\mathbb A^n: z=1$,
and \eqref{is3}, provided we show that $\mathfrak{L}'=J^{-1}$.
Directly, we have $1\subset \mathfrak{L}'\subset \Gamma '_{\le
\alpha}$; thus, identifying $x$ and $y$ variables (since we have
taken $z=1$) and letting $\mathfrak{J}'=\Ann(\mathfrak{L}')\subset
R'$, we have $(x_1,\ldots ,x_n)\supset \mathfrak{J}' \supset
(x_1,\ldots ,x_n)^{\alpha +1}$. Also, we have $\dim_{\sfÊk}
\mathfrak{L}'=\dim_{\sfÊk} \mathfrak{L}_b=H(R/I_\Z)_b=s$, with
$s=\deg(\Z)$, as there is no kernel in dehomogenizing from a vector
subspace of $\Gamma _b$. Clearly $\mathfrak{L}'$ is independent of
the choice of $b\ge \alpha$ by \eqref{is3}. Taking $b=2\alpha$ and
using equation \eqref{vanish} of Lemma \ref{homogann} we can see
that $\mathfrak{J}\subset \mathfrak{J}'(Y)$, but we have
$\dim_{\sfÊk}({{\sf k}}[x_1,\ldots ,x_n]/\mathfrak{J}')=\dim_{\sfÊk}(R/\mathfrak{J})=s$,
implying $\mathfrak{J}=\mathfrak{J}'$. It likewise follows from the
equality of dimensions that $\mathfrak{L}'=\mathfrak{J}^{-1}$. This
completes the proof of \eqref{is6}, and of Proposition
\ref{AuxMainL}.
\endproof
\begin{remark}\label{strange}{\it Homogenized component $L_i$ is not determined by $L'_i$.}
We note here a perhaps surprising property of the homogenized
inverse system $L=\Homog (J^{-1},Z)$, where $ J^{-1}\subset \Gamma
'$ is the inverse system of an ideal $J\subset R'$ defining a
zero-dimensional scheme $\Z$. Namely, $F\in L_i, i<\alpha (\Z)$,
does not imply that there is a (possibly nonhomogenous) element
$f\in J^{-1}$ of degree $i$ such that $F=f[i]$. There are also
elements of $L_i$ arising from homogenizing to degree $i$ those
elements of $J^{-1}$ having higher degree. See Example
\ref{easyhomog} below, where $X_1^{[2]}\in L_2, X_1^{[2]}=z\circ
(X_1^{[2]}Z-X_1X^{[2]})$, but is not a homogenization of an element
of $J^{-1}_{\le 2}$. Likewise, as mentioned earlier, $X_1^{[2]}\in
L_2$ does \emph{not} imply $\Homog(X_1^{[2]},Z,3)=X_1^{[2]}Z\in
L_3$; rather the corresponding element of $L_3$ is
$X_1^{[2]}Z-X_1X_2^{[2]}$. However, if $i\ge \alpha (\Z )$, then
$F\in L_i $ and if $j\ge i$, $ F[j] \in L_j$ by Proposition
\ref{AuxMainL}(\ref{is3}). For similar reasons, the condition $b\ge
\alpha$ in Proposition \ref{AuxMainL}(v) cannot be removed, and we
may have $((L_\Z)_a)_{Z=1}\not\subset ((L_\Z)_b)_{Z=1}$ when $a<b$.
(See Example \ref{easyhomog} below).
\par Note that $\sigma (\Z)$ may be rather less than $\alpha +1$,
the upper bound of (\ref{is5}), and is almost always less than
$\alpha +1$ when the defining ideal of $\Z$ in $R'$ is
non-homogeneous. (See Examples
\ref{findhomogid},\,\ref{regdeg},\,\ref{noncloserp}). \par If we
write $\alpha(\Z)=\sigma(\Z)+k(\Z)$, it is not clear how to bound above
$k(\Z)$. The examples where $\Z$ is defined locally by a
general enough compressed Gorenstein ideal of $\mathcal O_p$, in the
sequel article \cite{ChoI2} show that there is no constant upper
bound. On the other hand, these examples satisfy  $k(\Z)\le \sigma
(\Z)$, suggesting that the latter bound might be valid for $\Z$
supported at a single point. \par For any zero-dimensional scheme
$\Z$,  Proposition 1.13 shows that the inverse system $L$ of $I_\Z$
is determined by $L_{\sigma}$; thus $L_i=(L_{\sigma}:R_{i-\sigma})$
if $i\ge \sigma$, and $L_i=R_{\sigma -i}\circ L_\sigma$ if $i\le
\sigma$. However, when both $\alpha ,i> \sigma$, $L_i$ may not be
obtained by simply raising the $Z$-power of elements of  $L_\sigma$
even when $\Z$ is concentrated at $p_0$ (see Example
\ref{noncloserp}).
\end{remark}

Below we set $Z^{[u]}=0$ if $u<0$.
\begin{example}\label{regdeg} {\it $L_{\tau}$ may not determine $L$.} Let $R={{\sf k}}[x_1,x_2,x_3],p=p_0=(0:0:1),
\mathcal I_p = (y_1y_2,y_1^2-y_2^3), f'=(Y_1^{[2]}+Y_2^{[3]})$; then
$\mathcal I_p\supset (y_1y_2,y_1^3,y_2^4)$ and $H(R_{p}/\mathcal I_{p})=(1,2,1,1)$, and $\alpha (\Z)=3$.
The homogenization
$I_\Z=(x_1x_2,x_1^2z-x_2^3,x_1^3,x_2^4)$, and $H(R/I_\Z)=(1,3,5,5,\ldots )$, so $\tau (\Z )=2,
\sigma (\Z )=3$. The inverse system $L={I_\Z}^{-1}$ satisfies, by Lemma \ref{homoginvsys}
\begin{align*}
L_j= \,&\langle X_1^{[2]}Z^{[j-2]}+X_2^{[3]}Z^{[j-3]},X_2^2Z^{[j-2]},X_2Z^{[j-1]},X_1Z^{[j-1]},Z^{[j]}\rangle\\
=\,&R_3\circ \left( \Homog (f',Z,j+3)\right) = R_3\circ \left( X_1^{[2]}Z^{[j+1]}+X_2^{[3]}Z^{[j]}\right)\\
=\,&\Homog (V',Z,j),\text{ where } V'= R'\circ f'=\langle
f',Y_2^{[2]}Y_2,Y_1,1\rangle .
\end{align*}
Note that $L_\sigma = L_3$ determines $L$, but the space $L_\tau =
L_2$ does not. This corresponds to $(I_\Z )_\sigma$ determining
$I_\Z$ (see Theorem \ref{zero-dim}(ii)). Also, since $\Delta
H(R/I_\Z)=(1,2,2,0)$, which is not symmetric, $\Z$ is not
arithmetically Gorenstein; however, $\Z$ is locally Gorenstein and
has a single point of support.
\end{example}
\begin{example}\label{CayleyB} More generally, with $R,p$ as above, let $f' =(Y_1^{[2]}+Y_2^{[j]}),
j\ge 3$; then $\mathcal I_p = (y_1y_2,y_1^2-y_2^j)$ and
$H(R_p/\mathcal I_p)=(1,2,1,\dots ,1_j)$, with $j-1$ ones at the
end, determining a zero-dimensionl scheme $\Z$ at $p$ of degree
$j+2$, for which $\alpha (\Z)=j$. The homogenization
$I_\Z=(x_1x_2,x_1^2z^{j-2}-x_2^j,x_1^3)$, so $H=(1,3,5,6,\ldots
,j+2,j+2, \ldots )$, and $\Delta H(R/I_\Z)=(1,2,2,1,\ldots ,1,0 )$,
with $j-3$ ones at the end, so $\tau (\Z )=j-1, \sigma (\Z)=j$. The
inverse system $L=I^{-1}$ is determined by
\begin{equation*} L_{\sigma}=\langle X_1^{[2]}\cdot Z^{[j-2]}+X_2^{[j]},X_2^{[j-1]}Z,\ldots
,X_2Z^{[j-1]},X_1Z^{[j-1]},Z^{[j]}\rangle ,
\end{equation*}
in the same sense as Example \ref{regdeg}, but is not so determined by $L_{\tau}$. \par
 When $j=3$ or $j\ge 5$ then
$\Z$ is not arithmetically Gorenstein, since $\Delta H(R/I_\Z)$ is
not symmetric. When $j=4$ then $H_\Z=(1,3,5,6,6,\ldots ), \Delta
H=(1,2,2,1)$, and $\Z$ is arithmetically Gorenstein if and only if
it satisfies the Cayley-Bacharach property that $H(R/I_{\Z
'})_{\tau-1}=H(R/I_{\Z})_{\tau-1}$ for any subscheme $\Z' \subset
\Z$ of degree $s-1$ (see \cite{Kr},\,\cite[Theorem 4.1.10]{Mig}).
Here $\tau (\Z) =3,\sigma (\Z)=4$, and the local inverse system is
$L'=\langle
1,Y_1,Y_2,Y_2^{[2]},Y_2^{[3]},Y_1^{[2]}+Y_2^{[4]}\rangle$. The only
$R'$-closed system $L''$ of codimension one is $L''=\langle
1,Y_1,Y_2,Y_2^{[2]},Y_2^{[3]}\rangle$. Since $L''$, and thus
$J''=\Ann (L'') \subset R''$ defining the subscheme $\Z ''$ is
graded, Proposition \ref{aGor} implies that $H(R/I_{\Z
''})=(1,3,4,5,5,\ldots )$, the sum function of
$H(R'/J'')=(1,2,1,1)$. Thus $\Z$ does not satisfy the
Cayley-Bacharach condition, and is not arithmetically Gorenstein.
\end{example}
\begin{example}\label{easyhomog} {\it Component $L_2$ not the homogenization of $L'_2$.}
If r=3, $R={{\sf k}}[x_1,x_2,x_3]$, $\Gamma '={{\sf k}}[ Y_1,Y_2] , f=Y_1^{[2]}-Y_1Y_2^{[2]}\in \Gamma ' $, then
$I'=(y_1^2+y_1y_2^2,y_2^3)$, of Hilbert function
$H(R'/I')=(1,2,2,1)$, and $\alpha (\Z)=3$. The related homogeneous ideal $I$ in $R$ determining
the degree-6 scheme $\Z$
concentrated at $p_0=(0:0:1)$ in $\mathbb P^2$ is
\[
     I=(x_1^2z+x_1x_2^2,x_2^3,x_1^3,x_1^2x_2),
\]
of Hilbert function $H_\Z= (1,3,6,6,\ldots )$, so $\tau (\Z )=2,\sigma (\Z )=3$.
 Here the
homogenization of $f$ to degree $\alpha$ is
 $G=f[3]=X_1^{[2]}Z-X_1X_2^{[2]}$. By Lemma \ref{homoginvsys}, letting
$L=L_\Z={I_{\Z}}^{-1}$, we have that $L$ is simply determined by the actions of the pair $(z,Z)=(x_3,X_3)$
on
$L_{\alpha}$, which satisfies
\begin{align} L_{3}=R_3\circ F[6]&=
R_3\circ
(X_1^{[2]}Z^{[4]}-X_1X_2^{[2]}Z^{[3]}) \notag\\
&=
 \langle G,Z^{[3]},X_1Z^{[2]},X_1X_2Z,X_2^{[2]}Z,X_2Z^{[2]}\rangle .\notag
\end{align}
 Likewise, $L_{2}=R_1\circ {L}_{3}=\langle X_1^{[2]},X_1Z,X_1X_2,X_2^{[2]},X_2Z,Z^{[2]}\rangle \subset \Gamma
_2$. Note that
$L_2$ contains
 $X_1^{[2]}$, which is the partial of $G=f[3]$ with respect to $z$, but is not the homogenization of an
element of $L'_{\le 2}={{I'}^{-1}}_{\le 2}$, as $L'=\langle
f,Y_2^{[2]},Y_1Y_2,Y_2,Y_1,1\rangle$.
\end{example}
\begin{example} When $R',p$ are as above, and $I'=(y_1^2,y_2^3), f' = Y_1Y_2^{[2]}$, the local Hilbert
function is $H(R'/I')=(1,2,2,1)$, $\alpha (\Z)=3$, then
$I=I_\Z=(x_1^2,x_2^3)$, $H(R/I_\Z)=(1,3,5,6,6,\ldots )$, so $\sigma
(\Z )=4$, $\tau (\Z )=3=\alpha (\Z)$, and $L=I^{-1}$ is determined
by 
\begin{equation*}L_\tau = \langle
X_1X_2^{[2]},ZX_2^{[2]},Z^{[2]}X_2,ZX_1X_2,Z^{[2]}X_1,Z^{[3]}
\rangle ,
\end{equation*}
in the stronger sense that $L_j=R_{\tau}\circ (L_\tau
\cdot _{rp}Z^j)$ if $j\ge \tau$, and $L_j=R_{\tau -j}\circ L_{\tau}$
when $j\le \tau$. This example and Example \ref{regdeg} above
illustrate that $L$ must be determined by $L_{\sigma}$, but $L$ is
also determined by $L_{\tau}$ if $I$ is generated in degrees less or
equal to $\tau$. The next example shows that this
determination by $L_\tau$ (or by $L_\sigma)$ is usually in a
$\it{weaker}$ sense than here.
\end{example}

\begin{example}\label{noncloserp} {\it How does $L_\sigma$ determine $L$?} We choose a curvilinear ideal (one with $\mathcal R_p/
\mathcal I_p\cong {\sf k}[y]/y^n]$)
$\mathcal I_p=(y_1+y_2^2+y_2^3+y_2^4, y_2^5)\subset
\mathcal O_p$, of local Hilbert function $H(\mathcal O_p/\mathcal I_p)=(1,1,1,1,1)$. Using the computer
algebra program {\sc{macaulay}} \cite{BSE} we calculated its homogenization as
$I_\Z=(x_1z+x_2^2-x_1x_2,x_1z^2+x_2^2z+x_1^2z+x_2^3, x_1^3)$, of Hilbert function
$H_\Z=(1,3,5,5,\ldots  )$, so $\sigma (\Z)=3 < \alpha (\Z)=4$.  A local dual generator $f'\in
L'=(I')^{-1}$ is $f'=Y_2^{[4]}-Y_1Y_2^{[2]}+Y_1^{[2]}-Y_1Y_2-Y_2^{[2]}$. Since $\mathcal I_p$ is not homogeneous, its dual generator is determined only up to multiple $\lambda \circ f'$ by a unit
$\lambda$  of $R'$.  We have $ L'=\langle f',Y_2^{[3]}-Y_1Y_2-Y_1,Y_2^{[2]}-Y_1,Y_2,1 \rangle$, and,
letting
$F=\Homog(f',Z,4)$, we have
\begin{equation*}
L_4=\langle
F,X_2^{[3]}Z-X_1X_2Z^{[2]}-X_1Z^{[3]},X_2^{[2]}Z^{[2]}-X_1Z^{[3]},X_2Z^{[3]},Z^{[4]}\rangle
.
\end{equation*}
Here $L_3$ contains
$f'[3]=\Homog(f',Z,3)=-X_1X_2^{[2]}+X_1^{[2]}Z-X_1X_2Z-X_2^{[2]}Z$.
Note that $Z\cdot_{rp}f'[3]\notin L_4$. Here $L_4=(L_3:R_1)$
so $L_3$ determines $L_4$ (Proposition \ref{satinvsys} Equation
\eqref{sigmadetermines}), but $L_4\not= L_3\cdot_{rp}Z$ --- unlike the
simple relation $L_{i+1}=L_i\cdot_{rp}Z$ when $i\ge \alpha (\Z)$.
Also $L_4\not= R_{3}\circ (f'[3]\cdot_{rp} Z^{[4]})$. Rather, by
Lemma \ref{homoginvsys}, we need to use  $f'[\alpha]$: thus,
$L_j=R_{4}\circ (f'[4]\cdot_{rp} Z^{[j]})$.
\end{example}

We now return to one of our themes, deciding when a locally
Gorenstein zero-dimensional scheme is arithmetically Gorenstein,
with the aid of the inverse system.
\begin{proposition}\label{aGor}{\sc Cones that are aG.}
Suppose that $\Z\subset \mathbb P^n$ is a a degree-$s$
zero-dimensional locally Gorenstein subscheme, concentrated at a
single point $p\in \mathbb P^n$. Suppose further that $\Z$ is
defined by a homogeneous ideal $\mathcal I_p$ of the local ring
$\mathcal O_p$ at $p$ (we say that $\Z$ is $\it{conic}$, see
\cite[Lemma 6.1]{IK}).  Then $\Z$ is arithmetically Gorenstein.
\end{proposition}
\begin{proof} We may suppose that the point is $p=(0:\ldots :0:1)$, and that $\mathcal{I}_p$ is defined by
$I'\subset R'=k[y_1,\ldots ,y_n]$. Then, letting $z=x_{n+1}$ we have
$(I_\Z)_i=\bigoplus_0^i z^{i-a}\cdot I'$, whence it follows that
$R/(I_\Z,z)\cong R'/I'$, implying that $R/I_\Z$ is Gorenstein, and
that $\Delta H_\Z=H(R'/I')$.
\end{proof}
\begin{remark} The converse of Proposition \ref{aGor} is false in $\mathbb P^2$. The ideal
$\mathcal I_p=(y_1^2,y_1y_2-y_2^3)\subset \mathcal O_p, p = (0:0:1)$
has local Hilbert function $H(\mathcal O_p/\mathcal
I_p)=(1,2,1,1,1)$, and is not homogeneous. The homogenized ideal
$I_\Z\subset R={{\sf k}}[x_1,x_2,x_3]$ satisfies
$I_\Z=(x_1^2,x_1x_2z-x_2^3,x_2^5)$, of Hilbert function
$H_\Z=(1,3,5,6,6,\ldots )$. Here $z$ is a non-zero divisor for
$R/I_\Z$, and the quotient $R/(z,I_\Z) \cong
{{\sf k}}[y_1,y_2]/(y_1^2,y_2^3)$, so $\Z$ is arithmetically Gorenstein.
\end{remark}
D. Bernstein and the second author \cite{BeI}, M. Boij
and D. Laksov \cite{BjL} gave examples of graded Gorenstein Artin algebras having
non-unimodal  Hilbert functions. Later M. Boij gave examples of such algebras whose
Hilbert functions have arbitrarily many maxima \cite{Bj2}. It follows from
Proposition \ref{aGor} that these examples lead to "thick points" that are
arithmetically Gorenstein schemes $\Z$ in $\mathbb P^n$, with non-unimodal
$h$-vector $\Delta H_\Z$. We give the first such
example of lowest embedding dimension, then socle degree.
\begin{corollary}\label{aGor2} There is an arithmetically Gorenstein, $\it{conic}$, zero-dimensional scheme $\Z$
concentrated at a single point $p\in \mathbb P^5$ with $\Delta H_\Z$ non-unimodal, and
satisfying
\begin{equation} \Delta H_\Z = (1,5,12,22,35,51,70,91,90,91, \ldots ,5,1).
\end{equation}
The homogeneous form $F'\in \Gamma '={{\sf k}}_{DP}[U,V,W,X,Y]$ defining $\mathcal I_p=\Ann(F')$ is
$F'=Uf+Vg$ where $f,g$ are general enough degree-15 forms in $W,X,Y$.
\end{corollary}
\noindent $\mathbf{Remark.}$ When $\Z$ is concentrated at a single
point, defined by $\mathcal I_p\subset \mathcal O_p$, and the local
Hilbert function $H(\mathcal O_p/\mathcal I_p)$ is symmetric, then
it is known that $\mathcal I_p$ is Gorenstein if and only if the
associated graded ideal $\mathcal I^*_p$ is also Gorenstein
(\cite[Proposition 1.9]{Wa},\,\cite[Proposition 1.7]{I3}). It is not
hard to show that when $H(\mathcal O_p/\mathcal I_p)$ is symmetric,
$\Z$ is arithmetically Gorenstein if and only if $\Z$ is Gorenstein
and $\mathcal I_p=\mathcal I_p^\ast$.

\subsection{Schemes concentrated at an arbitrary point of $\mathbb P^n$.}\label{arbitrarypoint}
 We now extend the results of the previous subsection to any point
$p\in \mathbb A^n\subset \mathbb P^n$. We translate the point
to the origin using the action of the linear group, and use the adjoint representation on $\Gamma$, to
translate the inverse system.
Following F.~H.~S. Macaulay, we take
$\widehat{\Gamma '}= {{\sf k}}_{DP}\{\{Y_1,\ldots ,Y_n\}\}$, the divided power analog of the power series ring,
upon which the polynomial ring
$R'={{\sf k}}[y_1,\ldots ,y_n]$ acts by contraction, as before.  The rings $R, \Gamma$, remain the same, but a
finite inverse system will be an $R'$-submodule of $\widehat{\Gamma '}$ having finite dimension as
${{\sf k}}$-vector space. When $p=(a_1:\ldots : a_n:1)\in \mathbb P^n$, we will
sometimes use
$q=(a_1,\ldots ,a_n)$ to specify the point $q=(a_1,\dots ,a_n)$ of $\mathbb A^n$ without regard to
$\mathbb P^n$. We let
\begin{equation}\label{fp}
f_q =(1-\sum a_iY_i)^{-1}=1+\sum_{k\ge 1} (\sum a_iY_i)^{[k]}=\sum_k \sum_{U\mid |U|=k}
a^UY^{[U]}.
\end{equation}
Here $f_q$ is the divided power analog of the exponential series $F_q=\exp  ({\sum a_iY_i})$ in
the usual power series ring ${\widehat{\mathcal R'}}$. We will sometimes use $f_p,F_p$ to denote the
corresponding $f_q,F_q$.
\begin{lemma}\label{Macaffinedual}\cite[\S 64, p. 73]{Mac}\quad {\sc  Inverse systems for ideals with
support an arbitrary point.} The finite inverse system $J\subset
\widehat{\Gamma '}$ (respectively, $J'\subset {\widehat{\mathcal
R'}}$ in the differentiation action of $R'$ on ${\widehat{\mathcal
R'}}$) is the inverse system of an ideal of $R'$ with support the
point $p=(a_1,\ldots , a_n)\in \mathbb A^n$ if and only if there
exists an integer $N$ such that
\begin{equation}\label{expdp}
 f_q\subset J\subset \Gamma '_{\le N}\cdot
f_q \subset \widehat{\Gamma '}
\end{equation}
or, respectively,
\begin{equation}\label{exp}
\exp (\sum a_iY_i)\subset J' \subset \mathcal R '_{\le N}\cdot
\exp (\sum a_iY_i)\subset \mathcal {\widehat{\mathcal R'}}.
\end{equation}
\end{lemma}
\begin{proof}[Proof Outline] Here \eqref{expdp} is the divided power analog of \eqref{exp}. To show
\eqref{exp}, note
first that $m_q=(y_1-a_1,\ldots ,y_n-a_n) \subset R'$ annihilates the one-dimensional
vector space $\exp (\sum a_iY_i) \in \mathcal {\widehat{\mathcal R'}}$, since $y_i$ acting by
differentiation on this series is the same as multiplication by $a_i$. Likewise,
$(m_q^{N+1})^{\perp}\subset
\mathcal R '_{\le N}\cdot \exp
(\sum a_iY_i)$ is immediate, and a dimension check shows \eqref{exp}.
\end{proof}
\par
The following lemma is a consequence of \cite[\S 64,66]{Mac} (see
Remark \ref{compare} below). Given $q=(a_1,\ldots , a_n)\in \mathbb
A^n$ and an ideal $J$ of $R'$ concentrated at the origin, we denote
by $T_q(J)$ the translated ideal $T_q(J)=(h(y_1-a_1,\ldots
,y_n-a_n)\mid h\in J)$. Clearly, $T_q(J)$ is concentrated at $q$.
\begin{lemma}\label{transforma}{\sc Macaulay's Comparison Lemma: change of origin in $\mathbb
A^n$.}
\begin{enumerate}[\rm(i)]
\item\label{transformai} If $h'\in R',\  f'$ in ${\widehat{\Gamma '}}$, then $h'(y_1-a_1,\ldots
,y_n-a_n)\circ (f'\cdot f_q)=(h'(y_1,\ldots ,y_n)\circ f')\cdot f_q$.
\item\label{transformaii} If $L''\subset {\widehat{\Gamma '}}$ is the inverse system of an ideal $J$
of $R'$ that is concentrated at the origin, then $L''\cdot f_q$ is the inverse system of $T_q(J)$.
\item\label{transformaiii} Let $\mathcal I_q\subset \mathcal O_q$ be the ideal $J'\cdot \mathcal
O_q$, where $J'\subset R'$ has support $q$. Then the inverse system $L'=(J')^{-1}\subset
\widehat{\Gamma '}$ has the form $L'=L''\cdot f_q$, where $L''\subset \Gamma '$ is the inverse system
of the ideal $J=(T_q)^{-1}(J')$, concentrated at the origin. \par
\item\label{transformaiv} The $R'$ submodules of $\widehat{\Gamma '}$ generated by $L'$ and by $L''$
in (iii) are isomorphic.
\item\label{transformav} The analogous statements to (i)-(iv) are true for
the partial differentiation action of $R'$ on the power series ring $\widehat{\mathcal R'}$, with $f_q$
replaced by
$F_q=\exp(\sum a_iY_i)$.
 \end{enumerate}
\begin{proof} The part (\ref{transformaii}) is implied by (\ref{transformai}). It suffices to show
part
(\ref{transformai}) for monomials $h'$; by induction on degree,
we may suppose that $h'=y_i$ (since the statement is obvious for $h'=$ constant). Then
we have by additivity of contraction, and Lemma \ref{lindiff},
\begin{align*} (y_i-a_i)\circ (f'\cdot f_q) &= y_i\circ (f'\cdot f_q)-a_i\circ (f'\cdot f_q)\\
&=\left( y_i\circ f' \cdot f_q +f'\cdot y_i \circ f_q\right)-a_if'\cdot f_q\\
&=y_i\circ f' \cdot f_q+f'\cdot a_if_q -a_if'\cdot f_q \\
&= (y_i\circ f') \cdot f_q,
\end{align*}
as claimed. This completes the proof of (\ref{transformaii}). Any ideal $J'$ of $R'$ concentrated
at $p$ satisfies, $J'=T_q(J), J=(T_q)^{-1}(J')$, so (\ref{transformaii}) implies
(\ref{transformaiii}). Also, (\ref{transformaiv}) is immediate.
\end{proof}
\end{lemma}
\begin{remark}\label{compare} Macaulay \cite[\S 64,
p.72]{Mac} describes the transform in Lemma~\ref{transforma}, as follows. Let $F=\sum a_{p_1,\ldots
,p_n}y_1^{p_1}\cdots y_n^{p_n}$ be a polynomial, let $E=\sum
c_1^{p_1}\cdots c_n^{p_n}\left( y_1^{p_1}\cdots y_n^{p_n}\right)
^{-1}$ be a modular equation, and consider the new origin
$(-a_1,-a_2,\ldots ,-a_n)$. The transformed polynomial is $F'=\sum
a_{p_1,\ldots ,p_n}(y_1-a_1)^{p_1}\cdots (y_n-a_n)^{p_n}$, and the
transformed modular equation is 
\begin{equation*}
E'=\sum (c_1+a_1)^{p_1}\cdots
(c_n+a_n)^{p_n}\left( y_1^{p_1}\cdots y_n^{p_n}\right) ^{-1}.
\end{equation*}
 Here
the coefficients $c$ are in symbolic notation: that is, after
expanding the expressions, $c_1^{p_1}\cdots c_n^{p_n}$ is to be put
equal to the coefficient $c_{p_1,\ldots ,p_n}$. For
$E=1$, then $E'=\sum a_1^{p_1}\cdots a_n^{p_n}\left( y_1^{p_1}\cdots
y_n^{p_n}\right) ^{-1}$, the inverse function of $(x_1-a_1,\ldots
,x_n-a_n)$. \par Macaulay translates a mutually
perpendicular polynomial/inverse system pair at the origin, to one
concentrated at the point $q=(a_1,\ldots ,a_n)$. We rewrite
Macaulay's formula for $E'$ using the multiindex $U=(u_1,\ldots
,u_n)$ where $U\le P$ means, $u_i\le p_i$ for each $i$, as follows:
\begin{align*}
E' &= \sum_{U,P\mid 0\le U\le P} c_{u_1,\ldots ,u_n}{{p_1}\choose{u_1}}\cdots
{{p_n}\choose{u_n}}a_1^{p_1-u_1}\cdots a_n^{p_n-u_n}\cdot \left( y_1^{p_1}\cdots
y_n^{p_n}\right) ^{-1} \text{ (Macaulay's notation) }\\
&=\sum_{U,P\mid 0\le U\le P} c_{u_1,\ldots ,u_n}{{p_1}\choose{u_1}}\cdots
{{p_n}\choose{u_n}}a_1^{p_1-u_1}\cdots a_n^{p_n-u_n}\cdot Y_1^{[p_1]}\cdots
Y_n^{[p_n]} \text{ (our notation)}\\
&=\sum_{U,P-U\mid 0\le U,0\le P-U} c_{u_1,\ldots ,u_n}\left( Y_1^{[u_1]}\cdots
Y_n^{[u_n]}\right) a_1^{p_1-u_1}\cdots a_n^{p_n-u_n}\cdot Y_1^{[p_1-u_1]}\cdots Y_n^{[p_n-u_n]}\\
&=\quad E\cdot f_q .\\
&\text{ The product in the last two steps is that of the divided power
ring
${\widehat{\Gamma '}}$}.
\end{align*}

Note that our Lemma
\ref{transforma}, Equation \eqref{transformai} when $h'\circ f'=0$, is equivalent to Macaulay's
formula, so Lemma \ref{transforma} \eqref{transformaii} is a consequence of Macaulay's formula for
changing the point of origin.
\end{remark}

Fix a  point $q=(a_1,\ldots ,a_n)\in \mathbb A^n\subset \mathbb P^n$ with projective
coordinates $p=(a_1:\ldots :a_n:1)$. Let $J' \subset R'$ be an ideal
supported at $q$, so that $(R'/J')\cong
\mathcal O_q/ \mathcal I_q, \mathcal I_q=J'\cdot \mathcal O_q$ defines an Artin quotient.
By Lemma
\ref{transforma} its inverse system $L'=(J')^{-1}\subset \widehat{\Gamma '}$ satisfies $L'=L''\cdot
f_q$, where $L''\subset \Gamma '$ is the inverse system of $J=(T_q)^{-1}(J')$. Recall that
$L_p=a_1X_1+\cdots +a_nX_n+Z$. Recall the homogenization $\Homog(L'',L_p,u)$ for inverse systems $L''\subset \Gamma'$ (Definition
\ref{defhomcoordp}).
\begin{theorem}\label{Homogp} {\sc Comparison Theorem.} Let
$I_\Z\subset R$ be the saturated ideal defining the scheme $\Z$
concentrated at the point $p\in \mathbb A^n\subset \mathbb P^n$, and let
$L_\Z=I_\Z^{-1}\subset \Gamma$ be its global inverse system. Let
$J'\subset R'$ be the ideal defining $\Z\subset \mathbb A^n$ and
$L'=(J')^{-1}\subset \widehat{\Gamma '}$ its affine inverse system.
Let $J=T_q^{-1}(J')$, and $L''=J^{-1}\subset \Gamma'$ its inverse
system.  Let $\alpha=\alpha (\Z)$  and suppose that $V''\subset
\Gamma '_{\le \alpha}$ generates $L''$ (so $L''=R'\circ V''$), and
set $V=\Homog(V'',L_p,\alpha)$.
\begin{enumerate}[\rm (i)]
\item Then the global inverse system $L_\Z$ satisfies
\begin{align}\label{Homogpeq} (L_\Z)_i = \Homog (L''_{\le \alpha},L_p,i)
          = R_\alpha \circ(V\cdot_{rp} L_p^{[i]}).
\end{align}
\item  Furthermore, let $g$ denote the linear transformation of $R$
taking $p$ to the origin, and $g^\ast$ the contragradient transform
on $\Gamma$, and set $\Z_o=\Proj (R/g(I_\Z))$,
$L_o=(I_{\Z_o})^{-1}$. Then we have
\begin{equation}
L=g^{\ast}\circ L_o.
\end{equation}
The $R$-module $L_\Z$ is isomorphic to $L_o$. Also, if $\Z$ is any
zero-dimensional scheme concentrated at $p$, then $L_\Z=(I_\Z)^{-1}
$ satisfies the first part of \eqref{Homogpeq}, for a suitable
$L''\subset \Gamma '_{\le \alpha}$, where $\alpha = \alpha (\Z)$;
conversely, if an inverse system $L$ satisfies $L_i=\Homog(L''_{\le
\alpha},L_p,i)$, then $L=L_\Z$ for a zero-dimensional scheme $\Z$
concentrated at $p$.
\end{enumerate}
\end{theorem}
\begin{proof} The linear transformation of $R$ taking $p_0=(0:\ldots :0:1)$ to
$p=(a_1:\ldots :a_n:1)$ is $g(x_1)=x_1'=x_1-a_1z,\dots
,g(x_n)=x_n'=x_n-a_nz, g(z)=z'=z$. The contragradient transform of
$\Gamma=R^\vee$ satisfies $g^\ast (v^{\ast})(v)=v^{\ast}(g^{-1}v)$,
and is readily seen to be $g^\ast (X_i)=X_i, 1\le i \le n; \text {
and } g^{\ast}(Z)=L_p$. The contraction map is equivariant (see, for
example \cite{Mac}, or \cite[Prop. A3]{IK}), so for $h\in R,F\in
\Gamma$, $g^{\ast}(h\circ F)=g(h)\circ (g^\ast F)$. Thus,
\eqref{Homogpeq} follows from Lemma \ref{homoginvsys} and in
particular Equation \eqref{isgen}. The last statement follows from
Proposition \ref{AuxMainL} (\ref{is6}), similarly by translation to
$p$.
\end{proof}\par
\noindent $\mathbf{Remark.}$ We believe that equation
\eqref{Homogpeq} could also be approached directly from Lemma
\ref{transforma}, using the fact, homogenizing $v\cdot f_q, v\in
\Gamma '$ to a given degree, with respect to $Z$, is the same as
homogenizing $v$ with respect to $L_p$, since
$Z^{[j]}\cdot_{rp}\left( 1+\sum_{U\mid 1\le |U|\le
j}a^U(X/Z)^{[U]}\right) =L_p^{[j]}$.  By
Corollary \ref{macdual},   $\dim_{\sfÊk} V''=\type \,\mathcal O_\Z$ is the minimum possible dimension for $V''$.

\begin{example}\label{homogpex1} Let $\Z$ denote the degree-$4$ scheme concentrated at $p_1=(1:0:1)$,
determined by
$f'=(Y_1^{[2]}+Y_2^{[2]})\cdot f_{p_1}$. Then $I_{\Z}$
is the translation to $p_1$ of $(x_1x_2,x_1^2-x_2^2)$,
so $I_{\Z}=(x_1x_2-zx_2,x_1^2-2zx_2+z^2-x_2^2)$, of Hilbert function
$H_{\Z}=(1,3,4,4,\ldots )$, with $\tau(\Z)=\alpha(\Z)=2$. Here $L_3$ determines $L=(I_\Z)^{-1}$,
and satisfies, by Theorem \ref{Homogp},
\begin{align*}  L_3=&\Homog (V,X_1+Z,3), \text{ where } V =\langle
X_1^{[2]}+X_2^{[2]},X_1,X_2,1\rangle\\ =&\langle
3X_1^{[3]}+X_1^{[2]}Z+X_1X_2^{[2]}+X_2^{[2]}Z,\langle X_1,X_2\rangle\cdot
(X_1+Z)^{[2]},(X_1+Z)^{[3]}\rangle
\end{align*}
\end{example}
\begin{example}\label{homogpex2} Again, consider the point $p=(1:0:1)\in \mathbb P^2$ and the ideal
$\mathcal I_p\subset \mathcal O_p$ defined by $\mathcal I_p=\Ann (f'\cdot f_p)$
where
$ f'=Y_1^{[2]}+Y_2^{[3]}$: this ideal is the translation to
$p$ of the ideal found in Example \ref{regdeg}, concentrated at $p_0=(0:0:1)$.
We have
$\mathcal I_p=\left( (y_1-1)y_2,(y_1-1)^2-y_2^3,(y_1-1)^3,y_2^4\right)$, and its homogenization
in $R={{\sf k}}[x_1,x_2,z]$ is $I=\left( (x_1-z)x_2,(x_1-z)^2z-x_2^3,(x_1-z)^3,x_2^4\right)$, of Hilbert
function
$H_\Z=(1,3,5,5,5, \ldots )$, defining a scheme $\Z$ of regularity $\sigma (\Z )=3$. By Theorem~\ref{Homogp}
the
inverse system $L=L_\Z$ is determined by the ``generator'' element
$F=\Homog (f',L_p,3)=X_1^{[2]}\cdot L_p+X_2^{[3]}\in L$, $L_p=X_1+Z$: so
$L_i=R_3\circ G_{i+3}, G_{i+3}=  F\cdot_{rp}L_p^{[i]}$. Thus we have for $L_3$, which
determines
$L$,
\begin{align*}
L_3&=R_3\circ F\cdot_{rp} (X+Z)^{[3]}=R_3\circ \left( X_1^{[2]}\cdot (X_1+Z)^{[4]}+X_2^{[3]}\cdot
(X_1+Z)^{[3]}\right) \\ &=R_3\circ \left[ \left(
15X_1^{[6]}+10X_1^{[5]}Z+6X_1^{[4]}Z^{[2]}+3X_1^{[3]}Z^{[3]}+X_1^{[2]}Z^{[4]} \right) \right. +\\
 &\qquad \qquad \qquad \left. X_2^{[3]}\cdot \left( X_1^{[3]}+X_1^{[2]}Z+X_1Z^{[2]}+Z^{[3]}\right)
\right] .
\end{align*}
By the first part of \eqref{Homogpeq} in Theorem \ref{Homogp}, and Example \ref{regdeg} this is
\begin{equation*}
L_3=\Homog (V'',L_p,3),\text{ where } V''= R'\circ f'=\langle
f',Y_2^{[2]},Y_2,Y_1,1\rangle .
\end{equation*}
So $z^3\circ G_6=3X_1^{[3]}+X_1^{[2]}Z+X_2^{[3]}\in L_3$. Note the coefficient $3$ on the first
term; since $I\circ L=0$, we have $I\circ (z^3\circ G_6)=0$. Thus we have
\begin{align*}
(x_1-z)^3\circ (z^3\circ G_6)& =\left( x_1^3-3x_1^2z+3x_1z^2-z^3\right) \circ
(3X_1^{[3]}+X_1^{[2]}Z+X_2^{[3]})\\ &= x_1^3\circ
( 3X_1^{[3]})-3x_1^2z\circ (X_1^{[2]}Z)+0-0 \\
&= 0
\end{align*}
\end{example}

\begin{proposition}\label{satinvsysp} {\sc Inverse system of a scheme concentrated at a single point.}
Assume that $\cha {{\sf k}} = 0$, or $\cha {{\sf k}} >j$. Suppose $\Z \subset
\mathbb P$ is a degree-$s$ zero-dimensional scheme concentrated at
the point $p=(a_1:...:a_n:1)$, and regular in degree $\sigma$ with
$\alpha(\Z)=\alpha$ Then, $W \subset \Gamma$ is the inverse system
of $\Z$ if and only if each of the following equivalent conditions
holds:
\begin{enumerate}[\rm(i)]
\item\label{satinvsys(i)} $\exists \alpha\in \mathbb N\mid {{\sf k}}_{DP}[L_p]\subset W \subset \Gamma_{\le
\alpha}\cdot {{\sf k}}_{DP}[L_p]$;
\item
\begin{enumerate}[\rm(a)]\label{satinvsys(ii)}
\item\label{satinvsys(ii)(a)} $\dim_{\sfÊk} W_j=s\,\  \forall j\ge \tau=_{def}\sigma -1$ and \par
\item\label{satinvsys(ii)(b)}  $\forall (i,n)\mid n\ge \max\{i,\sigma \} ,\, R_{n-i}\circ
W_n = W_i$ \par
\item\label{satinvsys(ii)(c)} $W\subset \Gamma_{\le \alpha}\cdot
{{\sf k}}_{DP}[L_p]$;
\par
\end{enumerate}
\item\label{satinvsy(iii)}
\begin{enumerate}[\rm(a)]
\item\label{satinvsys(iii)(a)} With the conditions (iia) and (iib) above, and $\forall j,\,L_p^{[j]}\in W_j$, and \par
\item\label{satinvsys(iii)(b)} $\forall j,\, W_j=R_\alpha \circ (W_{\alpha}\cdot_{rp}L_p^{[j]})$.\par
\end{enumerate}

\end{enumerate}
\end{proposition}
\begin{proof} The condition (\ref{satinvsys(ii)(b)}) above implies
the corresponding condition of Proposition \ref{satinvsys}, so
(\ref{satinvsys(ii)(a)}),(\ref{satinvsys(ii)(b)}) are
equivalent to $I=\Ann (W)$ being the saturated ideal defining a
degree-$s$ zero-dimensional scheme $\Z\subset \mathbb P^n$. The
third condition (\ref{satinvsys(ii)(c)}) is that of Lemma
\ref{homogdefp}, and assures that $\Z$ has support the point $p$.
The specific bound $\alpha$ arises from the change of coordinates of
Lemma \ref{transforma} applied to the formulas $L_j=L_\alpha [j]$
and $L\subset \Gamma_{\le \alpha}\cdot {{\sf k}}[Z]$ -- note that ${{\sf k}}[Z]=
{{\sf k}}_{DP}[Z]$ of Proposition \ref{AuxMainL} (\ref{is3}). Thus, the
hypotheses on $\Z$ imply the first condition (\ref{satinvsys(ii)})
and conversely.  That the hypotheses imply (\ref{satinvsys(iii)(b)})
follows from Theorem \ref{Homogp}. Evidently, (\ref{satinvsys(iii)(b)})  implies
(\ref{satinvsys(ii)(c)}).
\end{proof}

\subsection{Schemes with finite support.}\label{schfinsup}
We now combine the results of previous sections, to determine the
inverse system of schemes concentrated at several points. We will
assume that coordinates are chosen so that any zero-dimensional
scheme $\Z$ considered lies entirely within the affine chart
$\mathbb A^n$ where $ x_{n+1}\not= 0$. Let $p_1,\ldots,p_k$ be
distinct points in $\mathbb P$ with $x_{n+1}\neq 0$ and $\Z(u)$ be a
degree-$s_u$ zero-dimensional scheme concentrated at $p_u$ with
local socle degree $\alpha(u)=\alpha(\Z(u))$, for $1\le u\le k$
(Definition \ref{socled}). Recall that we denote by
 $m_{p_u}\subset R$ the homogeneous ideal of
 $p_u=(a_{u_1}:...:a_{u_n}:1)$,
 for $1\le u\le k$.
Let $\Z=\cup_{u=1}^k \Z(u)$.
\begin{proposition} $I_\Z$ is the saturated ideal defining $\Z$
if and only if
\begin{equation}\label{int}
m_{p_1}\cap \cdots \cap m_{p_k}\supset I_\Z\supset m_{p_1}^{\alpha
(1)+1}\cap \cdots \cap m_{p_k}^{\alpha (k)+1}.
\end{equation}
The inverse system $L$ is that of such a scheme if and only if it is
saturated (Lemma \ref{sat}, equation \eqref{esat2}), and
\begin{equation}\label{sum2}
  \langle L_{p_1}^{[i]},\ldots ,L_{p_k}^{[i]} \rangle \subset L_i\subset
\langle \Gamma_{\alpha (1)}L_{p_1}^{[i-\alpha (1)]},\ldots
,\Gamma_{\alpha (k)}L_{p_k}^{[i-\alpha (k)]}\rangle
\end{equation}
\end{proposition}
\begin{proof} The condition \eqref{int} is the condition for the primary decomposition of $I_\Z$
to have $p_1,\ldots ,p_k$ as the associated points; the condition \eqref{sum2} is its translation
by \eqref{mpsannih} (see also Lemma \ref{homogdefp}).
\end{proof}

Further, let $I(1),...,I(k)$ be saturated ideals defining
$\Z(1),...,\Z(k)$ and let $L, L(1),...,L(k)$ be the global inverse
system of $I_{\Z},I(1),...,I(k)$ respectively. Recall that
$\tau(\Z)=\sigma(\Z)-1$.
\begin{theorem}{\sc Decomposition of the inverse system of a punctual scheme.}\label{decompIS}
With $\Z$, $\Z(u)$ and $\alpha(u)$ ($1\le u \le k$) as above, we
denote the regularity degree of $\Z$ by $\sigma$, and that of each
$\Z(u)$ by $\sigma (u)$ for $1\le u\le k$ and set $\Z'(u)=\Proj (
R/( m_{p_u}^{\alpha(u)+1}\cap (I(1)\cap \cdots \cap
\widehat{I(u)}\cap \cdots \cap I(k)))$. Then we have,
\begin{enumerate}[\rm(i)]
\item\label{sum} $L=L(1)+\cdots +L(k).$ \par
\item\label{directs} When $i\ge \sigma -1$, then $L_i=L(1)_i\bigoplus \cdots \bigoplus
L(k)_i ,$ and $I(1)_i,\ldots ,I(k)_i$ intersect properly in $R_i$.
\par
\item\label{summend} $L(u)_i\subset L_i\cap (\Gamma_{\alpha (u)} \cdot L_{p_u}^{[i-\alpha
(u)]})$, with equality for $i\ge \min\{ i\mid \dim( L_i\cap
(\Gamma_{\alpha (u)} \cdot L_{p_u}^{[i-\alpha (u)]}))=s_u\}$.
Certainly there is equality for $i\ge \tau(\Z'(u))$.  Also
$L(u)_i=R_{j-i}\circ L(u)_j$ if $i\le j$ and $j\ge \sigma (u)$.
\par
\end{enumerate}
\end{theorem}
\proof First, (\ref{sum}) follows from the exactness of the action
of $R_i$ on $\Gamma _i$; the perpendicular space in $\Gamma _i$ to
an intersecton $I(1)_i\cap \cdots \cap I(k)_i$ is the sum $L(1)_i+
\cdots + L(k)_i$. That the sum is direct when $i\ge \sigma -1$
arises from $H(R/I_\Z)_i=s=\sum_u s_u=\sum_u H(R/I_{\Z_u})_i$ when
$i \ge \sigma -1$ and this shows the first statement of
(\ref{directs}), which is equivalent by duality to the second. The
inclusion in (\ref{summend}) arises from the inclusion $I(u)\supset
I_\Z+ M(p_u)^{\alpha (u) +1}$ by duality, using Equation
\eqref{mpsannih}. When $i\ge \tau (\Z'(u))$ we have that
$(M(p_u)^{\alpha (u) +1})_i$ and $\left( I(1)\cap \cdots \cap
\widehat{I(u)}\cap \cdots \cap I(k)\right)_i $ intersect properly in
$R_i$ by (\ref{directs}), whence it is not hard to show
$I(u)_i=(I_\Z+M(p_u)^{\alpha (u)+1})_i$. Here is a proof: let
$L'(u)=L(1)+ \cdots + \widehat{L(u)} + \cdots + L(k)$. Then
\begin{equation} L_i\cap (\Gamma_{\alpha (u)} \cdot L_{p_u}^{[i-\alpha
(u)]})= (L'(u)+L(u))_i \cap (\Gamma_{\alpha (u)} \cdot
L_{p_u}^{([i-\alpha (u)]}) = (L(u)_i+K_i),
\end{equation}
where, when $i\ge \tau$ we may assume $K_i\subset L'(u)_i$, since
the sum $L'(u)_i+ L(u)_i$ is then direct; but when $i\ge
\tau(\Z'(u))$ we must have $K_i=0$.  The last statement follows from
Lemma \ref{sat}.
\endproof
\begin{remark} Let the degree $s$ of a zero-dimensional scheme $\Z\subset \mathbb P^n$ be given, also an
upper bound $N\ge \sigma(\Z)$ on the regularity degree. Suppose
that we can calculate the inverse system $(L_\Z)_i$ in any degree. We
may find the primary decomposition of $I_\Z$ as follows:  first,
determine the points $p_u$ of support by testing which powers
$L_{p_u}^{[N]}\in (L_\Z)_N$. Following Theorem \ref{decompIS}
\eqref{summend}, then choose $i\ge s+\dim_{\sfÊk}( R/m^{s+1})$ and form
the intersection $L(u)_i=(L_\Z)_i\cap (\Gamma_{s} \cdot
L_{p_u}^{[i-s]})$, from which $I(u)$ can be determined (see Example
\ref{exlink}). However, this may require working in a high degree.
 Theorem \ref{decompIS} \eqref{directs} shows that $(L_\Z)_N$ contains $L(u)_N$ as a direct summend: can we determine $L(u)_N$ from $(L_\Z)_N$?
\end{remark}
\begin{remark}\label{critagor}{\it Determining when $\Z$ is arithmetically Gorenstein.}
A Gorenstein Artin local algebra has a unique minimum length ideal,
its socle, of dimension one as ${{\sf k}}$-vector space. Thus if $\Z$ is a
zero-dimensional locally Gorenstein scheme in $\mathbb P^n$, each
irreducible component $\Z_i$ has a unique proper subscheme of degree
one less than $\Z_i$ and we denote by $\Z'_i$ its union with the
remaining components. To use the Cayley-Bacharach ($\CB$) criterion
(Example \ref{CayleyB}) for a Gorenstein zero-dimensional scheme with
$k$ irreducible components, one needs to check the Hilbert function
for the $k$ different subschemes $\Z'_1,\ldots ,\Z'_k$: the $\CB$
criterion is that (with $\tau=\tau(\Z)$)
\begin{equation}\label{CBcriteq}
H(R/I_{\Z
'_k})_{\tau-1}=H(R/I_{\Z})_{\tau-1} \text{ for each } \Z'_i, i=1,\ldots ,k.
\end{equation}
 We have seen in Example
\ref{CayleyB} that when $\Z$ is local, not $\it{conic}$, but $\Z'$
is conic, then $\Z$ fails the $\CB$ criterion. Since being
arithmetically Gorenstein is a global property, there are no local
criteria for it. Nevertheless, the above equation \eqref{CBcriteq},
or even the inverse system can be used to check the $\CB$ criterion,
as we illustrate in the next example.
\end{remark}
\begin{example}\label{critagorex1} {\it Non AG scheme.} Suppose $R={{\sf k}}[x_1,x_2,x_3,z]$, and
$\Gamma={{\sf k}}_{DP}[X_1,X_2,X_3,Z]$, let $I_\Z={m_p}\cap I(2)$, where
$m_p=(x_1-z,x_2-z,x_3-z)$, the maximal ideal at $p=(1,1,1,1)$, and
$I(2)=(x_1,x_2^2,x_3^2)$, a complete intersection concentrated at
$p_0=(0:0:0:1)$. Then $\Z=\Z(1)\cup \Z(2)\subset \mathbb P^3$ with
$\Z(1)=p,\ \Z(2)=\Proj(R/I(2))$, and
\begin{equation*}
I=I_\Z=(x_1^2-x_1x_3, x_1x_2-x_1x_3, x_1x_3-x_1z, x_2^2-x_1z,
x_3^2-x_1z)
\end{equation*}
$H_\Z=(1,4,5,5,\ldots )$, with $\tau(\Z)=2$. A calculation shows $\Z
'(2)=\Proj(R/I'(2))$, where $ I'(2)=m_p\cap (I(2))$, has Hilbert
function $H'(2)=(1,4,4,\ldots )$, satisfying the criterion, but
$\Z'(1)=\Z(2)$, of Hilbert function $H'(1)=H(2)= (1,3,4,\ldots )$,
so $\Z$ is not arithmetically Gorenstein. \par This can be seen
using the inverse systems as follows: taking
$W=L_\Z=(I_\Z)^{-1},W(1)=L_{\Z(1)},W(2)=L_{\Z(2)},
L_p=X_1+X_2+X_3+Z$, we have
\begin{equation*} W_j=W(1)_j+W(2)_j=L_p^{[j]}+\langle
X_2Z^{[j-1]},X_3Z^{[j-1]},X_2X_3Z^{[j-2]},Z^{[j]}\rangle .
\end{equation*}
The inverse system $W'(2)$ to $I_{\Z' (2)}$ is obtained by removing
from $W_j$ the $\it{generator}$ $X_2X_3Z^{[j-2]}$ of $W(2)$, not
affecting $\dim_{\sfÊk} W'(2)_1=4$. The dual module $W'(1)$ to
$I'_{\Z'(1)}$ is obtained by removing from $W_j$ the
$\it{generator}$ $L_p^{[j]}$, of $W(1)$ which gives $\dim_{\sfÊk}
W'(1)_1=3$, not 4, as required by the Cayley-Bacharach criterion
\eqref{CBcriteq}.
\end{example}
\begin{remark}{\it Regularity degree.} When $\Z$ is concentrated at a
single point  we showed that the regularity and local socle degree
are related by $\sigma (\Z)\le \alpha (\Z)+1$ (see Proposition
\ref{AuxMainL} (\ref{is5})). This result cannot extend to arbitrary
zero-dimensional schemes. When the degree-$s$ scheme $\Z$ is smooth,
we have $\alpha (\Z)=0$, but $H_\Z$ can be any sequence such that
$\Delta H_\Z$ is an $O$-sequence of length-$s$, by Theorem
\ref{zero-dim} (\ref{zeroiv}). Since for a punctual scheme $\Z$,
$\sigma(\Z)=1+\tau(\Z)$, with $\tau(\Z)=\max\{ i\mid (\Delta
H_\Z)_i\not=0\}$, the maximum regularity degree is $s$, which occurs when $\Delta
H_\Z=(1,1,\ldots ,1)$. Even the degree $\tau$ component of the ideal
$I_\Z$ or of the inverse system $I_\Z^{-1}$ may be far from
determining the support of $\Z$. For example, if $s=2$,
the smooth scheme $\Z=(1:0:1)\cup (0:0:1)\subset \mathbb P^2$ has
inverse system $I^{-1}$ satisfying $(I^{-1})_i=(I_\Z^{-1})_i=\langle
Z^{[i]},(X+Z)^{[i]}\rangle$, $\Delta H_\Z=(1,1)$, so $\tau(\Z)=1$.
But the degree-$\tau$ component of the inverse system,
$(I^{-1})_1=\langle Z,X+Z\rangle$, only restricts the two points of
$\Z$ to lie on the line $y=0$.\par There has been much study of
regularity questions for zero-dimensional schemes. For example M.
Chardin and P. Philippon show that if there are forms $f_1,\ldots
,f_n$ of degrees $d_1,\ldots ,d_n$ in $\mathbb P^n$, such that
$f_1=\cdots =f_n=0$ contains $\Z$, and they form a local complete
intersection (LCI) at each support point of $\Z$, then the
regularity degree of $\Z$ is at most $d_1+\cdots +d_n-n$
\cite[Theorem A]{CharP}. LCI schemes $\Z$ occur naturally in both
singularity theory  (see \cite{Mi})and also in the study of certain
hyperplane arrangements (see \cite{Schk}). It could be of interest
to explore such zero-dimensional schemes from an inverse system
point of view. However, to detect CI or LCI from the inverse system
is not so easy, and it is rather simpler to detect if $\Z$ is Gorenstein.
\end{remark}
We give the following basic result bounding the regularity degree of $\Z$ in terms of the socle degrees of
the irreducible components of $\Z$, when the number of components is small.  We say that $k$ points in
$\mathbb P^n$ are in (linearly) general position if each subset of
$s$ points spans a $\mathbb P^{s-1}$, for $s\le n+1$.
\begin{proposition}\label{regbound} Let $\Z$ be a zero-dimensional scheme, supported at $p(1),\ldots ,p(k)\subset
\mathbb P^n$, and suppose the socle degrees of the irreducible components $\Z(1),\ldots ,\Z(k)$ are
$\alpha(1)\le
\cdots
\le
\alpha(k)$. If
$k\le n+2$ and the
$k$ points are in linearly general position, then the
regularity degree
$\sigma(\Z)$ satisfies
\begin{equation}\label{eregbound}
\sigma(\Z)\le \alpha(k)+\alpha(k-1)+2.
\end{equation}
\end{proposition}
\begin{proof} The inverse system $W\subset \mathcal  R$ for $m_{p(1)}^{\alpha(1)+1}\cap \cdots \cap
m_{p(k)}^{\alpha(k)+1}$ satisfies, by the the partial
derivative action of $R$ on $\mathcal R$ analogue \eqref{mpsannih},
$W_i=\< L(1)^{[i-\alpha(1)]},\ldots ,L(k)^{[i-\alpha(k)]}\>_i$. The
hypothesis that the points are in linearly general position, implies
that the ideal $(L(1)^{[i-\alpha(1)]},\ldots ,L(k)^{[i-\alpha(k)]})$
is a complete intersection when $k\le n+1$, and an almost complete
intersection when $k=n+2$. Using the Hilbert function of CI's, or a
result of R. Stanley (see \cite[Lemma C]{I5}) when $k=n+2$, we have
$\dim_{\sfÊk} W_i=\sum_u \dim_{\sfÊk} R_{\alpha(u)}$ if and only if $i<
i-\alpha(k)+i-\alpha(k-1)$, or $i\ge \alpha(k)+\alpha(k-1)+1$; for
such $i$ the sum $(L_{\Z(1)})_i+\cdots +(L_{\Z(k)})_i$ is direct,
since each $L_\Z(u)$ has $\tau(\Z(u))\le \alpha(u)$, and we have
$\tau(\Z)\le \alpha(k)+\alpha(k-1)+1$, implying \eqref{eregbound}.
\end{proof}

Analogous inequalities when $k\ge n+3$ can be shown in some special
cases, with the hypothesis that the points of support are
$\it{generic}$. However, the general problem of bounding
$\sigma(\Z)$ in terms of the $\alpha(u)$ is equivalent to the
interpolation problem, of determining the Hilbert function of higher
order vanishing ideals at the $k$ points. This problem is open in
general, unless $\alpha(u)\le 2$, or $k\le n+2$ (see
\cite{AH,Ch1,Ch2,I4}). When $k=6$ points on $\mathbb P^3$, there is
exceptional behavior; calculation for $\alpha = 3,4,\ldots $ shows
that if $\Z(u)=\Proj(R/m_{p(u)}^{\alpha +1}), u=1,\ldots ,6$, then
$\sigma(\Z)=2\alpha+3$.

\section{When can we recover the scheme $\Z$ from a dual form $F$?}\label{application} \par
When can a zero-dimensional scheme $\Z\subset \mathbb P^n$
be recovered from a general element $F$ in $\Gamma_j=(I_\Z)^{-1}_j$, the
degree-$j$ component of its inverse system? 
We begin with two examples, first of the scheme $\Proj (R/{m_p}^2)$,
which  cannot be so recovered, and, second, of a non-CM scheme
$\Z$--- having components of different dimension --- that can be
recovered. We then restate and prove our main result, giving a
sufficient condition when $\dim \Z = 0$ (Theorem \ref{main}). We
give several improvements in special cases,
and Corollary \ref{Thm1}, a consequence concerning subfamilies of
the parameter space $\mathbb P\mathbf{GOR}(T)$. In Section
\ref{link} we briefly describe linkage as viewed through the lens of
inverse systems, and in Section \ref{gad} we interpret our results
in terms of generalized additive decompositions of forms (Theorems
\ref{GADd3} and \ref{GAD3}).\par The following example is similar to
\cite[Example 5.10]{IK}.

\begin{example}\label{fatpt} {\it Non-recoverable scheme.}
On $\mathbb P^2$ with coordinate ring $ R={{\sf k}}[x,y,z]$,  consider the
non-Gorenstein ideal ${m_p}^2, p
= (0:0:1)$ which defines a degree-3 subscheme $\Z\subset \mathbb P^2$. Here
$I=I_\Z = (x^2,xy,y^2) \subset R={{\sf k}}[x,y,z]$, of Hilbert function $H(R/I_\Z)=(1,3,3,\ldots ,)$,
and
local Hilbert function
$H(R'/I')=(1,2)$. Thus $\tau (\Z)=1$, $\sigma (\Z)=2$, and  we have
\begin{equation}
(I_\Z)_i^{-1}\cap \Gamma _i = \{Z^{[i]},Z^{[i-1]}X,Z^{[i-1]}Y\}.
\end{equation}
Taking a general element $F=\alpha Z^{[j]}+\beta XZ^{[j-1]}+\kappa
YZ^{[j-1]}$, we find that $\Ann(F) $ contains  $\kappa x-\beta y$,
so we cannot recover the ideal $I_\Z$ from a single form $F$.
However, we can recover $I_\Z$ using two forms $F,G$, so from a
level algebra of type 2.
\end{example}

\begin{example}\label{linept} {\it Line with embedded point.} Let $R={{\sf k}}[x,y,z], \Gamma = {{\sf k}}_{DP}[X,Y,Z]$.
Consider
$F=XZ^{[3]}+Y^{[3]}Z
\in
\Gamma _4$. Then
$\Ann(F)= (x^2,xy,xz^2-y^3,z^4)$ defines an Artin algebra $R/\Ann (F)$ of Hilbert function
$T=(1,3,4,3,1)$. However,
 $\Ann(F)_{\le 2}=(x^2,xy)$, defines a scheme $\Z\subset \mathbb P^2$ consisting of a line with an
embedded point, whose Hilbert function satisfies $ H_\Z =
(1,3,4,5,6, \ldots )$. \par Taking instead $F_1=XYZ^{[2]}$, we find
$\Ann(F_1)=(x^2,y^2,z^3)$, also of Hilbert function $T$, and
$\Ann(F)_{\le 2}$ defines a degree-$4$ scheme $x^2=y^2=0$. More
generally let $\Z_1=\Proj(R/(g,h))$ be any complete intersection
scheme concentrated at $p_0\in \mathbb P^2$, of local Hilbert
function $H(R'/(g,h))=(1,2,1)$, and let $f_1\in \Gamma '$ be a
generator of the local inverse system $F=f_1Z^{[2]}$. Then it is
easy to see directly or by Corollary \ref{nascboij} that
$\Ann(F)_{\le 2}=(g,h)$, so determines $\Z_1$.
\end{example}
\begin{remarkWONum} {\sc Nonexistence of a morphism from $\Gor(T)$ to the Hilbert scheme of points.} Example
\ref{linept} shows that when $T=(1,3,4,3,1)$, it is not possible to
define a morphism from all of $\mathbb P\mathbf{GOR}(T)$ (the family
of Gorenstein ideals of Hilbert function $T$, see Definition
\ref{tight} below) to the punctual Hilbert scheme $\mathbf{Hilb}^4
(\mathbb P^2)$ parametrizing degree-$4$ zero-dimensional subschemes
of $\mathbb P^2$. The above example also answers negatively a
question asked in \cite[p. 142]{IK}, whether $\Z$ locally Gorenstein
might be a necessary condition for $I_\Z$ to occur as the ideal
generated by the lower degree generators of a Gorenstein Artin
quotient of $R/I_\Z$ --- as here $\Z$ is not even Cohen-Macaulay.
The question of which $\Z$ occur is open, even when $\Z$ is
restricted to be pure zero-dimensional. See \cite[Remark 5.73 and
Chapter 6]{IK} for further discussion.
\end{remarkWONum}

\subsection{Recovering $\Z$: main results.}
We now show our main result about recovering the scheme $\Z$ from a
general element $F\in L_\Z$. Recall that for a zero-dimensional
degree-$s$ scheme $\Z\subset \mathbb P^n$ we denote by $\tau (\Z) =
\sigma (\Z)-1=\min\{ i\mid (H_\Z)_i = s\}$. We denote by $\alpha
(\Z)$ the maximum local socle degree of a component of $\Z$ (see
Definition \ref{socled}). We let $\beta (\Z )=\tau(\Z)+\max\{ \tau
(\Z ),\alpha (\Z )\}$, and $L_\Z=(I_\Z)^{-1}$. It is evident that
for any $F\in (L_\Z)_j$, we have $I_\Z\subset \Ann(F)$. We assumed
throughout the paper that $\cha \  {{\sf k}} = 0$, or $\cha \  {{\sf k}} = p >j$,
where $j$ is the maximum degree of any form considered, here the
degree of $F$ (see Example \ref{char} for the necessity of this
assumption). We assumed that ${\sf k}$ is algebraically closed in order
for the support of $\Z$ to consist of ${\sf k}$-rational points. The
sequence $\Sym(H_\Z,j)$ is defined in equation \eqref{symhz}.

\begin{theorem}\label{main}{\sc Recovering the scheme $\Z$ from a Gorenstein Artin quotient.}  Let $\Z$ be
a locally Gorenstein zero-dimensional subscheme of $\mathbb P^n$
over an algebraically closed field ${{\sf k}}$, $\cha {\sf k}=0$ or $\cha {\sf k}>j$,
and let $L_\Z=(I_\Z)^{-1}$. Then we have
\begin{enumerate}[\rm(i)]
\item\label{main1} If $j\ge \beta(\Z)$, and $F$ is a general enough element of $(L_\Z)_j$, then
$H(R/\Ann(F))=\Sym(H_\Z,j)$. \par
\item\label{main2} If $j\ge \beta(\Z)$, and $F$ is a general enough element of
$(L_\Z)_j$, then for $i$ satisfying $\tau(\Z)\le i\le j-\alpha  (\Z)$ we
have $\Ann(F)_i= (I_{\Z})_i$. Equivalently, we have $R_{j-i} \circ F = (L_\Z)_{i}$. \par
\item\label{main3} If $j\ge \max\{ \beta (\Z), 2\tau (\Z)+1\}$, and $F\in (L_\Z)_j$ is general enough, then
$\Ann(F)$ determines $\Z$ uniquely. If $I_\Z$ is generated in degree $\tau (\Z)$, then $j\ge \max\{\beta
(\Z),2\tau(\Z)\}$ suffices.
\end{enumerate}
\end{theorem}
\proof Since $H(R/\Ann(F))$ is symmetric about $j/2$,
(\ref{main1}) follows immediately from (\ref{main2}). We now show
(\ref{main2}). Suppose first that $\Z$ has support the single point
$p_0 = (0: \dots :0:1)$. Let $f'$ of degree $\alpha=\alpha (\Z)$
generate the local inverse system at $p_0$ of $\Z$, let
$f=\Homog(f',Z,\alpha)$, and let $L=L_\Z$.  Lemma \ref{homoginvsys}
shows that $\forall i, L_i=R_\alpha\circ (f\cdot_{rp} Z^{[i]}).$
Taking $G=f\cdot_{rp} Z^{[j-\alpha]}$, we have $G\in L_j$, and for
$i'\ge \alpha$, we have by Proposition \ref{AuxMainL} (\ref{is3})
\begin{equation}\label{mainincludea}
R_{i'}\circ G = R_{i'-\alpha}\circ ( R_\alpha \circ G) =
R_{i'-\alpha}\circ L_{j-\alpha}= L_{j-i'}.
\end{equation}
Taking $F=G$, this proves (\ref{main2}) in this case.
 Next, if $\Z$ has support an arbitrary single point $p\in \mathbb P^n$, the proof of \eqref{main2} is
similar, using Theorem \ref{Homogp} and \eqref{Homogpeq}. \par Next,
suppose that $\Z$ has degree $s$, and support $p(1),...,p(k)$; thus
$ I_\Z=I(1)\cap \cdots \cap I(k)$ with $I(u)$ being the ideal of $R$
defining a scheme $\Z(u)$ having degree $s_u$, and concentrated at
the point $p(u)$, with $\sum s_u=s$. Suppose that  $\Z(u)\subset
\mathbb A^n $ is defined by $I'(u)\subset R'$ whose inverse system
has generator $f'(u)$ (since $I'(u)$ is Gorenstein) in the sense
$I'(u)^{-1}=(R'\circ f'(u))\cdot f_{q(u)}$ where if
$p(u)=(a_1(u):\ldots :a_n(u):1)$, we denote by $ q(u)=(a_1(u),\ldots
,a_n(u))$ the coordinates of $p(u)$ in $\mathbb A^n$. Let
$G(1),...,G(k)$ in $\Gamma _j$ be the homogenizations
$G(u)=\Homog(f'(u),L_{p(u)},j)$ (see Definition \ref{defhomcoordp}).
Suppose that $i\ge \tau (\Z)$. Denote by $\overline{h}$ the class of
$h \mod I_\Z$, and similarly for ideals, and let $V(u)= {I(1)}\cap
\cdots \cap \widehat{I(u)}\cap \cdots \cap { I(k)}$. We show
first \begin{claimWONum}  For each $u, 1\le u \le k$ we have
\begin{equation}\label{IntEq}
 (\overline{I(u)}_i)\oplus (\overline{I(1)}\cap \cdots \cap \widehat{I(u)}\cap
\cdots
\cap \overline{I(k)})_i = R_i/{(I_\Z)}_i.
\end{equation}
Furthermore, if $i\ge \tau(\Z)$, then codim $I(u)_i=s_u$ in $R_i$,
and  $\dim_{\sfÊk} \overline{V(u)}_i=s_u$, and also the codimension of
$V(u)_i$ in $R_i$ satisfies codim $V(u)_i=s-s_u$.
\end{claimWONum}
\begin{proof}[Proof of Claim] That the sum in \eqref{IntEq} is
direct is immediate, since the intersection of the two summends is
$(I_\Z)_i$. Since $\Z(u)$ has degree $s_u$, $\overline {I(u)_i}$ has
codimension no greater than $s_u$ in $R_i/(I_{\Z})_i$. Likewise the
vector space ${V(u)_i}$ has codimension in $R_i$ at most $\left(
\sum_{v\not=u} s_v\right)=s-s_u$ and likewise, $\overline{V(u)_i}$
has codimension at most $s-s_u$ in $R_i/(I_{\Z})_i$.   Since $i\ge
\tau (\Z)$ we have $\dim_{\sfÊk} (R_i/(I_{\Z})_i) = s$, thus we have
likewise $\dim_{\sfÊk}(R_i/(I(u))_i=s_u$, $\dim_{\sfÊk} R_i/{V(u)}_i=s-s_u$. And
this shows the equality of the Claim.
\end{proof} \par
 Now
let $F=\lambda_1\cdot G'(1)+\cdots +\lambda_k\cdot G'(k)$, where
$G'(u)\in W(u)_j , W(u)=I(u)^{-1}$ satisfies \eqref{mainincludea},
with $G,W$ there replaced by $G'(u), W(u)$, where $\lambda_u\in {{\sf k}}$
and each $\lambda_u \neq 0$. Consider $w=h\circ G'(u), h\in R_{i'}$.
By applying \eqref{IntEq}, we conclude that $h=h'+h'', h'\circ
G'(u)=0, h''\in V(u)$, thus $h\circ G'(u)=h''\circ G'(u)= h''\circ
F$. Thus, we have if $i'\ge \tau$, then $ R_{i'}\circ F\supset
R_{i'}\circ G'(u)$. Since evidently $R_{i'}\circ F \subset
(R_{i'}\circ G'(1)+ \cdots +R_{i'}\circ G'(k))$ there is for $i'\ge
\tau$ an equality of vector spaces
\begin{equation}\label{mainincludeb}
   R_{i'}\circ F=(R_{i'}\circ G'(1)+ \cdots +R_{i'}\circ G'(k)).
\end{equation}
If we take $i'\ge \max \{ \tau(\Z),\alpha (\Z)\}$, we may take $G'(u)=G(u)$ and apply \eqref{mainincludea}
to each term
$G'(u)$ of \eqref{mainincludeb}, and conclude, letting $W(u)=(I_{\Z(u)})^{-1}$, and taking $F$ as above,
$i=j-i'$
\begin{equation}\label{mainincludec}
  R_{i'}\circ F=W(1)_i+ \cdots +W(k)_i\subset W_i
\end{equation}
When $i\ge \tau=\tau(\Z)$, the sum in \eqref{mainincludec} is direct, and the inclusion on the right is
an equality.  That a particular $F\in W_j$
satisfies $\dim_{\sfÊk}R_{j-i}\circ F=s$, the maximum value possible (so there is equality on the right
of \eqref{mainincludec}) implies a fortiori that a general element $F\in W_j$ will have the same
property. This completes the proof of (\ref{main2}). If
$j\ge \max\{ \beta (\Z), 2\tau(\Z)+1\}$, we have that
$\Ann(F)_{\sigma (\Z)}= (I_\Z)_{\sigma (\Z)}$, so by \eqref{mainincludec}, and Theorem \ref{zero-dim}
(\ref{zeroii}),
$F$ determines
$\Z$, showing (\ref{main3}).
 This completes the proof of Theorem \ref{main}
\endproof.
\begin{remark}\label{nascboij}
A necessary and sufficient condition for $F=\lambda_1\cdot G(1)+\cdots
+\lambda_k\cdot G(k)$ in Theorem~\ref{main} to be general enough to satisfy
the conclusion, is for each $\lambda_1,\ldots , \lambda_k$ to be nonzero.
 \end{remark}
\begin{proof} The sufficiency was just shown, see especially \eqref{mainincludeb}. For the necessity,
note that if we form $F'$ by omitting the term $G_i$ from $F$ then
$I(F')_{\le \tau}= I(\Z')_{\le \tau}$ where $\
\Z'=\Z-\Z_i$.
\end{proof} \par \noindent
\begin{remarkWONum} We have found no counterexample to show that we could not replace $\beta$ in Theorem
\ref{main} by some smaller value, $\beta ' \ge 2\tau(\Z)$. What is
needed is to establish \eqref{mainincludea} for $i'\ge \alpha
'=\beta '-\tau$ --- for example \eqref{mainincludea} for $i'\ge \tau
(\Z)$ would allow us to replace $j\ge \beta(\Z)$  in Theorem
\ref{main} (\ref{main2}) by $j\ge 2\tau(\Z)$, and to simply omit
$j\ge \beta(\Z)$ from the statement of Theorem \ref{main}
(\ref{main3}) (See Corollary \ref{maincor} below).  A measure of the
specialness of our result, and a hope for improvement, is given by
the rather special form of $F$ in \eqref{mainincludeb}, far from a
generic element of $(L_\Z)_j$. The special case $\Z$ smooth of
Theorem~\ref{main} was shown by M. Boij \cite{Bj1}, and the cases
$\Z$ smooth or local $\it{conic}$ by the second author and V.~Kanev
\cite[Theorem 5.3E, Lemma 6.1]{IK}.
\end{remarkWONum}
Recall that an {\it SI-sequence} $H=(h_0,h_1,\ldots h_{\lfloor j/2\rfloor},\ldots ,h_j=1)$ with $h_i=h_{j-i}$ for $1\le i\le j$ is one satisfying $h_1-h_0,h_2-h_1,\ldots  ,h_t-h_{t-1}, t=\lfloor j/2\rfloor $ is an O-sequence (Theorem \ref{zero-dim}). We let $n=h_1$. We recover the following result of T. Harima, within our limitaton on $\cha \sf k=0$, or $\cha k>j$.
\begin{corollary} \cite{Har} Given an SI sequence $H$ there is a Gorenstein Artin algebra $A$ with Hilbert function $H(A)=H$.
\end{corollary}
\begin{proof} By P. Maroscia's result, Theorem \ref{zero-dim}(iv), there is a smooth zero-dimensional scheme $\Z\subset \mathbf P^n$ with  $h$-vector $\Delta (H_\Z)=\Delta H_{\le j/2}$. Here $\tau_\Z=t$, and $\alpha_\Z=0$ for a smooth scheme. By Theorem \ref{main} a generic $F\in (L_\Z)_j, j=2\tau $ satisfies $H_F=\Sym (H_\Z,j)$, which is $H$. This completes the proof.
\end{proof}\par
J. Migliore and U. Nagel show further, that there is  a reduced arithmetically Gorenstein punctual scheme $\Z' \subset \mathbb P^n$, with $h$-vector $\Delta H_{\Z'}=H$ \cite[Theorem 1.1]{MiN}.\par
 The following Corollary,
which determines
$\beta (\Z)$ in special cases, shows that we indeed recover the previous results of M. Boij and V. Kanev
and the second author when
$\Z$ is smooth or conic.
\begin{corollary}\label{mainspecial} $\rm{(i)}$ Let $\Z$ be supported at a single point $p$. Then $\tau(\Z)\le \alpha (\Z)$,
and $\beta(\Z)=\tau(\Z)+\alpha(\Z)$. Also, a general $F\in W_j$
determines $\Z$ if  $j\ge \tau(\Z)+\alpha(\Z)+1$, or if
both $j=\tau(\Z)+\alpha(\Z)$ and $I_\Z$ is generated in degrees less or
equal $\tau(\Z)$. \par \noindent $\rm{(ii)}$  Let $\Z$ also be conic.
Then $\tau(\Z)=\alpha (\Z)$ and $\beta(\Z)=2\tau(\Z)$. If instead
$\Z$ is smooth, then $\alpha (\Z) =0$, and also
$\beta(\Z)=2\tau(\Z)$.
\par \noindent $\rm{(iii)}$ In either the conic or smooth case, a general $F\in W_j$
determines $\Z$ if either $j\ge 2\tau(\Z)+1$, or if both $j\ge
2\tau(\Z)$ and $I_\Z$ is generated in degrees less or equal
$\tau(\Z)$.
\end{corollary}
We now state the Corollary mentioned in the Remark above. We show
that if the statements of Theorem \ref{main} are true for each
component $\Z(u)$ of $\Z$, but with $\beta$ replaced by
$\beta'=\tau(\Z)+\alpha '$, then they are true for $\Z$ with $\beta$
replaced by $\beta '$. We let $L(u)=L_{\Z(u)}$.
\begin{corollary}\label{maincor} Suppose that $\Z=\Z(1)\cup \cdots \cup \Z(k)$, and that there is an
integer $\alpha '\ge \tau(\Z)$ for which \eqref{mainincludea} holds
for each $\Z(u), u=1,\ldots ,k$, with $G,L$ there replaced by a
suitable choice of general enough $G'(u)\in L(u)_j$, with
$j=\tau(\Z)+\alpha '$ and $i'=\alpha '$ Then the conclusions of
Theorem \ref{main} hold with $\beta(\Z)$ replaced by $\beta
'=\tau(\Z)+\alpha '$.
\end{corollary}
\begin{proof} Taking $F=\sum \lambda (u) G'(u)$ after \eqref{mainincludeb}, the proof is essentially the
same (except we no longer take $G'(u)=G(u)$). Since $G'(u)$ is
assumed to satisfy \eqref{mainincludea} for $j=\tau(\Z)+\alpha'$, with $i'=\alpha '$ in place of
$i=\alpha$, we obtain the conclusion of Theorem \ref{main} (\ref{main2}) but with $j=\tau(\Z)+\alpha
'$. For larger $j'=j+c,c\ge 0$, we note that  \eqref{mainincludea} is still satisfied, replacing
$G'(u)$ by
$G'(u)\cdot_{rp}L_p^{[c]}\in
\Gamma_{j'}$, and $i'=\alpha '$ by $i'=\alpha '+c$. This implies Theorem \ref{main}
(\ref{main2}),(\ref{main3}), but with $\beta$ replaced by $\beta '$. This complete the proof of Corollary
\ref{maincor}.
\end{proof}
\begin{example}\label{mainex1} Let $R={{\sf k}}[X_1,X_2,Z]$ and  $f=\Homog(f',Z,4)$ from Example \ref{noncloserp}
where $f'=Y_2^{[4]}-Y_1Y_2^{[2]}+Y_1^{[2]}-Y_1Y_2-Y_2^{[21}$. Here
$\Z$ is concentrated at a single point $p_0=(0:0:1)\in
\mathbb{P}^2$, the Hilbert function $H(R/I_\Z)= (1,3,5,5,\ldots  )$,
so $\tau (\Z)=2 ,\alpha (\Z)=4$, and $\beta (\Z)=2+4=6$. The
Corollary  \ref{mainspecial} implies that for $j\ge 6$, a general
$F\in L_j, L=(I_\Z)^{-1}$ has $H_F=\Sym(H_\Z,j)$. However, a
calculation shows that this occurs for a general $F\in L_4$ (see
Example \ref{noncloserp} for $L_4$), hence for $j\ge 4$. In
particular, if $F$ is a general element of $L_5$, $H(R/\Ann
F)=(1,3,5,5,3,1)$, and $\Ann(F)_{\le 3}= \left( x_1x_2-x_2^2-x_1z,
x_2^3+x_1^2z+x_2^2z+x_1z^2, x_1^3\right) =(I_{\Z})_{\le 3}$. Thus
$F$ determines $\Z$ since $\sigma (\Z)=3$.
\end{example}
\begin{example}\label{mainex2} Consider the subscheme $\Z=\Z(1)\cup \Z(2)$ of $\mathbb P^2$, with $\Z(1)$ the
scheme of Example \ref{mainex1} concentrated at $p_1=(0:0:1)$ and
and $\Z(2)$ the degree-$4$ scheme concentrated at $p_2=(1:0:1)$, determined by
$f'=(Y_1^{[2]}+Y_2^{[2]})\cdot f_{p_2}$, of Example \ref{homogpex1}, where $\tau(\Z(2))=\alpha(\Z(2))=2$.
 The intersection
$I_\Z=I_{\Z(1)}\cap I_{\Z(2)}$ satisfies (calculated in {\sc{macaulay}})
\begin{multline*}
I_\Z=\left( x_1^3+x_1^2x_2-2x_1x_2^2-2x_1^2z-x_1x_2z+x_2^2z+x_1z^2,\right.\\
\left. x_1^2x_2^2-4x_1x_2^3/3-x_1^2x_2z-x_1x_2^2z+x_2^3z+x_1x_2z^2, x_1x_2^3,
x_2^4-x_1^2x_2z+x_2^3z+x_1x_2z^2\right),
\end{multline*}
of Hilbert function $H_\Z=(1,3,6,9,9,\ldots )$, $\tau(\Z)=3, \alpha(\Z)=4$.
By Corollary
\ref{maincor} and the
calculation of Example
\ref{mainex1} for $\Z(1)$, as well as Corollary \ref{mainspecial} applied to $\Z(2)$, we may
replace $\beta(\Z)=\tau(\Z)+\alpha(\Z)=3+4$ in Theorem \ref{main} for $\Z$ by $\beta '=3+3=6$. Thus,
a general $F\in (L_\Z)_6$ satisfies $H(R/\Ann(F))=\Sym(H_\Z,6)=(1,3,6,9,6,3,1)$.
\end{example}
We now derive some further consequence of our main theorem, along the lines
of Lemma 6.1 of \cite{IK}, shown there in the special case of $\Z$ conic
or smooth. We introduce first some
definitions from \cite{IK}. For $F\in \Gamma_j$ we let $H_F=H(R/\Ann(F))$.
\begin{definition}\label{tight}
 A zero-dimensional scheme $\Z$ is an \emph{annihilating scheme} for $F\in
\Gamma$ if $I_\Z\subset I_F=\Ann(F)$. An annihilating scheme is
\emph{tight} if also $\deg \Z=\max_i \{ (H_F)_i\}$. If
$T=(1,h_1,...,h_{j-1} ,1) $ is a sequence of integers symmetric
about $j/2$ we denote by $\PGOR(T)$ the locally closed subvariety of
$\mathbb P (\Gamma _j)$ parametrizing forms $F\in \Gamma_j$ --- up
to constant multiple
--- such that $H_F=T$. We denote by ${\bf{PGOR}}(T)$ (in boldface)
the corresponding scheme, whose scheme structure is defined by
determinantal ideals of certain catalecticant matrices,
corresponding to the conditions $(H_F)_u=T_u$ (see \cite{IK}).
\end{definition} 
 K. Ranestad and A. Bernardi \cite{BeRa} point out that the proof of  W. Buczy\'{n}ska and J. Buczy\'{n}ski \cite[Lemma 2.4]{BuB} shows that a tight annihilating scheme $\Z$ must be locally Gorenstein: otherwise there would be smaller degree locally Gorenstein scheme $\Z'$ constructed from $\Z$ apolar to $F$, contradicting $\Z$ ``tight''. This answers simply a question posed in \cite[Remark p. 142]{IK}.\par
The tangent space $\mathcal T_F$ to the affine cone over
${\bf{PGOR}}(T)$ at $F$ is isomorphic to $R_j/((\Ann F)^2)_j$
\cite[Theorem 3.9]{IK}. We denote by $\nu=\nu(\Z)$ the order
$\nu(\Z) =\min \{ i|(H_{\Z})_i\neq r_i\}$ of $ I_\Z$. We denote by
$\U_\Z\subset \mathbb P\mathbf{GOR}(T), T=\Sym(H_\Z,j)$ or more
precisely by $\U_\Z(j)$ the family of $F\in \Gamma_j$, up to
constant multiple, such that $F\in {(I_\Z)_j}^\perp$ and $H_F=T$.
Evidently $F\in \Gamma_j$ satisfies $F\in \U_\Z(j)$ if and only if
$\Ann(F)_i=(I_\Z)_i$ for $i\le j/2$ (since $I_\Z\subset \Ann(F)$
when $F\in {(I_\Z)_j}^\perp$). Below, we will usually omit writing ${\it up \,to \,constant\, mutiple}$ when this is
clear from the context, or unimportant. The zero-dimensional Hilbert
scheme $\mathbf{Hilb}^s(\mathbb P^n)$ parametrizes degree-$s$
subschemes of $\mathbb P^n$ (see \cite{IKl}).

\begin{corollary}\label{Thm1}
Let $\Z$ be a zero-dimensional degree $s$ locally Gorenstein scheme of $\mathbb P^n$
having regularity degree $\sigma (\Z)$, let $j\ge 2\tau (\Z)$, and let $F\in {(I_\Z)_j}^\perp$.
\begin{enumerate}[\rm(i)]
\item  If $j\ge \beta (\Z)$ (or if $\Z$ satisfies the hypothesis of Corollary \ref{maincor} and $j\ge
\beta '(\Z)$), there is an open dense family
$F\in {{( I_\Z)}_j}^\perp$ such that $F\in \U_\Z(j)$. For such $F$, we have
$(\Ann(F))_i=(I_\Z)_i$ for $i \le j-\tau (\Z)$, and $\Z$ is a tight annihilating scheme of
$F$.

\item If $j\ge 2\tau(\Z)$, and $F$ satisfies $H_F=\Sym(H_\Z,j)$, and
    if $Y \subset \mathbb P^n$ is any zero-dimensional subscheme satisfying
$\deg(Y) \le s$ and $ I_Y\subset
\mbox {\rm {Ann}}(f)$, then $\deg (Y)=s$ and $({\mathcal I}_Y)_i=({\mathcal
I}_\Z)_i$ for  $i\le j-\tau (\Z)$.

\item If $F$ satisfies $H_F=\Sym(H_\Z,j)$, and
if also either \par $\rm{(a)}$ $j\ge  2\tau (\Z)+1$, or \par
$\rm{(b)}$ $j\ge 2\tau (\Z)$,and $(({\mathcal I}_\Z)_{\le
\tau})=\mathcal I_\Z$,
\par then
$\Z$ is the unique tight annihilating scheme of $F$.

\item\label{Thm1iv}  If $F$ satisfies $H_F=\Sym(H_\Z,j)$, then
$\Ann(F)^2_{\ i}=(I_\Z^2)_{ i}$ for $i \le j-(\tau
-  \nu)$. \par If also $\tau \le \nu $ and $\Z$ is a
tight annihilating scheme of $F$, then the tangent space
$\mathcal T_F$ to the affine cone over ${\bf{PGOR}} (T), T=\Sym(H_\Z,j)$ at $F$ satisfies
 \[
     \dim_{\sfÊk}\mathcal
T_F=s+\dim_{\sfÊk}((\mathcal I_\Z/({\mathcal I}_\Z)^2)_j).\]

\item  If $Y\subset \mathbf{Hilb}^s(\mathbb P^n)$ is locally closed, and $\Z_y, y\in
Y$ is the corresponding family of degree $s$ zero-dimensional
subschemes of $P^n$, if $H(R/\mathcal I_{\Z_y})= H$ for all $y\in
Y$, if $\sigma=\tau +1$ is the generic regularity degree of $\Z_y,
y\in Y$ (attained for an open subset of $Y$), and if $j, \mathcal
I_\Z$ satisfy (iiia) or (iiib) above,  and $T=\Sym(H_\Z,j)$, then
there exists a subfamily $U_Y \subset \mathbb P\mathbf{GOR} (T)$
satisfying
\par
  $\rm{(c)}$  $F\in U_y \Leftrightarrow H_F = T$ and $\Z_y $ is a tight annihilating
scheme of $F$, \par
  $\rm{(d)}$  $\dim(U_Y) = \dim(Y)+s-1$.
\end{enumerate}
\end{corollary}
\begin{proof} Here the main assertion (i) follows directly from Theorem \ref{main} and the proof of
\cite[Lemma 6.1]{IK}; we need in (i) the hypothesis $j\ge \beta(\Z)$
in order to use Theorem \ref{main}. For any $j\ge 2\tau(\Z)$, the
assumption $H_F=\Sym(H_\Z,j)$, and that $I_\Z\subset \Ann (F)$
entail most of (ii)-(iv).
\end{proof}
\begin{example}\label{mainex3} Consider the subscheme $\Z$ of Example \ref{mainex2}, for which Corollary
\ref{maincor} applies for $\beta '(\Z)=6$, and choose a general
$F\in (I_\Z)^{-1}_6$; then $T=H_F=\Sym(H_\Z,6)= (1,3,6,9,6,3,1)$. A
calculation shows that $\dim_{\sfÊk} R/((I_\Z)^2)_6=27$. Since $\nu=\tau$
for $\Z$, Theorem \ref{Thm1} (\ref{Thm1iv}) implies that $\dim_{\sfÊk}
\mathcal T_F=27$; this is easy to check directly since $\Ann
(F)_{\le 3}=\langle h_3\rangle $, so $\Ann (F)^2_6=\langle
h_3^2\rangle$ of codimension 1 in $R_6$. Since $r=3$, $\mathbb
P\mathbf{GOR}(T)$ is smooth; this here corresponds to the
smoothability of degree-9 schemes in $\Z$. The dimension of
$\PGOR(T)$ is 27, since $\dim(\mathbf{Hilb}^9(\mathbb P^2))=18$, and
the dimension of the fiber of $\mathbb P\mathbf{GOR}(T)$ over
$\mathbf{Hilb}^9(\mathbb P^2)$ is $9$.
\par Strikingly, if $j=7$, so $T'=(1,3,6,9,9,6,3,1)$, the analogous
dimension  is $\dim_{\sfÊk} \mathcal T_F=30$ (since
$(\Ann(F)^2)_7=(I_\Z)^2_7=h_3 \cdot (I_\Z)_4$, of dimension 6); when
$j\ge 8$ the dimension is again 27, as can be checked by caculating
$H(R/(I_\Z)^2)$.
\end{example}
\smallskip
\subsection{Dualizing module as ideal, and linkage.}\label{link}
We first recall a result of M. Boij, giving conditions under which the dualizing
module of $\Z$ is an ideal of $R/I_\Z$. A consequence of his
criterion and Theorem \ref{main} is that the dualizing module can
always be so realized when $\Z$ has dimension zero, and is locally
Gorenstein (Corollary \ref{boijus1}).  We then give an example to
illustrate how the inverse systems behave in linkage. \par M. Boij's
theorem pertains to $d-$dimensional Cohen Macaulay rings $B=R/I$,
and $d-1$ dimensional Gorenstein quotients. Let $\kappa (B)$ denote
the degree of the polynomial $(1-z)^d  \Hilb_X(z)$: here
$\Hilb_X(z)$ is the Hilbert series $\sum H_\Z(i) z^i$, so
$\kappa(B)$ is the highest socle degree of a minimal reduction of
$B$.
\begin{theorem}\label{boij1}{\rm (M. Boij [Bo2])}
Let $B=R/I$ be a Cohen-Macaulay algebra of dimension $d$, and let $J\subset B$ be an ideal of initial
degree at least $\kappa (B)+2$ such that $B/J$ is Gorenstein of dimension $d-1$. \par
Then there is an isomorphism $J\rightarrow \Ext_R^{r-d}(B,R) =\omega_B$, which is homogeneous of
degree $-\kappa (B/J)-r+d-1$.
\end{theorem}
 We consider the special case
$d=1$, and
$I=I_\Z$, the homogeneous defining ideal of a zero-dimensional scheme $\Z$. Then Boij's theorem
becomes,
\begin{corollary}\label{boij2} Let $\Z$ be a zero-dimensional subscheme of   $\mathbb P^n$, and let $J$ be an ideal
of
$B=R/I_\Z=\mathcal O_\Z$ having initial degree at least $\tau (\Z)+2$, such that $B/J$ is
Gorenstein of dimension zero and socle degree
$j$. Then there is an isomorphism $J\rightarrow \Ext_R^{r-1}(B,R)=\omega _{B}$,
which is homogeneous of degree $(-j-r)$.
\end{corollary}
Our Main Theorem \ref{main} and M. Boij's theorem imply
\begin{corollary}\label{boijus1} Let $\Z$ be a locally Gorenstein zero-dimensional scheme of $\mathbb P^n$. Then
there
\linebreak are ideals $J$ of $\mathcal O_\Z=R/I_\Z$ satisfying the conclusions of Corollary \ref{boij2},
with
$\mathcal O_\Z/J$ of socle degree $j$,
provided
$j\ge
\max\{\beta (\Z),2\tau(\Z)+1\}$. Any such ideal has the form $J=\Ann(F)/I_\Z, F\in (I_\Z^{-1})_j\subset
\Gamma_j$. Also, if
$j\ge 2\tau(\Z)+1$ and
$F$ is any element of
$(I_\Z^{-1})_j$ such that
$H_F=H(R/\Ann(F))=\Sym(H_\Z,j)$, then $J=\Ann(F)/I_\Z$ is isomorphic to the dualizing module of $\Z$.
\end{corollary}
\begin{proof} The second statement follows from M. Boij's theorem for $B=\mathcal O_\Z$, and Macaulay's result
connecting the socle of $R/J'$ for an Artinian quotient, and
generators of the inverse system of $J'$ (see Corollary
\ref{macdual}); $J'$ is Gorenstein if and only if ${J'}^{-1}$ is
principal. The third statement follows from Corollary \ref{boij2}
and the definition of $\Sym(H_\Z,j)$ (see Equation \eqref{symhz});
the restriction $j\ge 2\tau(\Z)+1$ and $H_F=\Sym(H_\Z,j)$ implies
that the order of $\Ann(F)/I_\Z$ is at least $\tau(\Z)+2$,
satisfying the hypotheses of M. Boij's theorem, and $\mathcal
O_\Z/J\cong R/\Ann(F)$, so is Gorenstein. By Theorem \ref{main} and
Corollary \ref{Thm1}, such $F$ exist with $H_F=\Sym (H_\Z,j)$ if
$j\ge \beta(\Z)$.
\end{proof}\par \noindent
M. Boij showed that when $\Z$ is smooth, then the conclusions of
Corollary \ref{boijus1} hold also for $j\ge 2\tau(\Z)-1$. His work
is related to that of M. Kreuzer in \cite{Kr1,Kr}. Corollary
\ref{boijus1} can be used as a test of whether a Gorenstein scheme
is arithmetically Gorenstein, since $\Z$ is aG if and only if the
dualizing module is principal.
\begin{example}\label{critagorex2} Let $\Z=\Z(1)\cup \Z(2)\subset \mathbb P^3$ be the scheme of Example
\ref{critagorex1}, where $\Z(1)=p, p=(1:1:1:1)$ is a smooth point,
and $\Z(2)=\Proj(R/I(2))$ where $ I(2)=(x_1,x_2^2,x_3^2)$, is a CI
at $p_0= (0:0:0:1)$. We found there that $\Z$ was not arithmetically
Gorenstein, although $\Delta H_\Z = (1,3,1)$ is the $h$-vector of a
Gorenstein ideal (it is a ${\it Gorenstein \,sequence}$). Since
$\alpha (\Z)=\tau(\Z)=2$, we have $\beta
(\Z)=\tau(\Z)+\alpha(\Z)=4$. By Corollary \ref{boijus1}, it suffices
to take a general element $F\in (L_\Z)_5$, to see the dualizing
module as the ideal $J=\Ann(F)/I_\Z$. We have, taking
$L_p=X_1+X_2+X_3+Z$
\begin{equation*} (L_\Z)_5=\langle X_2Z^{[4]},X_3Z^{[4]},X_2X_3Z^{[3]},Z^{[5},L_p^{[5]}\rangle .
\end{equation*}
A calculation shows that with
$F=X_2Z^{[4]}+X_3Z^{[4]}+X_2X_3Z^{[3]}+Z^{[5]}+L_p^{[5]}$, we have
$H_F=\Sym(H_\Z,5) = (1,4,5,5,4,1)$, and that
$\Ann(F)=(I_\Z,x_2x_3z^2-x_2z^3-x_3z^3+z^4, x_1z^4-z^5/2)$. The
dualizing module $\Ann(F)/I_\Z$ is not principal, confirming that
$\Z$ is not arithmetically Gorenstein.
\end{example}
We now give an example showing how inverse systems behave under linkage: here $\Z=\Z(1)\cup\Z(2)$ is AG
(even CI).
\begin{example}{\it Inverse systems of linked local CI's.}\label{exlink}
We consider inverse systems of three ideals in $R={{\sf k}}[x_1,x_2,z]$ defining
punctual subschemes $\Z=\Z(1)\cup \Z(2)$ of
 $\mathbb P^2$.  The ideal $I(1)=(x_1-z,x_2)=M(p_1)$ defines the simple point $\Z(1)=p_1=(1,0,1)$. The
ideal
$I(2)$, concentrated at $p=(0:0:1)$ defines a degree $5$ scheme $\Z(2)$, that of Example \ref{regdeg},
(there termed
$\Z$), which is a local complete intersection:
\begin{equation*}
 I(2)=I_{\Z (2)}=(x_1x_2,x_1^2z-x_2^3,x_1^3),
\end{equation*}
of Hilbert function $H_{\Z(2)}=(1,3,5,5,\ldots )$. Their
intersection is the ideal $I=(x_1-z,x_2)\cap I_{\Z(2)} =
(x_1x_2,x_1^3+x_2^3-x_1^2z)$, a complete intersection defining the
degree 6 punctual scheme $\Z$, of Hilbert function
$H_\Z=(1,3,5,6,6,\ldots )$. Thus $I(1)$ and $I(2)$ determine the two
irreducible components of $\Z$, which are linked through $\Z$.
Letting $W=I^{-1},W(1)=I(1)^{-1}$, and $W(2)=I(2)^{-1}$ denote the
corresponding inverse systems, we have $W=W(1)+W(2)$, where the sum
must be direct in degrees at least $\tau (\Z)=3$ by Theorem
\ref{decompIS} (\ref{directs}). The inverse system $W(1)$ satisfies
$W(1)_i=\langle (X_1-Z)^{[i]}\rangle$, while $W(2)$ satisfies, from
Example \ref{regdeg}
\begin{equation*}
W(2)_i=\langle
X_1^{[2]}Z^{[i-2]}+X_2^{[3]}Z^{[i-3]},X_2^2Z^{[i-2]},X_2Z^{[i-1]},X_1Z^{[i-1]},Z^{[i]}\rangle
\end{equation*}
By the Decomposition Theorem \ref{decompIS}(i), we have
\begin{align}
  W_i&=W(1)_i+W(2)_i\notag \\
 &=\langle (X_1-Z)^{[i]},
X_1^{[2]}Z^{[i-2]}+X_2^{[3]}Z^{[i-3]},X_2^2Z^{[i-2]},X_2Z^{[i-1]},X_1Z^{[i-1]},Z^{[i]}\rangle
.\notag
\end{align}
Note that the above sum is direct in degrees at least three, but not
direct in degrees less or equal two, as is evident by regarding
$H_{\Z(1)}=(1,1,\ldots )$ and $H_{\Z(2)}, H_{\Z}$. By
the Decomposition Theorem \ref{decompIS}(iii) we have
\[
W(2)_i=W_i\cap \langle {{\sf k}}[X_1,X_2]_{\le 3} \cdot_{rp} {{\sf k}}[Z]\rangle _i,
\]
the intersection of $W$ and the inverse system of ${m_p}^4$ (here $4=\alpha (\Z(2))+1$), whenever the
dimension of the right side is $5$, which occurs for $i\ge 4$. \par
That $\Z$ is AG can be seen from the inverse system, following Lemma \ref{socle}, by showing that
 $L_\Z\cap \Gamma_z=W\cap {{\sf k}}_{DP}[X_1,X_2]$ is a principal $R'={{\sf k}}[x_1,x_2]$-module: in fact, $G=X_1^{[3]}-X_2^{[3]}$
generates this intersection. \par Finally, from the properties of
linkage, $I(1)/I$ has dualizing module isomorphic to $R/(I(2))$, and
conversely $I(2)/I$ has dualizing module $R/I(1)$. In particular the
number of generators of $I(1)/I$ (here two) is the same as the
dimension of $\Soc(R/I(2))$ and the number of generators of $I(2)/I$
(here one) is $\dim_{\sfÊk} \Soc(R/I(1))$.  In addition, since $R/I$ is
locally Gorenstein, similar properties hold for the localizations at
$p,p_1$. Here at $p_1$, $(R/I(1))_{p_1}\cong R'/m_{p_1}$, has
one-dimensional socle, and $m_{p(1)}\cong I_{p_1}$, the
localization, so there are zero generators of the quotient. Also,
$(R'/I(2))_{p_1}=0$, so has zero socle, and $I(2)_{p_1}=R'_{p_1}$
has one generator.
\end{example} \noindent

\subsection{Generalized additive decompositions.}\label{gad}
We recall the GAD given in \eqref{gad1} for a degree-$j$ form of $\Gamma = {{\sf k}}[X,Y]$,
namely
\begin{equation*}
F=\sum_i B_iL_i^{[j+1-s_i]}, \deg B_i=s_i-1, \deg L_i=1, s=\sum s_i.
\end{equation*}
Each term $B_iL_i^{[j+1-s_i]}$ corresponds to a single support point
$p_i: l_i=0$ of $\mathbb P^1$, occuring with  multiplicity $s_i$.
Our aim is to model this kind of decomposition in $r\ge 3$
variables. The following definition is more general than that of
\cite[Def. 4A]{I4}, but is related to the concept of annihilating
scheme introduced there \cite[Def. 4D]{I4} (see Definition
\ref{tight} above).
\begin{definition}\label{GAD2} For $F\in \Gamma_j$, we say that $F=F_1+\cdots +F_k$ is a
\emph{generalized additive decomposition} (GAD) of $F$, having
(total) length $s=\sum s_i$, of partition $\pi=(s_1,\ldots ,s_k)$,
with $k$ parts, associated to the scheme $\Z$, if $\Z$ is a
degree-$s$ punctual scheme $\Z$ whose decomposition into irreducible
schemes is $\Z=\cup \Z_i$, where $\deg \Z_i=s_i$, and each $F_i\in
{I_{\Z_i}}^{\perp}$ for $i=1,\ldots ,k$.  We say that a GAD of $F$
is $\it{tight}$ if $\Z$ is a tight annihilating scheme of $F$:
namely, if $s=\deg \Z=\max_i \{ (H_F)_i\}$ (Definition \ref{tight}).
We say that a GAD is \emph{unique} if the $k$ summends $F_1,\ldots ,
F_k$ are unique.
\end{definition}\noindent
The form of each term $F_i$, corresponding to $\Z_i$, can be read
from Theorem \ref{Homogp} or Proposition~\ref{satinvsysp}; $F_i$ is
an element of the degree-$j$ homogenization of the local inverse
system of $\Z_i$.
\begin{lemma} If $F$ has a length-$s$ GAD, then $\forall i \ge 0$, we have $(H_F)_i\le s$.
\begin{proof} We have $\Ann(F)\supset \mathcal I_\Z$, hence $(H_F)_i\le (H_\Z)_i$, but $(H_\Z)_i$
is bounded above by $\deg \,\Z$.
\end{proof}
\end{lemma}
Which forms $F$ have a length-$s$ GAD? When is the GAD for $F$ unique? Recall that
we denote by
$\sigma (\Z)$ the
regularity degree of
$\Z$ (see Theorem~\ref{zero-dim}), and by $\tau (\Z)=\sigma(\Z)-1$. Evidently we have
\begin{lemma}\label{GADun1} If $F$ is annihilated by a zero-dimensional scheme $\Z,\Z=\Z_1\cup \cdots \cup
\Z_k$ as in Definition \ref{tight},  then $F$ has a
 GAD of length
$\le s$ associated to
$\Z$. If also
$\deg F\ge
\tau(\Z)$, then the GAD has length $s$, is of partition $(s_1,\ldots ,s_k), s_i=\deg \Z_i$, and
this GAD is the unique GAD of
$F$ that is associated to $\Z$.
\end{lemma}
\begin{proof} For $j\ge \tau(\Z)$ we have
$(H_\Z)_j=s$, hence $(I_\Z)_j^{\perp}=(I_{\Z_1})_j^{\perp}\bigoplus
\cdots \bigoplus (I_{\Z_k})_j^{\perp}$, and the GAD is unique.
\end{proof}\par
The following result is an immediate consequence of \cite[Theorem
5.31]{IK}, and Definition \ref{GAD2}. It does not extend simply to
$r>3$ (see \cite[Theorem 6.42]{Bj4}, \cite{ChoI2}, and the
discussion in \cite[\S 6.4]{IK}).
\begin{theorem}{\sc Uniqueness of GAD when $r=3$.}\label{GADd3} If $r=3$ and $H_F\supset (s,s,s)$
then
$F$ has a unique tight GAD of length $s$, up to permutation and change of scale, and no GAD's of
smaller length than $s$.
\end{theorem}
\begin{proof}By Theorem 5.31 of \cite{IK}, $F$ has a unique tight annihilating scheme $\Z$; this
determines a unique GAD by Lemma \ref{GADun1}, since $j=\deg F>\sigma(\Z)$ (as here we have $j\ge
2\sigma(\Z)$).
\end{proof} \par
Recall from Definition \ref{socled} that $\alpha (\Z)$ is the highest socle degree of a
component of $\Z$. Finally we have,
\begin{theorem}\label{GAD3} Suppose that $\Z$ is a Gorenstein zero-dimensional subscheme of $\mathbb P^n$,
and that
$j\ge
\max\{ \tau(\Z)+\alpha(\Z), 2\tau (\Z)+1\}$. If $F$ is a general enough element of
${(I_\Z)_j}^\perp$, then $F$ has a unique GAD of length $s$ associated to $\Z$, and no
GAD's of length less than $s$.
\end{theorem}
\begin{proof} Let $t=\lfloor j/2\rfloor$. By Theorem \ref{main} the hypotheses on $F$ and $\Z$
imply that $H_F=\Sym (H_\Z,j)$. Furthermore, the assumption on $j$
implies that $(H_F)_t=(H_F)_{t+1}=s$, so $H_F\supset (s,s)$. It
follows that any scheme $\Z'$ of degree at most $s$, such that
$I_{\Z'}\subset I_F$, satisfies $(H_{\Z'})_t=(H_{\Z'})_{t+1}=s$,
hence $(I_{\Z'})_t=(I_F)_{t}$ and $(I_{\Z'})_{t+1}=(I_F)_{t+1}$. By
Theorem \ref{zero-dim} this equality implies that $\Z'$ is regular
in degree $t+1$ (so $\sigma (\Z')\le t+1$), and is determined by
$F$, so we must have $\Z=\Z'$. Uniqueness of the GAD now follows
from Lemma \ref{GADun1}.
\end{proof}
\begin{example} Let $R={{\sf k}}[x,y,z]$ and denote by $\Upsilon$ the degree 3 scheme $\Upsilon=\Proj
(R/(x,y^3))$ concentrated at the origin $p_0=(0:0:1)$ of $\mathbb
P^2$; and denote by $\Z=\Z_1\cup \cdots \cup \Z_k$ the union of $k$
distinct subschemes, where $\Z_i$ denotes a translation of
$\Upsilon$ to a point $p_i=(a_{i0}:a_{i1}:1)\in \mathbb P^2$ (by
$T_{p_i}$ as in Lemma \ref{transforma}). By Theorem \ref{Homogp}, we
have that $\Z_i=\Proj(R/(x-a_{i0}z,(y-a_{i1}z)^3))$, and since the
inverse system $L(\Upsilon)\subset \Gamma ={{\sf k}}_{DP}[X,Y,Z]$ satisfies
$L(\Upsilon)_u=\mathcal (I_\Upsilon)^{\perp}=R_u\circ ( Y^{[2]}\cdot
Z^u)=\langle Y^{[2]}Z^{[u-2]},YZ^{[u-1]},Z^{[u]}\rangle$, we have
\begin{align*}
L(\Z_i)_u&=R\circ (Y^{[2]}\cdot (a_{i0}X+a_{i1}Y+Z)^{[u]})\\
&=\langle Y^{[2]}\cdot
(a_{i0}X+a_{i1}Y+Z)^{[u-2]},Y\cdot (a_{i0}X+a_{i1}Y+Z)^{[u-1]},(a_{i0}X+a_{i1}Y+Z)^{[u]}\rangle .
\end{align*}
Taking $k=2$, letting $p_1=p_0, p_2=(1:1:1)$, we have
\begin{equation*}
\mathcal I_\Z=(x,y^3)\cap (x-z,(y-z)^3)=(x^2-xz,3xy^2-y^3-3xyz+xz^2),
\end{equation*}
and $\Delta H_\Z=(1,2,2,1), H_\Z=(1,3,5,6,6,\ldots)$. By Lemma \ref{GADun1}, a  form $F\in
(I_\Z)_j^{\perp}$ for
$j\ge 3$ has a unique decomposition associated to $\Z$, into $k=2$ parts, each of length 3,
\begin{align*}
F=F_1+F_2\,\text { where } &F_1\in \langle
Y^{[2]}Z^{[j-2]},YZ^{[j-1]},Z^{[j]}\rangle,\\& F_2\in \langle
Y^{[2]}\cdot(X+Y+Z)^{[j-2]},Y(X+Y+Z)^{[j-1]},(X+Y+Z)^{[j]}\rangle .
\end{align*}
When $j=3$, the form $F=3Y^{[3]}+XY^{[2]}\in L_\Z$, as it is evidently annihilated by
$\mathcal I_\Z$ acting as contraction. Thus F has a GAD with two summands, each of length 3,
\begin{equation}F=Y^{[2]}\cdot (X+Y+Z)-Y^{[2]}Z.
\end{equation}\par
By Theorem \ref{GAD3}, since when $k=2$, $\Z$ is Gorenstein with
$\tau (\Z)=3$ and $\alpha (\Z)=2$, we have for $j\ge 7$ that a
general $F\in (L_\Z)_j$ has a tight annihilating scheme $\Z$, so a
unique GAD of length 6. However, if $j\ge 6$, and $F$ includes the terms
$Y^{[2]}(X+Y+Z)^{[j-2]}$ and $Y^{[2]}Z^{[j-2]}$, then it is
easily seen that $F$ determines $\Z$, as $I_\Z$ is generated in
degree 3, and $H_F=\Sym(H_\Z,j)$ by calculation.
\par Taking $k=3$, using translates of
$\Upsilon$ at the three points $p_1,p_2$, and $p_3=(2,3,1)$ we find
$\Delta H_\Z=(1,2,3,3)$; taking $k=4$ and points $p_1,p_2,p_3$, and
$p_4=(7,11,1)$ we find $\Delta H_\Z=(1,2,3,4,2)$. However, if we
take instead $p_4'=(2,5,1)$ we find $\Delta H_{\Z'}=(1,2,3,3,2,1)$.
We might ask, for a generic choice of $k$ points $\{p_i\}$, do we
obtain a degree $3k$ scheme $\Z$ in ${\it general\, position}$ -
having the same Hilbert function as $3k$ generic smooth points? This
is not the case for $k=2$ here, but is true for $k=3,4$, and presumably
for higher $k$.  \par Also, we may ask, what is the dimension of the
family $\mathcal F(\Upsilon,k,\mathbf P^2)$ of all degree $3k$
punctual subschemes of $\mathbb P^2$ having the form $\Z=\Z_1\cup
\dots \cup \Z_k$, with $\Z_i$ a translate of $\Upsilon$? In this
direction, the tangent space to such families have been studied
classically for power sum representations $F=\sum L_{p_i}^{[j]}$
(see \cite{T2,Br,AH,I4}, \cite[\S 2.1,2.2]{IK}, and for GAD's see
\cite{Eh,Tes}, also \cite{Ch2}).
\end{example}
\begin{remark}\label{rem324}
 In a sequel paper
\cite{ChoI2} we determine the global Hilbert functions $H_\Z$ for
compressed Gorenstein subschemes $\Z\subset \mathbb P^n$. Let $r=n+1$ and 
denote by $H_s(r)$ the sequence $H_s(r)_i=\min\{ \dim_{\sfÊk} R_i,s\}$. Then
$H_s(r)$ is the global Hilbert function of a generic degree-$s$
smooth scheme. We will show in the sequel that if $\Z$ is a general enough compressed
local Gorenstein scheme of degree $s$, then $H_\Z=H_s(r)$. Using
Theorem~\ref{main}, we will exhibit families $\mathbb
P\mathbf{GOR}(T)$ of graded Gorenstein Artin algebras of embedding
dimension $r$ and certain Hilbert functions
$T=H(s,j,r)=\Sym(H_s(r),j), r\ge 5, s$ large enough given $r$, that
contain several irreducible components. Each component is fibred
over a family of Gorenstein zero-dimensional schemes, with fibre an
open in a projective space $\mathbb P^{s-1}$.
 One component is fibred over general enough smooth schemes
$\Z\subset \mathbb P^n, n=r-1$ of degree $s$. The other component is
fibred over a family of compressed Gorenstein subschemes. Here
$T=H(s,j,r)=\Sym(H_s(r),j)$ and $ H(s,j,r)_i=\min\{ r_i,r_{j-i},s\}$
is the Hilbert function of GA algebras $R/\Ann (F),
F=L_1^{[s]}+\cdots +L_s^{[j]}\in \Gamma$, determined by a dual
generator $F$ that is a sum of $s$ general enough (divided) powers
of linear forms. Some of these results were reported in \cite[\S
6.4]{IK}. \par
\end{remark}
{\bf Acknowledgements}\par \smallskip
 We appreciated suggestions for the earlier version of Tony Geramita and the late Ruth
Michler, of V.~Kanev concerning the GAD viewpoint, discussions with Yuri Berest
on projective closure and with Peter Schenzel concerning the connections with dualizing modules, and
as
 well the comments of Bae-Eun Jung. 
The second author
gratefully acknowledges the influence of many discussions begun some
years ago with Jacques Emsalem concerning
inverse systems and ``points \'{e}pais'' \cite{Em}.\par

\bibliographystyle{amsalpha}

\small
{\bf e-mail addresses:} youngcho{\scriptsize @}\small math.snu.ac.kr and a.iarrobino\scriptsize@\small neu.edu
\end{document}